\documentclass[11pt]{article}

\usepackage{amsmath, amsthm,amsfonts,amssymb,mathtools,mathdots,scalerel}

\usepackage{enumitem, tocloft}
\usepackage[hidelinks]{hyperref}
\usepackage{tikz, stmaryrd, cmll}
\usepackage[british]{babel}
\usepackage[a4paper]{geometry}

\newtheorem{thm}{Theorem}[section]
\newtheorem{prop}[thm]{Proposition}
\newtheorem{cor}[thm]{Corollary}
\newtheorem{lem}[thm]{Lemma}
\newtheorem{sublem}[thm]{Sub-Lemma}
\newtheorem{conj}[thm]{Conjecture}
\theoremstyle{definition}
\newtheorem{dfn}[thm]{Definition}
\newtheorem{cons}[thm]{Construction}
\theoremstyle{remark}
\newtheorem{remark}[thm]{Remark}
\newtheorem{exm}[thm]{Example}

\usetikzlibrary{arrows, decorations.markings, shapes.geometric, decorations.pathmorphing, decorations.pathreplacing, intersections, patterns,calc, backgrounds}

\tikzset{
	0c/.style={circle, draw, fill, inner sep=1.5pt},
	1c/.style={->, thick, shorten <=2pt, shorten >=2pt},
	1clong/.style={->, thick},
	edge/.style={thick, shorten <=2pt, shorten >=2pt},
	1cdash/.style={->, densely dashed, thick, shorten <=2pt, shorten >=2pt},
	1cdot/.style={->, dotted, thick, shorten <=2pt, shorten >=2pt},
	1cinc/.style={right hook->, thick, shorten <=2pt, shorten >=2pt},
	1cincl/.style={left hook->, thick, shorten <=2pt, shorten >=2pt},
	2c/.style={double, thick, shorten <=6pt, shorten >=8pt, decoration={markings,mark=at position -6pt with {\arrow[scale=1.75]{>}}}, preaction={decorate}},
	3c1/.style={thick, double, double distance=3pt, shorten <=9pt, shorten >=11pt},
    	3c2/.style={thick, shorten <=9pt, shorten >=10pt},
	3c3/.style={shorten <=9pt, shorten >=10pt, decoration={markings,mark=at position -8pt with {\arrow[scale=3]{>}}},preaction={decorate}},
	follow/.style={->, >=stealth, ultra thick, shorten <=3pt, shorten >=3pt, color=gray!70},
	every node/.style={scale=.8},
}

\newcommand\cp[1]{\,{\scriptstyle\#}_{#1}\,}
\newcommand\comp[1]{\triangleright_{#1}}
\newcommand\bord[2]{\partial_{#1}^{#2}}
\newcommand\idd[1]{\varepsilon_{#1}}
\newcommand\opp[1]{{#1}^\mathrm{op}}
\newcommand\coo[1]{{#1}^\mathrm{co}}
\newcommand\oppall[1]{{#1}^\circ}
\newcommand\oppn[2]{D_{#1}({#2})}

\newcommand\invrs[1]{#1^{-1}}
\newcommand\skel[2]{\sigma_{\leq #1}#2}
\newcommand\nbd{\nobreakdash-\hspace{0pt}}
\newcommand\homset[1]{\mathrm{Hom}_{#1}}

\newcommand\sbord[2]{\Delta_{#1}^{#2}}
\newcommand\dmn[1]{\mathrm{dim}(#1)}
\newcommand\frdmn[1]{\mathrm{frdim}(#1)}
\newcommand\clos[1]{\mathrm{cl}#1}
\newcommand\hass[1]{\mathcal{H}#1}
\newcommand\hasso[1]{\mathcal{H}^o#1}
\newcommand\glob[2]{\mathcal{CM}o\ell_{#1}{#2}}
\newcommand\atom[2]{\mathcal{A}_{#1}{#2}}
\newcommand\mrg{\stackrel{*}{\leadsto}}
\newcommand\pfun{\rightharpoonup}
\newcommand\molec[2]{\mathcal{M}o\ell_{#1}{#2}}

\newcommand\submol{\sqsubseteq_\mathcal{M}}

\newcommand\loopd[2]{\mathcal{L}_{#1}#2}
\newcommand\maxd[2]{\mathcal{M}ax_{#1}#2}
\newcommand\hatto[1]{{#1_{\omega}}}
\newcommand\tensor{\,{\scaleobj{0.75}{\boxtimes}}\,}
\newcommand\join{\,{\star}\,}
\newcommand\jointil{\,{\tilde{\star}}\,}
\newcommand\face[1]{\mathfrak{F}{#1}}
\newcommand\incl{\hookrightarrow}
\newcommand\incliso{\stackrel{\sim}{\hookrightarrow}}

\newcommand\homoplax[2]{[#1,#2]_l}
\newcommand\homlax[2]{[#1,#2]_r}
\newcommand\joinl[2]{\langle#1,#2\rangle_l}
\newcommand\joinr[2]{\langle#1,#2\rangle_r}
\newcommand\colim[1]{\underset{#1}{\mathrm{colim}}\,}
\newcommand\celto{\Rightarrow}
\newcommand\divis[2]{\mathcal{U}niv_{#1}{#2}}
\newcommand\equi[2]{\mathcal{E}q_{#1}{#2}}
\newcommand\slice[2]{{#1}/{\raisebox{-2pt}{$#2$}}}
\newcommand\sliceco[2]{{#1}\!\!\stackrel{\mathrm{co}}{\scriptstyle/}\!\!{\raisebox{-2pt}{$#2$}}}

\newcommand\cat[1]{\mathbf{#1}}
\newcommand\globe{\cat{CAtom}}
\newcommand\globefree{\cat{CAtom}_\textit{free}}
\newcommand\globset{\omega\cat{Gph}}
\newcommand\pope{\cat{pOpe}}
\newcommand\popeset{\cat{pOpeSet}}
\newcommand\cpol{\cat{CPol}}
\newcommand\cpolfree{\cat{CPol}_\textit{free}}
\newcommand\rcpol{\cat{RCPol}}
\newcommand\pol{\cat{Pol}}
\newcommand\pos{\cat{Pos}}
\newcommand\ogpos{\cat{ogPos}_\textit{in}}
\newcommand\globpos{\cat{CDCpx}}
\newcommand\omegacat{\omega\cat{Cat}}
\newcommand\cghaus{\cat{cgHaus}}
\newcommand\sset{\cat{sSet}}
\newcommand\pomegacat{\cat{p}\omega\cat{Cat}}
\newcommand\adc{\cat{ADC}}
\newcommand\realis[1]{|#1|}

\setitemize{itemsep=0pt,topsep=1ex}
\setenumerate{itemsep=0pt,topsep=1ex}

\title{\Large\bfseries A combinatorial-topological \\shape category for polygraphs}

\author{Amar Hadzihasanovic\\[3pt] RIMS, Kyoto University}

\date{v1: June 2018 --- v2: July 2019}

\begin{document}
\maketitle

\vspace{-7pt}
\begin{minipage}{0.9\linewidth}
\small We introduce constructible directed complexes, a combinatorial presentation of higher categories inspired by constructible complexes in poset topology. Constructible directed complexes with a greatest element, called atoms, encompass common classes of higher-categorical cell shapes, including globes, cubes, oriented simplices, and a large sub-class of opetopes, and are closed under lax Gray products and joins. We define constructible polygraphs to be presheaves on a category of atoms and inclusions, and extend the monoidal structures. 

We show that constructible directed complexes are a well-behaved subclass of Steiner's directed complexes, which we use to define a realisation functor from constructible polygraphs to $\omega$\nbd categories. We prove that the realisation of a constructible polygraph is a polygraph in restricted cases, and in all cases conditionally to a conjecture. Finally, we define the geometric realisation of a constructible polygraph, and prove that it is a CW complex with one cell for each of its elements.
\end{minipage}

\tableofcontents

\section{Introduction}

This article is a study in the combinatorics of higher categories, in view of applications in higher-dimensional rewriting theory based on polygraphs \cite{burroni1993higher,lafont2007algebra,guiraud2016polygraphs}, also known as computads \cite{street1976limits}. It is technically most indebted to two different, and separate traditions of research:
\begin{enumerate}
	\item combinatorial presentations of strict $\omega$\nbd categories, such as pasting schemes \cite{johnson1989combinatorics,power1991pasting}, parity complexes \cite{street1991parity}, and most of all Steiner's directed complexes \cite{steiner1993algebra} and augmented directed complexes \cite{steiner2004omega};
	\item poset topology, in particular the theory of shellable and CW posets \cite{bjorner1980shellable,bjorner1983lexico,bjorner1984posets}, and the theory of constructible complexes \cite{hachimori2000constructible}.
\end{enumerate}
The two have much in common, yet as far as we know no significant connection has been made before.

The starting point of our work is the observation, made already by Burroni \cite{burroni1993higher}, that a polygraph is a ``formal CW complex'' inside the category $\omegacat$ of strict $\omega$\nbd categories, with certain models of directed $n$\nbd cells standing in for the topological $n$\nbd balls. This analogy has been effectively developed in a series of works \cite{metayer2003resolutions,metayer2008cofibrant}, culminating in the definition, by Lafont, M\'etayer, and Worytkiewicz, of a model structure on $\omegacat$ \cite{lafont2010folk} where polygraphs have the same role, as cofibrant objects, that CW complexes have in the classical Quillen-Serre model structure on topological spaces \cite[Section II.3]{quillen1967homotopical}.

Under some version of Grothendieck's homotopy hypothesis --- that the homotopy theory of topological spaces can be formulated in terms of higher groupoids or, \emph{a fortiori}, higher categories --- we would expect a direct relation between polygraphs and CW complexes, with polygraphs the strictly more general notion, admitting directed as well as undirected cells. However, this comparison is made difficult by the unrestricted combinatorics of $\omega$\nbd categorical (globular) pasting, which even for ``good'' classes of polygraphs that form presheaf categories may allow models of $n$\nbd cells whose natural geometric realisation is non-contractible; see the discussion in \cite[Appendix A]{henry2019non}.

Independently, and at the same time, S.\ Henry and the author came to the \emph{regularity} constraint as a solution \cite[Section 3]{hadzihasanovic2017algebra}. A regular polygraph is one whose $n$\nbd cell shapes have $k$\nbd dimensional boundaries shaped as $k$\nbd dimensional balls. Intuitively, any $\omega$\nbd categorical pasting diagram can be turned into a ball-shaped diagram by inserting enough units or degenerate cells: in this sense, regular polygraphs together with a notion of weak units should be sufficient as a foundation for a theory of higher categories. In \cite{henry2018regular}, Henry proved a version of the homotopy hypothesis for regular polygraphs together with an algebraic notion of ``regular composition'' and a non-algebraic notion of weak units. 

In contrast with the algebraic approach of Henry's work, where classes of polygraphs are proved to form presheaf categories before presenting a specific site of definition, we take a combinatorial approach, starting from an explicit description of cell shapes. In higher-dimensional rewriting based on strict $\omega$\nbd categories, the $\omega$\nbd categorical algebra structure is used fruitfully to describe the pasting of generators, but has also unintended effects:
\begin{enumerate}
	\item there is no internal way to distinguish generators from other cells, and one is forced to use complex syntactic arguments to perform basic operations, such as ``surgery'' of cells or reversing the direction of a generator;
	\item one may unwittingly introduce ``Eckmann-Hilton''-type degeneracies. 
\end{enumerate}
The second point is related to the well-known fact that strict $\omega$\nbd categories are strictly less general than weak higher categories \cite[Chapter 4]{simpson2009homotopy}, which also limits the contexts in which higher-dimensional theories presented as polygraphs can be intepreted. 

We suggest that one should separate the combinatorics of pasting of generators of a polygraph from the algebra of composition in its intended models. In this approach, the combinatorics of pasting are handled by the shape category, so that elements of a presheaf on the shape category correspond to generators of a polygraph, and manipulating generators is easy. Such presheaves are then used as underlying structures of a model of higher categories (which also ``has composition''), in which one can interpret higher-dimensional theories even if they do not internally ``have composition''. 

The idea of basing a model of higher categories on combinatorial presentations of pasting diagrams (richer than the ``special'' shapes such as globes, cubes, oriented simplices) was present in Kapranov and Voevodsky \cite{kapranov1991infty}, whose main results are notoriously incorrect, and seems to not have been pursued afterwards. We outline such a model in Appendix \ref{sec:represent}, and develop the idea further in a subsequent paper \cite{hadzihasanovic2019rds}.

\subsubsection*{Constructibility and lax Gray products}

Another motivation of this work is the computation of lax Gray products of polygraphs; in \cite[Chapter 2]{hadzihasanovic2017algebra}, we gave evidence of their relevance to categorical universal algebra and rewriting. 

The lax Gray product of $\omega$-categories \cite{al2002multiple}, an oriented, non-symmetric counterpart to the cartesian product of spaces, is notoriously difficult to express in the algebra of globular composition. Steiner's definition \cite{steiner2004omega}, while elegant, requires a long detour via augmented directed complexes, and the proof of its validity has been completed only recently, with some effort, by Ara and Maltsiniotis \cite{ara2016joint}. The \emph{join} is another monoidal structure which can be derived from lax Gray products, and similar considerations apply.

The problem may be characterised as follows. An $\omega$\nbd categorical $n$\nbd cell $x$ has, for $k < n$, an input and output $k$\nbd boundary, $\bord{k}{+}x$ and $\bord{k}{-}x$, which in good cases is $k$\nbd dimensional; a globular pasting of two cells is, topologically, a gluing along an entire $k$\nbd boundary (output for one cell, input for the other). If $x$ is an $n$\nbd cell and $y$ an $m$\nbd cell, their lax Gray product $x \tensor y$ is $(n+m)$\nbd dimensional, and its input $(n+m-1)$\nbd boundary splits as
\begin{equation} \label{eq:laxgray_decomp}
	(\bord{n-1}{-}x \tensor y) \cup (x \tensor \bord{m-1}{\pm}y),
\end{equation}
the second sign depending on the parity of $n$. But this decomposition is \emph{not} a globular pasting: the two parts only share a portion of their $(n+m-2)$\nbd boundary. The situation only becomes more complicated for lower-dimensional boundaries. 

With that in mind, we turned to the combinatorics of face posets of regular CW complexes, or CW posets \cite{bjorner1984posets}. Certain conditions that have been studied for CW posets, such as shellability \cite{bjorner1980shellable}, turn out to have clear oriented analogues, suggesting a notion of composition alternative to the globular model. 

In particular, the oriented analogue of \emph{constructibility} \cite{hachimori2000constructible} individuates a class of pasting diagrams which are iterated composites of $n$\nbd cells along an $(n-1)$\nbd dimensional diagram in their boundary; we call this kind of composition a \emph{merger}. The decomposition (\ref{eq:laxgray_decomp}) is, in fact, a merger; thus the oriented version of constructibility is very naturally compatible with lax Gray products and joins. In \cite{hadzihasanovic2018weak}, we gave a reformulation of bicategory theory where the algebra of composition is based on mergers instead of globular composition, and showed that it is equivalent to the standard one. 

We introduce \emph{constructible directed complexes} as a combinatorial presentation of pasting diagrams obtained by iterated mergers, which we call \emph{constructible molecules}. The terminology is a twist on Steiner's theory of directed complexes \cite{steiner1993algebra}, whose ``molecules'' are pasting diagrams obtained by iterated globular compositions. It is not immediately obvious that constructible molecules are molecules, and that constructible directed complexes are directed complexes; we prove it in Section \ref{sec:omegacat}. Moreover, constructible molecules are regular: the geometric realisation of a constructible $n$\nbd molecule is a regular CW $n$\nbd ball, as shown in Section \ref{sec:geometric}.

We define \emph{constructible polygraphs} as a generalisation of constructible directed complexes: they are presheaves on a category of constructible atoms (``indecomposable'' constructible molecules) and inclusions. The monoidal structures of lax Gray products and joins canonically extend, and so does geometric realisation: the geometric realisation of a constructible polygraph is a CW complex with one cell for each of its elements. 

Moreover, constructible polygraphs can be realised as $\omega$\nbd categories, but we do not have a complete understanding of this construction yet: in particular, while we conjecture that the $\omega$\nbd categorical realisation of a constructible polygraph always admits the structure of a polygraph with one generator for each of its elements, we were only able to prove it in restricted cases.

\subsubsection*{Structure of the article}
We start in Section \ref{sec:basic} by recalling some notions of order theory and poset topology. Then, we define an orientation on graded posets, in the form of a labelling of their Hasse diagrams, and progressively introduce the components of the definition of constructible directed complex: oriented thinness (Definition \ref{dfn:thinness}) and constructible molecules (Definition \ref{dfn:globes}). We define the category $\globpos$ of constructible directed complexes and inclusions, and prove some of its properties, most importantly that constructible molecules do not have non-trivial automorphisms (Lemma \ref{lem:noautomorph}). We discuss some special classes of constructible molecules, including a sub-class of the opetopes, and then define the category $\cpol$ of constructible polygraphs.

In Section \ref{sec:properties}, we expand on the combinatorics of constructible directed complexes with the use of two main tools: substitutions (Construction \ref{cons:substitution}) and mergers (Construction \ref{cons:mergers}), where constructible molecules with two maximal elements are sequentially replaced with \emph{atoms}, constructible molecules with a greatest element. We prove that any constructible molecule can be reduced by mergers to an atom (Lemma \ref{lem:atomicmerge}), and we use this to prove that constructible molecules satisfy the globularity condition: their input and output boundary have the same boundaries (Theorem \ref{thm:globularity}). We then consider some consequences of globularity: $n$\nbd dimensional atoms are classified by pairs of constructible $(n-1)$\nbd molecules with isomorphic boundaries (Proposition \ref{prop:atomicnglobe}), and any constructible molecule can be reduced to an atom by a sequence of mergers (Lemma \ref{lem:standardmerge}). We also give a direct characterisation of the $k$\nbd dimensional boundaries of a constructible molecule (Proposition \ref{prop:globelike}).

In Section \ref{sec:monoidal}, we focus on monoidal structures and constructions. We define the suspension (Construction \ref{cons:suspension}), the lax Gray product (Construction \ref{cons:laxgray}), the join (Construction \ref{cons:join}), and the duals (Construction \ref{cons:jdual}) of oriented graded posets, and prove that constructible directed complexes are closed under each of them. We also show that some common classes of higher-categorical cell shapes --- the globes, the cubes, and the oriented simplices --- are constructible molecules. We then study the extensions of lax Gray products and joins to constructible polygraphs, and prove that the first are part of a monoidal biclosed structure, and the second of a locally biclosed structure (Proposition \ref{prop:locallybiclosed}) on $\cpol$. The latter result relies on a \emph{slice} construction for constructible polygraphs (Construction \ref{cons:slices}).

In Section \ref{sec:omegacat}, we first recall the definitions of $\omega$\nbd category and of polygraph. Then, we rephrase Steiner's definition of directed complexes in our notation (Definition \ref{dfn:directedcomplex}), and prove that constructible directed complexes are directed complexes (Theorem \ref{thm:globmolec}). We use this to define a realisation functor for constructible polygraphs in the category $\omegacat$ of $\omega$\nbd categories, and give some partial results on the relation between the monoidal structures of $\cpol$ and those of $\omegacat$.

In Section \ref{sec:polygraphs}, we study the relation between constructible polygraphs and polygraphs as freely generated $\omega$\nbd categories. The realisation of a constructible polygraph in $\omegacat$ admits the structure of a polygraph if constructible directed complexes are ``freely generating'' in the sense of Definition \ref{dfn:freegen} (Theorem \ref{thm:rpolintopol}), which we conjecture to be true, but were not able to prove in the general case. Nevertheless, we give some criteria for free generation which are weaker than loop-freeness as considered by Steiner (Proposition \ref{prop:max_free} and Corollary \ref{cor:flow-conn-split}), and may serve as intermediate steps towards a general proof.

In Section \ref{sec:geometric}, we define the geometric realisation of an oriented graded poset. We prove that the underlying poset of a constructible molecule is recursively dividable in the sense of Hachimori (Proposition \ref{prop:recdivid}), and derive that the geometric realisation of a constructible $n$\nbd molecule is an $n$\nbd ball (Theorem \ref{thm:regularball}). We extend the geometric realisation to constructible polygraphs, and show that the geometric realisation of a constructible polygraph is a CW complex with one cell for each of its elements (Theorem \ref{thm:georegpoly}). 

In Appendix \ref{sec:appendix}, we give a more detailed proof of the fact that lax Gray products of constructible molecules are constructible molecules. In Appendix \ref{sec:represent}, we present the first steps in an approach to weak higher categories based on constructible polygraphs, via the key notion of universality of a cell at a constructible submolecule of its boundary (Definition \ref{dfn:universality}), up to the definition of representable constructible polygraph (Definition \ref{dfn:represent}).

\subsubsection*{Outlook and open problems}

Our first concern is to develop the ideas of this work into a good framework for higher-dimensional rewriting, including a suitable model of higher categories to serve as semantics. The main limitation that we see is that operations or rewrites with a ``nullary'' input cannot be directly modelled in a constructible or regular polygraph; this should be overcome with an adequate theory of units or degeneracies simulating nullary inputs. We do that in \cite{hadzihasanovic2019rds}.

We have stated three conjectures of quite different nature. Conjecture \ref{conj:subglobe} is a problem in the same vein as many ``hierarchy problems'' in combinatorial topology, investigating the relation between different properties of posets or simplicial complexes (being shellable, partitionable, Cohen-Macaulay...); it may be rephrased as (oriented) constructibility not being the same as \emph{extendable} constructibility. We believe that a counterexample should exist, but were not able to find one.

On the other hand, Conjecture \ref{conj:monoidality}, that the lax Gray product and join of constructible polygraphs coincide with those of $\omega$\nbd categories \cite{ara2016joint}, is almost obviously true, and the fact that there is no obvious proof says more about how difficult it is to work with these structures on $\omegacat$ than anything else. Similarly, we are almost certain that our positive opetopes (Definition \ref{dfn:pope}) coincide with those defined by Zawadowski \cite{zawadowski2017positive}, and the only effort required is setting up a framework in which to compare the two.

The most important open problem is Conjecture \ref{conj:freegen}. Because of Example \ref{exm:nonsplit}, the strategy used essentially in all previous proofs of free generation does not apply to general constructible directed complexes. Proving the conjecture would be a useful sanity check, reassuring us that whatever we do is a continuation of the general theory of polygraphs. A refutation, on the other hand, would not rule out constructible polygraphs as a notion of ``directed space'', but it would make them somewhat incomparable with strict $\omega$\nbd categories, and indicate that the combinatorics of higher-categorical pasting are even subtler than we think.

Finally, given that constructible polygraphs rely on quite simple data structures, it would be interesting to pursue a formal implementation in the vein of \cite{curien2019syntactic} for opetopes and opetopic sets. A potential obstacle is that our definition of constructible molecules is combinatorial, but it is not enumerative: so far, we have no inductive procedure for enumerating isomorphism classes of constructible molecules without repetition.

\subsubsection*{Errata and notes on an earlier version}
With respect to the first preprint version of this article, we have made extensive changes to the terminology: what are now \emph{constructible molecules} and \emph{constructible directed complexes} used to be called \emph{globes} and \emph{globular posets}, while \emph{constructible polygraphs} used to be \emph{regular polygraphs}. 

The fact that constructible directed complexes are directed complexes in Steiner's sense is not obvious, and a good part of this article is devoted to its proof; but once it is established, it seems unnecessary to introduce an entirely separate set of terms. We thus suggest adopting Steiner's ``molecule'' for the most general combinatorial presentation of an $\omega$\nbd categorical composite, and prefixing it with various adjectives. It seems likely that there should be a hierarchy of classes of molecules as rich as that of their unoriented counterparts in poset topology. 

We use ``constructible'' here instead of ``regular'' to avoid confusion with S.\ Henry's notion of regular polygraph, which we find is the more general and appropriate use of the term: constructibility implies regularity, but the converse is not true (see Remark \ref{rmk:nonconstructible}). We discuss ``regular directed complexes'' in the follow-up \cite{hadzihasanovic2019rds}.

The first version also contained some mistakes: Lemma 2.23 was wrong, and so were Theorem 4.41 and the lemmas leading to it, whose proof made crucially use of it. The corresponding results in Section \ref{sec:polygraphs} here have now additional hypotheses, which do not hold for all constructible directed complexes. In particular, it is not true, as claimed previously, that all molecules in a constructible directed complex are split. The validity of Theorem 4.43 from the first version, that the $\omega$\nbd categorical realisations of constructible polygraphs are all polygraphs, is now conditional on Conjecture \ref{conj:freegen}, which will need a different proof strategy.

\subsubsection*{Acknowledgements}
This work was supported by a JSPS Postdoctoral Research Fellowship and by JSPS KAKENHI Grant Number 17F17810. Many thanks to Joachim Kock and Jamie Vicary for their feedback on the parts which overlap with my thesis. For the revised version, I would like to thank Dimitri Ara, Yves Guiraud, Simon Henry, and Georges Maltsiniotis for helpful feedback and suggestions.

\section{Basic definitions} \label{sec:basic}

In this section, we introduce our basic definitions. We borrow extensively from poset topology, introducing ``oriented'' or ``directed'' versions of several notions, in order to model the directionality of higher-categorical cells. The key definitions are those of a \emph{constructible molecule} and of a \emph{constructible directed complex}. For now we do not go beyond some basic properties of constructible directed complexes; instead we list some familiar examples, and proceed to describe several constructions by which we can build new constructible directed complexes from simpler ones.

First of all, we recall some basic notions from order theory and poset topology; we refer to \cite{wachs2006poset} for an overview of the subject. As a rule, we will not notationally distinguish a poset from its underlying set. 

\begin{dfn}
Let $P$ be a finite poset with order relation $\leq$. For all elements $x, y \in P$, we say that $y$ \emph{covers} $x$ if $x < y$ and, for all $y' \in X$, if $x < y' \leq y$, then $y' = y$. 

The \emph{Hasse diagram} of $P$ is the finite directed graph $\hass{P}$ with $\hass{P}_0 := P$ as set of vertices, and $\hass{P}_1 := \{c_{y,x}: y \to x \,|\, y \text{ covers } x\}$ as set of edges. We can reconstruct the partial order on $P$ from its Hasse diagram, letting $x \leq y$ in $P$ if and only if there is a path from $y$ to $x$ in $\hass{P}$.

Let $P_\bot$ be $P$ extended with a least element $\bot$. We say that $P$ is \emph{graded} if, for each $x \in P$, all paths from $x$ to $\bot$ in $\hass{P}_\bot$ have the same length. If $P$ is graded, for each $x \in P$, let $n+1$ be the length of paths from $x$ to $\bot$; then, we define $\dmn{x} := n$, the \emph{dimension} of $x$, and let $P^{(n)} := \{x \in P \,|\, \dmn{x} = n\}$. 

For all $x, y \in P$ such that $x \leq y$, we call the subset $[x,y] := \{z \in P \,|\, x \leq z \leq y\}$ the \emph{interval} from $x$ to $y$. If $P$ is graded, all paths from $y$ to $x$ in $\hass{P}$ have length $\dmn{y} - \dmn{x}$; this is the \emph{length} of the interval $[x,y]$. A graded poset $P$ is \emph{thin} if all intervals of length 2 in $P_\bot$ contain precisely 4 elements, that is, they are of the form
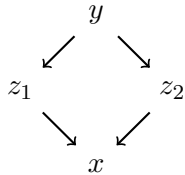
\begin{equation*}
\begin{tikzpicture}[baseline={([yshift=-.5ex]current bounding box.center)}]
	\node[scale=1.25] (0) at (0,2) {$y$};
	\node[scale=1.25] (1) at (-1,1) {$z_1$};
	\node[scale=1.25] (1b) at (1,1) {$z_2$};
	\node[scale=1.25] (2) at (0,0) {$x$};
	\draw[1c] (0) to (1);
	\draw[1c] (0) to (1b);
	\draw[1c] (1) to (2);
	\draw[1c] (1b) to (2);
\end{tikzpicture}
\end{equation*}
in the Hasse diagram $\hass{P}_\bot$.
\end{dfn}

\begin{remark}The Hasse diagram may have the opposite direction elsewhere in the literature.
\end{remark}

\begin{dfn}
Let $P$ be a poset, and $U \subseteq P$. The \emph{closure} of $U$ is the subset $\clos{U} := \{x \in P \,|\, \exists y \in U \,x \leq y\}$ of $P$. We say that $U$ is \emph{closed} if $U = \clos{U}$. 

Suppose $P$ is graded, and $U \subseteq P$ is closed; $U$ is also graded with the partial order inherited from $P$. We write $\dmn{U} := \mathrm{max}\{\dmn{x} \,|\, x \in U\}$ if $U$ is inhabited, and $\dmn{\emptyset} = -1$; in particular, $\dmn{\clos{\{x\}}} = \dmn{x}$. We say that $U$ is \emph{pure} if all its maximal elements have dimension $n = \dmn{U}$, or equivalently if $U = \clos{(U^{(n)})}$.
\end{dfn}

The intuition should be that elements $x$ of dimension $n$ in a graded poset correspond to $n$\nbd cells in a pasting diagram, and $(n-1)$\nbd dimensional elements covered by $x$ to $(n-1)$\nbd cells in the boundary of $x$. The boundary of a higher-categorical cell is partitioned into an input and an output part: we express this through a labelling of Hasse diagrams.

\begin{dfn}
An \emph{oriented graded poset} is a finite graded poset $P$ together with an edge-labelling $o: \hass{P}_1 \to \{+,-\}$ of the Hasse diagram of $P$ (an \emph{orientation}). 

If $P, Q$ are oriented graded posets, an \emph{inclusion} $\imath: P \hookrightarrow Q$ is a closed embedding of posets that is compatible with the orientations, that is, $o_Q(c_{\imath(y),\imath(x)}) = o_P(c_{y,x})$ for all $y, x \in P$ such that $y$ covers $x$. An inclusion is an \emph{isomorphism} if it is surjective. 

We write $\ogpos$ for the category of oriented graded posets and inclusions. There is an obvious forgetful functor $\ogpos \to \pos$ to the category of posets and order-preserving maps.
\end{dfn}

\begin{remark}
As a closed embedding of graded posets, an inclusion preserves dimensions and preserves and reflects the covering relation: if $y' = \imath(y)$ and $y'$ covers $x'$, then there is a unique $x$ such that $y$ covers $x$ and $x' = \imath(x)$. 

Any closed subset $U \subseteq P$ inherits the structure of an oriented graded poset, which makes the inclusion of subsets an inclusion $U \hookrightarrow P$ of oriented graded posets.
\end{remark}

\begin{cons}
Given an oriented graded poset $P$, let $\hasso{P}$ be the directed graph obtained from $\hass{P}$ by reversing the edges labelled $-$: that is, for all $c_{y,x}: y \to x$ in $\hass{P}$,
\begin{itemize}
	\item if $o(c_{y,x}) = +$, there is an edge $c^+_{y,x}: y \to x$ in $\hasso{P}$;
	\item if $o(c_{y,x}) = -$, there is an edge $c^-_{y,x}: x \to y$ in $\hasso{P}$.
\end{itemize}
If $P$ is graded, we can reconstruct $P$ and its orientation from $\hasso{P}$ and the vertex-labelling $\mathrm{dim}: P \to \mathbb{N}$: it suffices to reverse the edges going from a vertex of lower dimension to one of higher dimension in order to recover $\hass{P}$, and those edges are precisely the ones labelled $-$ by the orientation.

In fact, oriented graded posets correspond bijectively to finite directed graphs $H$ with a vertex-labelling $\mathrm{dim}: H_0 \to \mathbb{N}$, satisfying the following two properties:
\begin{enumerate}
	\item for all edges $y \to x$, either $\dmn{x} = \dmn{y} - 1$ or $\dmn{x} = \dmn{y} + 1$;
	\item for all vertices $y$, if $\dmn{y} > 0$, there exist a vertex $x$ with $\dmn{x} = \dmn{y} - 1$ and an edge $y \to x$ or $x \to y$.
\end{enumerate}

To depict an oriented graded poset, we will use interchangeably $\hass{P}$ with an edge-labelling or $\hasso{P}$ with relative vertical placement indicating dimension.
\end{cons}

\begin{dfn}
Let $U$ be a closed $n$\nbd dimensional subset of an oriented graded poset. For $\alpha \in \{+,-\}$, let
\begin{align*}
	\sbord{}{\alpha} U & := \{x \in U \,|\, \dmn{x} = n-1 \text{ and, for all $y \in U$, if $y$ covers $x$, then $o(c_{y,x}) = \alpha$} \}, \\
	\bord{}{\alpha} U & := \clos{(\sbord{}{\alpha} U)} \cup \{ x \in U \,|\, \text{for all $y \in U$, if $x \leq y$, then $\dmn{y} < n$} \}, \\
	\sbord{}{} U & := \sbord{}{+}U \cup \sbord{}{-}U, \quad \qquad \quad \bord{}{}U := \bord{}{+}U \cup \bord{}{-} U.
\end{align*}
In particular, when $U$ is pure, $\bord{}{\alpha} U = \mathrm{cl}(\sbord{}{\alpha} U)$.

We call $\bord{}{-}U$ the \emph{input boundary} and $\bord{}{+}U$ the \emph{output boundary} of $U$; together, they form the \emph{boundary} $\bord{}{}U$ of $U$. For all $x \in P$, we will use the short-hand notation $\sbord{}{\alpha}x := \sbord{}{\alpha}\clos{\{x\}}$ and $\bord{}{\alpha}x := \bord{}{\alpha}\clos{\{x\}}$.
\end{dfn}

\begin{remark} \label{rmk:precomplex}
Another equivalent presentation of an oriented graded poset consists of the function $\mathrm{dim}: P \to \mathbb{N}$ together with disjoint sets $\sbord{}{+}x$ and $\sbord{}{-}x$ of $(n-1)$\nbd dimensional elements, for all $n$ and elements $x \in P^{(n)}$. If we remove the requirement that $\sbord{}{+}x$ and $\sbord{}{-}x$ be disjoint, we obtain what Steiner calls a \emph{directed precomplex} in \cite{steiner1993algebra}.
\end{remark}

Given a graded poset $P$ with orientation $o$, we extend $o$ to $P_\bot$, by setting $o(c_{x,\bot}) := +$ for all minimal elements $x$ of $P$. This turns $(-)_\bot$ into an endofunctor on $\ogpos$. 

In what follows, we assume the usual ``sign rule'' multiplication on $\{+,-\}$. We will let variables $\alpha, \beta, \ldots$ range implicitly over $\{+,-\}$.

\begin{dfn} \label{dfn:thinness}
An \emph{oriented thin poset} is an oriented graded poset $P$ with the following property: $P$ is thin, and for each interval $[x,y]$ of length 2 in $P_\bot$, the labelling
\begin{equation} \label{eq:signed}
\begin{tikzpicture}[baseline={([yshift=-.5ex]current bounding box.center)}]
	\node[scale=1.25] (0) at (0,2) {$y$};
	\node[scale=1.25] (1) at (-1,1) {$z_1$};
	\node[scale=1.25] (1b) at (1,1) {$z_2$};
	\node[scale=1.25] (2) at (0,0) {$x$};
	\draw[1c] (0) to node[auto,swap] {$\alpha_1$} (1);
	\draw[1c] (0) to node[auto] {$\alpha_2$} (1b);
	\draw[1c] (1) to node[auto,swap] {$\beta_1$} (2);
	\draw[1c] (1b) to node[auto] {$\beta_2$} (2);
\end{tikzpicture}
\end{equation}
satisfies $\alpha_1\beta_1 = -\alpha_2\beta_2$. 
\end{dfn}

\begin{exm} \label{exm:standard}
For each $n \in \mathbb{N}$, let $O^n$ be the poset with a pair of elements $\underline{k}^+, \underline{k}^-$ for each $k < n$ and a greatest element $\underline{n}$, with the partial order defined by $\underline{j}^\alpha \leq \underline{k}^\beta$ if and only if $j \leq k$. This is a graded poset, with $\dmn{\underline{n}} = n$ and $\dmn{\underline{k}^\alpha} = k$ for all $k < n$. 

With the orientation $o(c_{y,x}) := \alpha$ if $x = \underline{k}^\alpha$, and $\alpha \in \{+,-\}$, it becomes an oriented thin poset; in fact, it is the smallest $n$\nbd dimensional oriented thin poset. The following are depictions of $\hass{O^2}$ and of $\hasso{O^2}$, together with the corresponding pasting diagram. 

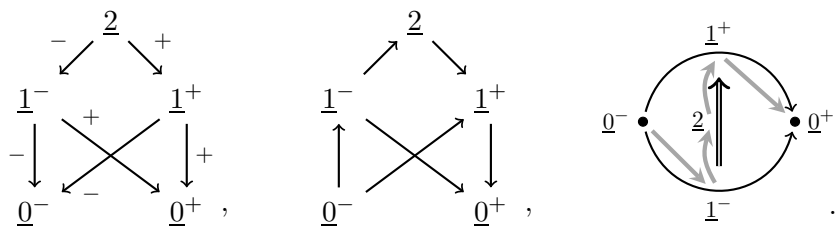
\begin{equation*}
\begin{tikzpicture}[baseline={([yshift=-.5ex]current bounding box.center)}]
\begin{scope}
	\node[scale=1.25] (0) at (0,2.5) {$\underline{2}$};
	\node[scale=1.25] (1) at (-1,1.5) {$\underline{1}^-$};
	\node[scale=1.25] (1b) at (1,1.5) {$\underline{1}^+$};
	\node[scale=1.25] (2) at (-1,0) {$\underline{0}^-$};
	\node[scale=1.25] (2b) at (1,0) {$\underline{0}^+$};
	\draw[1clong] (0) to node[auto,swap] {$-$} (1);
	\draw[1clong] (0) to node[auto] {$+$} (1b);
	\draw[1clong] (1) to node[auto,swap] {$-$} (2);
	\draw[1clong] (1b) to node[auto, pos=0.8] {$\!\!-$} (2);
	\draw[1clong] (1) to node[auto, pos=0.2] {$\!\!+$} (2b);
	\draw[1clong] (1b) to node[auto] {$+$} (2b);
	\node[scale=1.25] at (1.5,0) {$,$};
\end{scope}
\begin{scope}[shift={(4,0)}]
	\node[scale=1.25] (0) at (0,2.5) {$\underline{2}$};
	\node[scale=1.25] (1) at (-1,1.5) {$\underline{1}^-$};
	\node[scale=1.25] (1b) at (1,1.5) {$\underline{1}^+$};
	\node[scale=1.25] (2) at (-1,0) {$\underline{0}^-$};
	\node[scale=1.25] (2b) at (1,0) {$\underline{0}^+$};
	\draw[1clong] (1) to (0);
	\draw[1clong] (0) to (1b);
	\draw[1clong] (2) to (1);
	\draw[1clong] (2) to (1b);
	\draw[1clong] (1) to (2b);
	\draw[1clong] (1b) to (2b);
	\node[scale=1.25] at (1.5,0) {$,$};
\end{scope}
\begin{scope}[shift={(8,1.2)}]
	\node[0c] (a0) at (-1,0) [label=left:$\underline{0}^-$] {};
	\node[0c] (b0) at (1,0) [label=right:$\underline{0}^+$] {};
	\draw[1c, out=75, in=105, looseness=1.5] (a0) to node[auto] {$\underline{1}^+$} (b0);
	\draw[1c, out=-75, in=-105, looseness=1.5] (a0) to node[auto,swap] {$\underline{1}^-$} (b0);
	\draw[follow] (a0) to (-.1,-.9);
	\draw[follow, out=120, in=-105, looseness=1.2] (0,-0.9) to (-.1,0);
	\draw[follow, out=105, in=-120, looseness=1.2] (-.1,0) to (0,.9);
	\draw[follow] (0,.9) to (b0);
	\draw[2c] (0,-0.8) to node[auto] {$\underline{2}\;$} (0,0.8);
	\node[scale=1.25] at (1.5,-1.2) {.};
\end{scope}
\end{tikzpicture}
\end{equation*}
We can see the graph $\hasso{O^2}$ as describing a ``flow'' on the cells of $O^2$, from the input boundary of a cell, into the cell, into the output boundary; this is indicated by the grey arrows.
\end{exm}

\begin{dfn}
We call the oriented graded poset $O^n$ the \emph{$n$\nbd globe}.
\end{dfn}

\begin{remark}
Let $\cat{O}$ be the full subcategory of $\ogpos$ whose objects are the globes. The category $\globset$ of presheaves on $\cat{O}$ is the category of $\omega$\nbd graphs \cite[Section 1.4]{leinster2004higher}, also known as globular sets.
\end{remark}

Thinness (also known as the \emph{diamond property} in the theory of abstract polytopes \cite{mcmullen2002abstract}) is a local condition, which imposes that the cells be ``manifold-like'' by ruling out irregular situations, such as three 1-cells in the boundary of a 2-cell meeting at a single point. Compatibility of the orientation will be necessary for globularity (case $\alpha_1 \neq \alpha_2$), and composability of cells in the same half of the boundary (case $\alpha_1 = \alpha_2$). 

In order to obtain regular pasting diagrams, we need to complement oriented thinness with a global condition, preventing cells from having globally non-spherical (for example, toroidal) boundaries, and ensuring that they have two composable, ball-shaped hemispheres.

\begin{dfn} \label{dfn:globes}
Let $P$ be an oriented thin poset. We define for each $n \in \mathbb{N}$ a family $\glob{n}{P}$ of pure $n$\nbd dimensional subsets of $P$, the \emph{constructible $n$\nbd molecules} of $P$, together with a partial order $\sqsubseteq$ on $\glob{n}{P}$, to be read ``is a constructible submolecule of''.

The family $\glob{0}{P}$ is $\{\{x\} \,|\, \dmn{x} = 0\}$, with the discrete order. Let $\glob{n-1}{P}$ be defined, and $U$ be pure and $n$\nbd dimensional. Then $U \in \glob{n}{P}$ if and only if $\bord{}{+}U$ and $\bord{}{-}U$ are constructible $(n-1)$\nbd molecules, and, inductively on the number of maximal elements of $U$, either
\begin{itemize}
	\item $U$ has a greatest element, in which case we call it an \emph{atom}, or
	\item $U = U_1 \cup U_2$, where $U_1$ and $U_2$ are constructible $n$\nbd molecules with fewer maximal elements, such that
	\begin{enumerate}
		\item $U_1 \cap U_2 = \bord{}{+}U_1 \cap \bord{}{-}U_2$ is a constructible $(n-1)$\nbd molecule, and
		\item $\bord{}{-}U_1 \sqsubseteq \bord{}{-}U$, $\bord{}{+}U_2 \sqsubseteq \bord{}{+}U$, $U_1 \cap U_2 \sqsubseteq \bord{}{+}U_1$, and $U_1 \cap U_2 \sqsubseteq \bord{}{-}U_2$.
	\end{enumerate}
\end{itemize}
We define $\sqsubseteq$ to be the smallest partial order relation on $\glob{n}{P}$ such that $U_1, U_2 \sqsubseteq U$, for all triples $U, U_1, U_2$ in the latter situation.

We say that the oriented thin poset $P$ is a constructible $n$\nbd molecule if $P \in \glob{n}{P}$, and a constructible $n$\nbd atom if it is a constructible $n$\nbd molecule with a greatest element.
\end{dfn}

\begin{remark}
Our notion of constructible molecule is an oriented version of a recursively dividable poset, which in turn is the order\nbd theoretic version of a constructible complex, as defined in \cite{hachimori2000constructible}. We will make the connection precise in Section \ref{sec:geometric}.
\end{remark}

\begin{remark}
The definition of constructible $n$\nbd molecule can be formulated in any oriented graded poset, not necessarily thin, but it leads to identifying as ``molecules'' posets that do not correspond to any $n$\nbd categorical cell shape. For example, $\clos{\{x\}}$ is a ``constructible 2-molecule'' in the oriented graded poset (depicted as the diagram $\hasso{P}$)
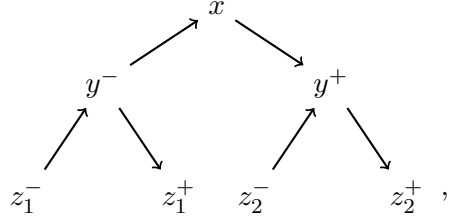
\begin{equation*}
\begin{tikzpicture}[baseline={([yshift=-.5ex]current bounding box.center)}]
	\node[scale=1.25] (0) at (0,2.5) {$x$};
	\node[scale=1.25] (1) at (-1.5,1.5) {$y^-$};
	\node[scale=1.25] (1b) at (1.5,1.5) {$y^+$};
	\node[scale=1.25] (2) at (-2.5,0) {$z_1^-$};
	\node[scale=1.25] (2b) at (-.5,0) {$z_1^+$};
	\node[scale=1.25] (2c) at (.5,0) {$z_2^-$};
	\node[scale=1.25] (2d) at (2.5,0) {$z_2^+$};
	\draw[1clong] (1) to (0);
	\draw[1clong] (0) to (1b);
	\draw[1clong] (2) to (1);
	\draw[1clong] (1) to (2b);
	\draw[1clong] (2c) to (1b);
	\draw[1clong] (1b) to (2d);
	\node[scale=1.25] at (3,0) {$,$};
\end{tikzpicture}
\end{equation*}
which does not describe a 2-cell.
\end{remark}

\begin{exm}
The $n$\nbd globes $O^n$ are constructible $n$\nbd atoms. This is obvious for $O^0$; for $n > 0$, $\bord{}{+}O^n$ and $\bord{}{-}O^n$ are both isomorphic to $O^{n-1}$, which is a constructible $(n-1)$\nbd atom by the inductive hypothesis.
\end{exm}

\begin{exm}
Let us enumerate the lowest-dimensional constructible molecules. By definition, there is only one constructible 0-molecule, the 0-globe. 

The 1-globe $O^1$, which we also denote by $\vec{I}$, is the only oriented thin poset with a 1-dimensional greatest element, and also the only constructible 1-atom. Any other constructible 1-molecule is the ``concatenation'' $\#_0^k \vec{I}$ of $k > 0$ atoms, corresponding to a pasting diagram
\begin{equation*}
\begin{tikzpicture}[baseline={([yshift=-.5ex]current bounding box.center)}]
\begin{scope}
	\node[0c] (0) at (-2,0) {};
	\node[0c] (1) at (0,0) {};
	\node[0c] (2) at (2,0) {};
	\draw[1c] (0) to (1);
	\draw[1cdot] (1) to node[auto] {$k-1$} (2);
	\node[scale=1.25] at (2.75,0) {\text{or}};
\end{scope}
\begin{scope}[shift={(4.5,0)}]
	\node[0c] (0) at (-1,0) {};
	\node[0c] (1) at (1,0) {};
	\draw[1cdash] (0) to node[auto] {$k$} (1);
	\node[scale=1.25] at (1.5,0) {$;$};
\end{scope}
\end{tikzpicture}
\end{equation*}
here, we used a dotted edge to indicate a possibly empty finite sequence of edges, and a dashed one to indicate a non-empty one. It is easy to check that any constructible 1-molecule included in another constructible 1-molecule is a submolecule.

We will later prove (Proposition \ref{prop:atomicnglobe}) that constructible $n$\nbd atoms, for $n > 0$, are classified by pairs of constructible $(n-1)$\nbd molecules with isomorphic boundaries. Because any two constructible 1-molecules have isomorphic boundaries, constructible 2-atoms are classified by pairs $(n,m)$ of non-zero natural numbers:
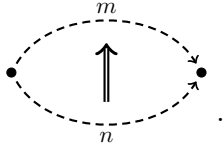
\begin{equation*}
\begin{tikzpicture}
	\node[0c] (0) at (-1.25,0) {};
	\node[0c] (1) at (1.25,0) {};
	\draw[1cdash, out=60, in=120] (0) to node[auto] {$m$} (1);  
	\draw[1cdash, out=-60, in=-120] (0) to node[auto,swap] {$n$} (1);
	\draw[2c] (0,-0.6) to (0,0.6);
	\node[scale=1.25] at (1.5,-0.625) {.};
\end{tikzpicture}

\end{equation*}
A non-atomic, constructible 2-molecule must have the same boundary as a constructible 2-atom, and split into two constructible 2-molecules $U_1$, $U_2$ whose intersection is a constructible 1-molecule contained in $\bord{}{+}U_1 \cap \bord{}{-}U_2$; the submolecule conditions are always satisfied. By a case distinction, this must happen in one of the following four ways:
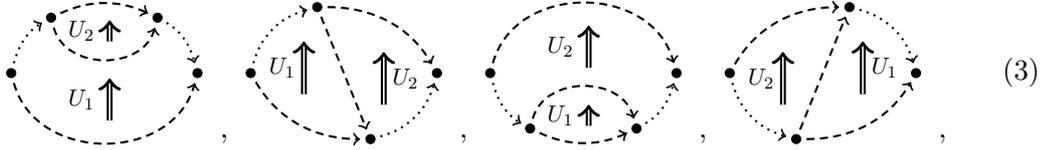
\begin{equation} \label{eq:fourways}
\begin{tikzpicture}[baseline={([yshift=-.5ex]current bounding box.center)}, scale=0.7]
\begin{scope}
	\node[0c] (0) at (-1.75,0) {};
	\node[0c] (a1) at (-1,1.05) {};
	\node[0c] (a2) at (1,1.05) {};
	\node[0c] (1) at (1.75,0) {};
	\draw[1cdot, out=75, in=-150] (0) to (a1);  
	\draw[1cdash, out=25,in=155] (a1) to (a2);
	\draw[1cdot, out=-30, in=105] (a2) to (1);
	\draw[1cdash, out=-75,in=-105, looseness=1.25] (0) to (1);
	\draw[1cdash, out=-75, in=-105, looseness=1.25] (a1) to (a2);
	\draw[2c] (.1,-1.2) to node[auto] {$U_1\;$} (.1,.2);
	\draw[2c] (.1,.3) to node[auto] {$U_2\;$} (.1,1.3);
	\node[scale=1.25] at (2.25,-1.25) {,};
\end{scope}
\begin{scope}[shift={(4.5,0)}]
	\node[0c] (a) at (-1.75,0) {};
	\node[0c] (b) at (.5,-1.25) {};
	\node[0c] (c) at (-.5,1.25) {};
	\node[0c] (d) at (1.75,0) {};
	\draw[1cdash, out=-60, in=180] (a) to (b);
	\draw[1cdash, out=0, in=120] (c) to (d);
	\draw[1cdot, out=15, in=-105] (b) to (d);
	\draw[1cdot, out=75, in=-165] (a) to (c);
	\draw[1cdash] (c) to (b);
	\draw[2c] (-.75,-.7) to node[auto] {$U_1$} (-.75,.9);
	\draw[2c] (.75,-.9) to node[auto,swap] {$U_2$} (.75,.7);
	\node[scale=1.25] at (2.25,-1.25) {,};
\end{scope}
\begin{scope}[shift={(9,0)}]
	\node[0c] (0) at (-1.75,0) {};
	\node[0c] (a1) at (-1,-1.05) {};
	\node[0c] (a2) at (1,-1.05) {};
	\node[0c] (1) at (1.75,0) {};
	\draw[1cdot, out=-75, in=150] (0) to (a1);  
	\draw[1cdash, out=-25,in=-155] (a1) to (a2);
	\draw[1cdot, out=30, in=-105] (a2) to (1);
	\draw[1cdash, out=75,in=105, looseness=1.25] (0) to (1);
	\draw[1cdash, out=75, in=105, looseness=1.25] (a1) to (a2);
	\draw[2c] (.1,-.2) to node[auto] {$U_2\;$} (.1,1.2);
	\draw[2c] (.1,-1.3) to node[auto] {$U_1\;$} (.1,-.3);
	\node[scale=1.25] at (2.25,-1.25) {,};
\end{scope}
\begin{scope}[shift={(13.5,0)}]
	\node[0c] (a) at (-1.75,0) {};
	\node[0c] (b) at (-.5,-1.25) {};
	\node[0c] (c) at (.5,1.25) {};
	\node[0c] (d) at (1.75,0) {};
	\draw[1cdot, out=-75, in=165] (a) to (b);
	\draw[1cdot, out=-15,in=105] (c) to (d);
	\draw[1cdash, out=0, in=-120] (b) to (d);
	\draw[1cdash, out=60,in=180] (a) to (c);
	\draw[1cdash] (b) to (c);
	\draw[2c] (.75,-.7) to node[auto,swap] {$U_1$} (.75,.9);
	\draw[2c] (-.75,-.9) to node[auto] {$U_2$} (-.75,.7);
	\node[scale=1.25] at (2.25,-1.25) {,};
\end{scope}
\end{tikzpicture}
\end{equation}
where $U_1$ and $U_2$ may be atoms, or may themselves split in one of the four ways. 

Constructible 2-atoms are precisely the shapes of cells of a regular 2-polygraph, and constructible 2-molecules  are the shapes of composable diagrams in a merge-bicategory, as defined in \cite{hadzihasanovic2018weak}. They are also in bijection with planar, connected string diagrams whose nodes have at least one input and one output edge \cite{joyal1991geometry}.
\end{exm}

\begin{cons}
Unravelling the inductive definition, if $U \subseteq P$ is a constructible $n$\nbd molecule with maximal elements $U^{(n)} = \{x_1,\ldots,x_m\}$, we obtain a rooted binary tree whose vertices are labelled with subsets $V \subseteq U$, such that
\begin{enumerate}
	\item the root is labelled $U$,
	\item the leaves are labelled with the subsets $\clos\{x_i\}$, $i = 1,\ldots,m$, and
	\item the children of a vertex labelled $V$ have labels $V_1, V_2$ with $V = V_1 \cup V_2$.
\end{enumerate}
The sets $V$ satisfy conditions that depend entirely on their boundaries:
\begin{enumerate}
	\item for each vertex with label $V$, $\bord{}{+}V$ and $\bord{}{-}V$ are constructible $(n-1)$\nbd molecules;
	\item if the children of a vertex with label $V$ have labels $V_1, V_2$, then $\bord{}{+}V_1 \cap \bord{}{-}V_2$ is a constructible submolecule of $\bord{}{+}V_1$ and of $\bord{}{-}V_2$, while $\bord{}{-}V_1$ and $\bord{}{+}V_2$ are constructible submolecules of $\bord{}{-}V$ and $\bord{}{+}V$, respectively.
\end{enumerate}
We call any such tree a \emph{merger tree} for the constructible molecule $U$. Intuitively, any branching in a merger tree corresponds to a well-formed pasting of two $n$\nbd cells along a shared $(n-1)$\nbd dimensional diagram in their boundary.

In general, a single constructible $n$\nbd molecule can have many merger trees. For example, the pasting diagram
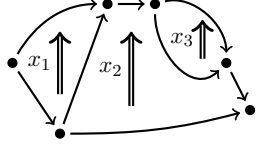
\begin{equation*}
\begin{tikzpicture}[baseline={([yshift=-.5ex]current bounding box.center)},scale=1.25]
	\node[0c] (0) at (-1.25, .125) {};
	\node[0c] (1) at (1.25, -.375) {};
	\node[0c] (o1) at (-.25, .75) {};
	\node[0c] (o2) at (.25, .75) {};
	\node[0c] (o3) at (1,.125) {};
	\node[0c] (i1) at (-.75, -.625) {};
	\draw[1c, out=60, in=180] (0) to (o1);
	\draw[1c] (o1) to (o2);
	\draw[1c, out=15, in=90] (o2) to (o3);
	\draw[1c] (o3) to (1);
	\draw[1c] (0) to (i1);
	\draw[1c, out=0, in=-165] (i1) to (1);
	\draw[1c] (i1) to (o1);
	\draw[1c, out=-90, in=-135, looseness=1.25] (o2) to (o3);
	\draw[2c] (-.75,-.375) to node[auto] {$x_1$} (-.75,.625);
	\draw[2c] (0,-.5) to node[auto] {$x_2$} (0,.625);
	\draw[2c] (.75,0) to node[auto] {$x_3$} (.75,.75);
\end{tikzpicture}
\end{equation*}
represents a constructible 2-molecule with two valid merger trees:
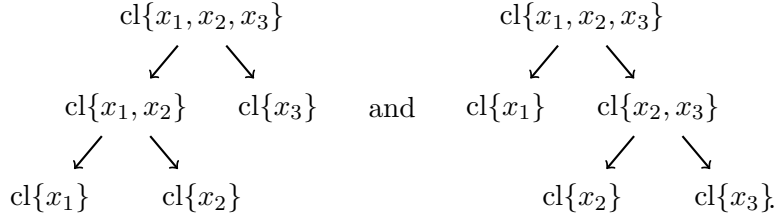
\begin{equation*}
\begin{tikzpicture}[baseline={([yshift=-.5ex]current bounding box.center)}, yscale=1.2]
\begin{scope}
	\node[scale=1.25] (0) at (0,2) {$\clos\{x_1, x_2, x_3\}$};
	\node[scale=1.25] (1) at (-1,1) {$\clos\{x_1, x_2\}$};
	\node[scale=1.25] (1b) at (1,1) {$\clos\{x_3\}$};
	\node[scale=1.25] (2) at (-2,0) {$\clos\{x_1\}$};
	\node[scale=1.25] (2b) at (0,0) {$\clos\{x_2\}$};
	\draw[1c] (0) to (1);
	\draw[1c] (0) to (1b);
	\draw[1c] (1) to (2);
	\draw[1c] (1) to (2b);
	\node[scale=1.25] at (2.5,1) {and};
\end{scope}
\begin{scope}[shift={(5,0)}, xscale=-1]
	\node[scale=1.25] (0) at (0,2) {$\clos\{x_1, x_2, x_3\}$};
	\node[scale=1.25] (1) at (-1,1) {$\clos\{x_2, x_3\}$};
	\node[scale=1.25] (1b) at (1,1) {$\clos\{x_1\}$};
	\node[scale=1.25] (2) at (-2,0) {$\clos\{x_3\}$};
	\node[scale=1.25] (2b) at (0,0) {$\clos\{x_2\}$};
	\draw[1c] (0) to (1);
	\draw[1c] (0) to (1b);
	\draw[1c] (1) to (2);
	\draw[1c] (1) to (2b);
	\node[scale=1.25] at (-2.5,-.1) {.};
\end{scope}
\end{tikzpicture}
\end{equation*}
\end{cons}

The following is immediate from the definition.

\begin{prop}
Let $V \subseteq U$ be two constructible $n$\nbd molecules. Then $V \sqsubseteq U$ if and only if there exists a merger tree for $U$ with a vertex labelled $V$.
\end{prop}

\begin{remark}
We do not know any examples of pairs $V \subseteq U$ of constructible $n$\nbd molecules such that $V$ is \emph{not} a constructible submolecule of $U$; thus, there remains the possibility that any subset of a constructible $n$\nbd molecule which is itself a constructible $n$\nbd molecule is, in fact, a constructible submolecule. However, we conjecture that this is not the case in general.
\end{remark}

\begin{conj} \label{conj:subglobe}
There exist constructible $n$\nbd molecules $U$ and $V$ such that $V \subseteq U$, but $V \not\sqsubseteq U$.
\end{conj}

\begin{lem} \label{lem:basic}
Let $U$ be a constructible $n$\nbd molecule, $n > 0$. Then:
\begin{enumerate}[label=(\alph*)]
	\item $\sbord{}{+}U$ and $\sbord{}{-}U$ are disjoint and both inhabited;
	\item each $x \in \sbord{}{}U$ is covered by a single element, and each $x \in U^{(n-1)}\setminus \sbord{}{}U$ is covered by exactly two elements with opposite orientations.
\end{enumerate}
\end{lem}
\begin{proof}
For the first point, observe that $\bord{}{\alpha}U$, being a constructible $(n-1)$\nbd molecule, is pure and $(n-1)$\nbd dimensional, hence it is equal to the closure of $\sbord{}{\alpha}U$, which is necessarily inhabited. Disjointness of $\sbord{}{+}U$ and $\sbord{}{-}U$ is obvious if $U$ is an atom; otherwise, $U$ splits as $U_1 \cup U_2$ with $U_1 \cap U_2 = \bord{}{+}U_1 \cap \bord{}{-}U_2 = \clos{(\sbord{}{+}U_1 \cap \sbord{}{-}U_2)}$. Then 
\begin{equation*}
	\sbord{}{-}U = \sbord{}{-}U_1 + \sbord{}{-}U_2\setminus \sbord{}{+}U_1, \quad \quad \quad \sbord{}{+}U = \sbord{}{+}U_2 + \sbord{}{+}U_1\setminus \sbord{}{-}U_2,
\end{equation*}
where + indicates disjoint union of sets, and we can derive the statement for $U$ from the statement for $U_1$ and $U_2$.

The second point is obvious if $U$ is an atom. Otherwise, suppose $U$ splits as $U_1 \cup U_2$; then
\begin{align*}
	\sbord{}{}U & = (\sbord{}{}U_1 \setminus \sbord{}{}U_2) + (\sbord{}{}U_2 \setminus \sbord{}{}U_1), \\
	U^{(n-1)}\setminus \sbord{}{}U & = (U_1^{(n-1)} \setminus \sbord{}{}U_1) + (U_2^{(n-1)} \setminus \sbord{}{}U_2) + (\sbord{}{+}U_1 \cap \sbord{}{-}U_2).
\end{align*}
By the inductive hypothesis applied to $U_1$, elements of $\sbord{}{}U_1 \setminus \sbord{}{}U_2$ are covered by a single element of $U_1$ and no elements of $U_2$, while elements of $U_1^{(n-1)} \setminus \sbord{}{}U_1$ are covered by two elements of $U_1$ with opposite orientations and no elements of $U_2$. We similarly deal with elements of $\sbord{}{}U_2 \setminus \sbord{}{}U_1$ and $U_2^{(n-1)} \setminus \sbord{}{}U_2$. Finally, elements of $\sbord{}{+}U_1 \cap \sbord{}{-}U_2$ are covered by a single element of $U_1$ with orientation $+$ and a single element of $U_2$ with orientation $-$. This completes the case distinction.
\end{proof}

\begin{lem} \label{lem:noautomorph}
Let $U$ be a constructible molecule and $\imath: U \hookrightarrow U$ an automorphism. Then $\imath$ is the identity on $U$.
\end{lem}
\begin{proof}
We proceed by induction on the dimension of $U$. If $U$ is a 0-globe, the statement obviously holds, so let $U$ be a constructible $n$\nbd molecule with $n > 0$. Inclusions preserve the dimension of elements; moreover, if $x$ is covered by a single element $y$ with orientation $\alpha$, then $\imath(x)$ is covered only by $\imath(y)$ with the same orientation: since $\imath$ is surjective, any $\tilde{y}$ covering $\imath(x)$ is the image of some $y'$ covering $x$. By Lemma \ref{lem:basic}, $\imath(\sbord{}{\alpha}U) \subseteq \sbord{}{\alpha}U$. Because $\imath$ is injective, in fact $\imath(\sbord{}{\alpha}U) = \sbord{}{\alpha}U$, and because it is also closed, $\imath(\bord{}{\alpha}U) = \bord{}{\alpha}U$. Thus $\imath$ restricts to an automorphism of the constructible $(n-1)$\nbd molecule $\bord{}{\alpha}U$: by the inductive hypothesis, it is the identity on $\bord{}{}U$.

By the same reasoning, if $\imath(x) = x$ for an $n$\nbd dimensional element $x$, then $\imath$ is also the identity on $\bord{}{}x$, that is, $\imath$ is the identity on $\clos\{x\}$. Therefore, to prove that $\imath$ is the identity on $U$, it suffices to prove that it fixes all the $n$\nbd dimensional elements.

For all $n$\nbd dimensional $x \in U$, there exists a path $x = x_0 \to y_0 \to \ldots \to x_m \to y_m$ in $\hasso{U}$, where the $x_i$ are $n$\nbd dimensional, the $y_i$ are $(n-1)$\nbd dimensional, and $y_m \in \sbord{}{+}U$: given $x_i$, we can always pick some $y_i \in \sbord{}{+}x_i$; if $y_i \in \sbord{}{+}U$, we stop, otherwise we take $x_{i+1}$ such that $y_i \in \sbord{}{-}x_{i+1}$. Any such path is mapped by $\imath$ to a path $\imath(x) \to \imath(y_0) \to \ldots \to \imath(x_m) \to \imath(y_m) = y_m$; but $y_m$ is only covered by $x_m$ in $U$, so $\imath(x_m) = x_m$. It follows that $\imath$ is the identity on $\bord{}{}x_m$, hence also $\imath(y_{m-1}) = y_{m-1}$; but $y_{m-1}$ is only covered by $x_m$ and by $x_{m-1}$ in $U$, so $\imath(x_{m-1}) = x_{m-1}$. Proceeding backwards in this way, we find that $\imath(x) = x$, and we conclude.
\end{proof}

\begin{remark}
It follows that two constructible molecules can only be isomorphic in a single way, and we can speak of constructible molecules ``being isomorphic'' without having to indicate a specific isomorphism.
\end{remark}

\begin{dfn} \label{dfn:globpos}
A \emph{constructible directed complex} is an oriented thin poset $P$ such that, for all $x \in P$, $\clos{\{x\}}$ is a constructible atom.

We write $\globpos$ for the full subcategory of $\ogpos$ on constructible directed complexes. 
\end{dfn}

\begin{prop} \label{prop:pushouts_exist}
The category $\globpos$ has an initial object and pushouts, created by the forgetful functor to $\cat{Set}$.
\end{prop}

\begin{remark}
It follows that $\globpos$ also has all finite coproducts. 
\end{remark}
\begin{proof}
The empty constructible directed complex $\emptyset$ is clearly initial.

Let $\imath_1: Q \incl P_1, \imath_2: Q \incl P_2$ be a span of inclusions. We let $P_1 \cup_Q P_2$ be the pushout of the underlying span of sets, that is, the quotient of the disjoint union $P_1 + P_2$ of sets by the relation $\imath_1(x) \sim \imath_2(x)$ for all $x \in Q$. This comes with injective functions $j_1: P_1 \incl P_1 \cup_Q P_2$ and $j_2: P_2 \incl P_1 \cup_Q P_2$, and becomes a poset by letting $j_i(x) \leq j_i(y)$ if and only if $x \leq y$ in $P_i$. Now $\clos\{j_i(x)\} \simeq \clos\{x\}$ for all $x \in J_i$, and every element of $P_1 \cup_Q P_2$ is of the form $j_i(x)$ for some $i$ and $x \in J_i$, so $P_1 \cup_Q P_2$ is graded, and inherits an orientation from $P_1$ and $P_2$, compatibly on $Q$. For the same reason, with this orientation $P_1 \cup_Q P_2$ is a constructible directed complex. The universal property is easy to check.
\end{proof}

\begin{cor} \label{cor:globpos_is_colimit}
Any constructible directed complex is the colimit of the diagram of inclusions of its atoms.
\end{cor}
\begin{proof}
This is true of the underlying sets and functions, and the colimit can be constructed by pushouts and finite coproducts.
\end{proof}

\begin{exm}
It is immediate from the definition that any closed subset of a constructible directed complex, with the order and orientation obtained by restriction, is also a constructible directed complex. 

Any constructible $n$\nbd molecule $U$ in an oriented thin poset is a constructible directed complex. This is obvious for $n = 0$; suppose $n > 0$, and $x \in U$. If $\dmn{x} = n$, $\clos{\{x\}}$ is required by definition to be a constructible $n$\nbd atom. If $\dmn{x} < n$, then $x$ belongs to the constructible $(n-1)$\nbd molecule $\bord{}{\alpha}\clos{\{y\}}$ for some $y$, and we can apply the inductive hypothesis to conclude that $\clos{\{x\}}$ is a constructible atom.
\end{exm}

\begin{exm}
The following is an oriented thin poset which is not a constructible directed complex:
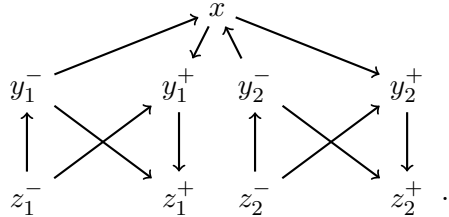
\begin{equation*}
\begin{tikzpicture}[baseline={([yshift=-.5ex]current bounding box.center)}]
	\node[scale=1.25] (0) at (0,2.5) {$x$};
	\node[scale=1.25] (1) at (-2.5,1.5) {$y_1^-$};
	\node[scale=1.25] (1b) at (-.5,1.5) {$y_1^+$};
	\node[scale=1.25] (1c) at (.5,1.5) {$y_2^-$};
	\node[scale=1.25] (1d) at (2.5,1.5) {$y_2^+$};
	\node[scale=1.25] (2) at (-2.5,0) {$z_1^-$};
	\node[scale=1.25] (2b) at (-.5,0) {$z_1^+$};
	\node[scale=1.25] (2c) at (.5,0) {$z_2^-$};
	\node[scale=1.25] (2d) at (2.5,0) {$z_2^+$};
	\draw[1clong] (1) to (0);
	\draw[1clong] (0) to (1b);
	\draw[1clong] (1c) to (0);
	\draw[1clong] (0) to (1d);
	\draw[1clong] (2) to (1);
	\draw[1clong] (2) to (1b);
	\draw[1clong] (1) to (2b);
	\draw[1clong] (1b) to (2b);
	\draw[1clong] (2c) to (1c);
	\draw[1clong] (2c) to (1d);
	\draw[1clong] (1c) to (2d);
	\draw[1clong] (1d) to (2d);
	\node[scale=1.25] at (3,0) {$.$};
\end{tikzpicture}
\end{equation*}
This is because neither $\clos{\{y_1^-, y_2^-\}}$ nor $\clos{\{y_1^+,y_2^+\}}$ are constructible 1-molecules, so $\clos{\{x\}}$ is not a constructible 2-molecule.
\end{exm}

We briefly consider some restricted classes of constructible molecules that may be of independent interest.

\begin{dfn}
A constructible $n$\nbd molecule $U$ is \emph{simple} if $n = 0$, or if $n > 0$, $\bord{}{+}U$ and $\bord{}{-}U$ are simple constructible $(n-1)$\nbd molecules, and either $U$ is an atom, or $U$ splits into simple constructible submolecules $U_1, U_2$ such that $U_1 \cap U_2 = \bord{}{+}U_1 \cap \bord{}{-}U_2$ is an atom.
\end{dfn}

\begin{exm}
The 0-globe and all constructible 1-molecules are simple; consequently, all constructible 2-atoms are simple. However, a general constructible 2-molecule $U$ is only simple if the intersection of any pair of submolecules is, at most, an atom, that is, if $U$ splits as
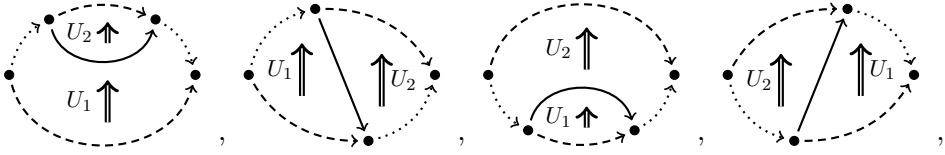
\begin{equation*}
\begin{tikzpicture}[baseline={([yshift=-.5ex]current bounding box.center)}, scale=0.7]
\begin{scope}
	\node[0c] (0) at (-1.75,0) {};
	\node[0c] (a1) at (-1,1.05) {};
	\node[0c] (a2) at (1,1.05) {};
	\node[0c] (1) at (1.75,0) {};
	\draw[1cdot, out=75, in=-150] (0) to (a1);  
	\draw[1cdash, out=25,in=155] (a1) to (a2);
	\draw[1cdot, out=-30, in=105] (a2) to (1);
	\draw[1cdash, out=-75,in=-105, looseness=1.25] (0) to (1);
	\draw[1c, out=-75, in=-105, looseness=1.25] (a1) to (a2);
	\draw[2c] (.1,-1.2) to node[auto] {$U_1\;$} (.1,.2);
	\draw[2c] (.1,.3) to node[auto] {$U_2\;$} (.1,1.3);
	\node[scale=1.25] at (2.25,-1.25) {,};
\end{scope}
\begin{scope}[shift={(4.5,0)}]
	\node[0c] (a) at (-1.75,0) {};
	\node[0c] (b) at (.5,-1.25) {};
	\node[0c] (c) at (-.5,1.25) {};
	\node[0c] (d) at (1.75,0) {};
	\draw[1cdash, out=-60, in=180] (a) to (b);
	\draw[1cdash, out=0, in=120] (c) to (d);
	\draw[1cdot, out=15, in=-105] (b) to (d);
	\draw[1cdot, out=75, in=-165] (a) to (c);
	\draw[1c] (c) to (b);
	\draw[2c] (-.75,-.7) to node[auto] {$U_1$} (-.75,.9);
	\draw[2c] (.75,-.9) to node[auto,swap] {$U_2$} (.75,.7);
	\node[scale=1.25] at (2.25,-1.25) {,};
\end{scope}
\begin{scope}[shift={(9,0)}]
	\node[0c] (0) at (-1.75,0) {};
	\node[0c] (a1) at (-1,-1.05) {};
	\node[0c] (a2) at (1,-1.05) {};
	\node[0c] (1) at (1.75,0) {};
	\draw[1cdot, out=-75, in=150] (0) to (a1);  
	\draw[1cdash, out=-25,in=-155] (a1) to (a2);
	\draw[1cdot, out=30, in=-105] (a2) to (1);
	\draw[1cdash, out=75,in=105, looseness=1.25] (0) to (1);
	\draw[1c, out=75, in=105, looseness=1.25] (a1) to (a2);
	\draw[2c] (.1,-.2) to node[auto] {$U_2\;$} (.1,1.2);
	\draw[2c] (.1,-1.3) to node[auto] {$U_1\;$} (.1,-.3);
	\node[scale=1.25] at (2.25,-1.25) {,};
\end{scope}
\begin{scope}[shift={(13.5,0)}]
	\node[0c] (a) at (-1.75,0) {};
	\node[0c] (b) at (-.5,-1.25) {};
	\node[0c] (c) at (.5,1.25) {};
	\node[0c] (d) at (1.75,0) {};
	\draw[1cdot, out=-75, in=165] (a) to (b);
	\draw[1cdot, out=-15,in=105] (c) to (d);
	\draw[1cdash, out=0, in=-120] (b) to (d);
	\draw[1cdash, out=60,in=180] (a) to (c);
	\draw[1c] (b) to (c);
	\draw[2c] (.75,-.7) to node[auto,swap] {$U_1$} (.75,.9);
	\draw[2c] (-.75,-.9) to node[auto] {$U_2$} (-.75,.7);
	\node[scale=1.25] at (2.25,-1.25) {,};
\end{scope}
\end{tikzpicture}
\end{equation*}
where $U_1$ and $U_2$ are both simple.

The simple constructible 2-molecules are precisely the shapes of composable diagrams in regular poly-bicategories \cite{hadzihasanovic2018weak}, a variant of poly-bicategories \cite{cockett2003morphisms}, which themselves are the version with many 0-cells of planar polycategories \cite{cockett1997weakly}. The restriction to a single shared 1-dimensional element models the restriction to a single shared formula in the cut rule of sequent calculus. 

Among the planar string diagrams that correspond to constructible 2-molecules, the simple ones are those that are loop-free, or contractible as graphs.
\end{exm}

\begin{dfn} \label{dfn:pope}
A constructible $n$\nbd molecule $U$ is a \emph{positive opetope} if $n=0$, or if $n > 0$ and, for all $x \in U$ with $\dmn{x} > 0$, $\bord{}{+}x$ is an atom.
\end{dfn}

The name ``positive opetope'' has been given by Zawadowski \cite{zawadowski2017positive} to a restricted class of opetopes \cite{baez1998higher, hermida2000weak}. Our results in Section \ref{sec:polygraphs} imply that presheaves over positive opetopes in our sense are positive opetopic sets in Zawadowski's sense. We have not attempted a formal proof, but we are confident that the two notions of positive opetope coincide.

\begin{remark}
If an opetope is not a positive opetope, then it is not a constructible molecule: non-positivity implies that the input boundary of an $n$\nbd dimensional element is $k$\nbd dimensional, with $k < n-1$. 
\end{remark}

\begin{lem} \label{lem:popeboundary}
Let $U$ be a positive opetope. Then $\bord{}{+}U$ is an atom.
\end{lem}
\begin{proof}
By induction on the number of maximal elements of $U$: if $U$ is an atom, the statement is true by definition. Otherwise, suppose $U$ splits as $U_1 \cup U_2$: then $U_1 \cap U_2 \sqsubseteq \bord{}{+}U_1$, and by the inductive hypothesis $\bord{}{+}U_1$ is an atom, so $U_1 \cap U_2 = \bord{}{+}U_1$. It follows that $\bord{}{+}U = \bord{}{+}U_2$ is an atom.
\end{proof}

\begin{prop}
All positive opetopes are simple.
\end{prop}
\begin{proof}
Let $U$ be a positive opetope; we proceed by induction on the dimension $n$ of $U$. If $n = 0$, there is nothing to prove, so suppose $n > 0$. Clearly, if $V \subseteq U$ is a constructible molecule, then it is also a positive opetope. By the inductive hypothesis, $\bord{}{+}U$ and $\bord{}{-}U$ are simple constructible molecules. If $U$ is an atom, there is nothing else to prove; otherwise, $U$ splits as $U_1 \cup U_2$, and by Lemma \ref{lem:popeboundary} $U_1 \cap U_2 \sqsubseteq \bord{}{+}U_1$ implies $U_1 \cap U_2 = \bord{}{+}U_1$, an atom.
\end{proof}

\begin{exm}
The 0-globe and all constructible 1-molecules are positive opetopes. The atomic 2-dimensional positive opetopes are classified by a single non-zero natural number, which we identify with the planar rooted tree with a root and $n$ leaves:
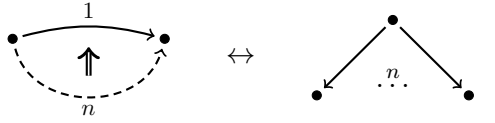
\begin{equation*}
\begin{tikzpicture}
\begin{scope}
	\node[0c] (0) at (-1,.5) {};
	\node[0c] (1) at (1,.5) {};
	\draw[1c, out=15, in=165] (0) to node[auto] {$1$} (1);  
	\draw[1cdash, out=-75, in=-105, looseness=1.25] (0) to node[auto,swap] {$n$} (1);
	\draw[2c] (0,-.2) to (0,0.6);
	\node[scale=1.25] at (2,0.25) {$\leftrightarrow$};
\end{scope}
\begin{scope}[shift={(4,0)}]
	\node[0c] (0) at (0,.75) {};
	\node[0c] (1) at (-1,-.25) {};
	\node[0c] (n) at (1,-.25) {};
	\draw[1c] (0) to (1);
	\draw[1c] (0) to (n);
	\node[scale=1.25] at (0,0) {$\stackrel{n}{\ldots}$};
	\node[scale=1.25] at (1.375,-.25) {.};
\end{scope}
\end{tikzpicture}
\end{equation*}
A non-atomic 2-dimensional positive opetope splits as
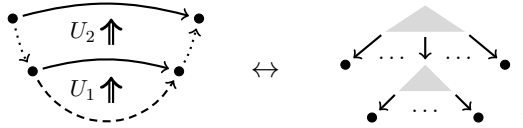
\begin{equation*}
\begin{tikzpicture}[baseline={([yshift=-.5ex]current bounding box.center)}, scale=0.7]
\begin{scope}
	\node[0c] (0) at (-1.75,.5) {};
	\node[0c] (a1) at (-1.375,-.5) {};
	\node[0c] (a2) at (1.375,-.5) {};
	\node[0c] (1) at (1.75,.5) {};
	\draw[1cdot, out=-75, in=120] (0) to (a1);  
	\draw[1cdash, out=-60,in=-120, looseness=1.25] (a1) to (a2);
	\draw[1cdot, out=60, in=-105] (a2) to (1);
	\draw[1c, out=15,in=165] (0) to (1);
	\draw[1c, out=15, in=165] (a1) to (a2);
	\draw[2c] (.1,-.3) to node[auto] {$U_2\;$} (.1,.8);
	\draw[2c] (.1,-1.4) to node[auto] {$U_1\;$} (.1,-.3);
	\node[scale=1.25] at (3,-.5) {$\leftrightarrow$};
\end{scope}
\begin{scope}[shift={(6,0)}]
	\node[0c] (1) at (-1.5,-.375) {};
	\node[0c] (n) at (1.5,-.375) {};
	\node[0c] (1b) at (-1,-1.375) {};
	\node[0c] (nb) at (1,-1.375) {};
	\draw[1c] (-.75,.25) to (1);
	\draw[1c] (.75,.25) to (n);
	\draw[1c] (0,.25) to (0,-.375);
	\draw[1c] (-.5,-.875) to (1b);
	\draw[1c] (.5,-.875) to (nb);
	\path[fill=gray!30] (-.75,.25) to (.75,.25) to (0,.75) to (-.75,.25) ;
	\path[fill=gray!30] (-.5,-.875) to (.5,-.875) to (0,-.375) to (-.5,-.875);
	\node at (-.625,-.25) {$\ldots$};
	\node at (.625,-.25) {$\ldots$};
	\node at (0,-1.25) {$\ldots$};
	\node[scale=1.25] at (1.875,-1.375) {,};
\end{scope}
\end{tikzpicture}
\end{equation*}
where $U_1$ and $U_2$ are both positive opetopes. We see that 2-dimensional positive opetopes are in bijection with finite planar rooted trees with at least one edge, and they are precisely the shapes of composable diagrams in a regular multi-bicategory \cite{hadzihasanovic2018weak}.
\end{exm}

\begin{remark}
General opetopes have been classified in \cite{kock2010polynomial} in terms of combinatorial objects called zoom complexes. If our positive opetopes coincide with Zawadowski's, they should be classified by a sub-class of zoom complexes, where each ``dot'' in a tree, as defined there, has at least one input edge.
\end{remark}

We define another subclass of constructible molecules, which includes the positive opetopes, and will be of technical interest in Section \ref{sec:polygraphs}.
\begin{dfn} \label{dfn:flowconnect}
Let $U$ be a constructible $n$\nbd molecule. We say that $U$ is \emph{flow-connected} if, for all $x \in \sbord{}{-}U$ and $x' \in \sbord{}{+}U$, there is a path from $x$ to $x'$ in $\hasso{U}$ passing only through $n$\nbd dimensional and $(n-1)$\nbd dimensional elements.
\end{dfn}

\begin{exm}
Clearly, all atoms are flow-connected, but not all constructible molecules are: for example, in (\ref{eq:fourways}), if $U_1$ and $U_2$ are flow-connected, then $U$ is in general flow-connected only in the first and third setup.
\end{exm}

\begin{prop} \label{prop:pope-flowconn}
All positive opetopes are flow-connected.
\end{prop}
\begin{proof}
Let $U$ be a positive opetope. If $U$ is an atom, the statement is obvious. Otherwise, $U$ splits as $U_1 \cup U_2$, where $U_1$ and $U_2$ are positive opetopes, $\sbord{}{+}U_1 \cap \sbord{}{-}U_2$ consists of a single element $y$, $\sbord{}{-}U_1 \subseteq \sbord{}{-}U$, and $\sbord{}{+}U = \sbord{}{+}U_2$. 

Given $x \in \sbord{}{-}U$ and $x' \in \sbord{}{+}U$, if $x \in \sbord{}{-}U_2$, we use the inductive hypothesis on $U_2$ directly. If $x \in \sbord{}{-}U_2$, by the inductive hypothesis there are a path $x \to \ldots \to y$ in $\hasso{U_1}$ and a path $y \to \ldots \to x'$ in $\hasso{U_2}$, which join to form a path $x \to \ldots \to x'$ in $\hasso{U}$. 
\end{proof}

Let $\globe$ be a skeleton of the full subcategory of $\globpos$ whose objects are the atomic globes of all dimensions.

\begin{dfn}
A \emph{constructible polygraph} $X$ is a presheaf $X: \opp{\globe} \to \cat{Set}$. For any constructible atom $U$ and $x \in X(U)$, we call $x$ a \emph{cell} of $X$ of \emph{shape} $U$. Constructible polygraphs and their morphisms of presheaves form a category $\cpol$.
\end{dfn}

We have the Yoneda embedding $\globe \incl \cpol$. By Corollary \ref{cor:globpos_is_colimit}, and the universal property of $\cpol$ as a free cocompletion, the left Kan extension of the Yoneda embedding along $\globe \incl \globpos$ gives a full and faithful functor $\globpos \incl \cpol$; in this sense, constructible polygraphs extend constructible directed complexes. 

\begin{cons} \label{cons:nskeleton}
For $n > 0$, let $\globe_n$ be the full subcategory of $\globe$ whose objects are the constructible atom of dimension $k \leq n$. We call a presheaf on $\globe_n$ a \emph{constructible $n$\nbd polygraph}, and write $n\cpol$ for their category.

The restrictions $-_{\leq n}: \cpol \to n\cpol$ have full and faithful left adjoints,
\begin{equation*}
	\imath_n: n\cpol \to \cpol.
\end{equation*}
If $X$ is a constructible polygraph, we write $\skel{n}{X} := \imath_nX_{\leq n}$, and call the counit $\skel{n}{X} \to X$ the \emph{$n$\nbd skeleton} of $X$. 

The restrictions factor as a sequence of restriction functors 
\begin{equation*}
	\ldots \to n\cpol \to (n-1)\cpol \to \ldots \to 0\cpol,
\end{equation*}
all with full and faithful left adjoints. By universal properties, the skeleta of $X$ form a sequence of morphisms
\begin{equation*}
	\skel{0}{X} \to \ldots \to \skel{n-1}{X} \to \skel{n}{X} \to \ldots
\end{equation*}
over $X$, whose colimit is $X$. 
\end{cons}

Let $\atom{n}{X}$ be the set of $n$\nbd cells of $X$, and for $x \in \atom{n}{X}$, let $U(x)$ be the shape of $x$. We have the following result, whose proof is straightforward.
\begin{prop} \label{prop:polygraph_ext} 
For all $n > 0$, the diagram
\begin{equation*}
\begin{tikzpicture}[baseline={([yshift=-.5ex]current bounding box.center)}]
	\node[scale=1.25] (0) at (0,2) {$\displaystyle \coprod_{x\in\atom{n}{X}} \bord{}{} U(x)$};
	\node[scale=1.25] (1) at (3.5,2) {$\displaystyle\coprod_{x\in\atom{n}{X}} U(x)$};
	\node[scale=1.25] (2) at (0,0) {$\skel{n-1}{X}$};
	\node[scale=1.25] (3) at (3.5,0) {$\skel{n}{X}$};
	\draw[1c] (0) to node[auto,swap] {$(\bord{}{}x)_{x\in\atom{n}{X}}$} (2);
	\draw[1c] (1) to node[auto] {$(x)_{x\in\atom{n}{X}}$} (3);
	\draw[1cinc] (0) to (1);
	\draw[1cinc] (2) to (3);
	\draw[edge] (2.5,0.2) to (2.5,0.8) to (3.3,0.8);
\end{tikzpicture}
\end{equation*}
is a pushout in $\cpol$.
\end{prop}

\section{Substitutions, mergers, and globularity} \label{sec:properties}

In this section, we develop some basic tools for manipulating constructible directed complexes. First of all, we study substitutions of constructible submolecules: these give a concrete way of performing ``surgery'' of cells, which in the theory of polygraphs is usually done by syntactic means. 

We then introduce \emph{mergers}, a natural notion of composition for cells in a constructible directed complex or constructible polygraph: the substitution of an $n$\nbd atom for a pair of $n$\nbd atoms intersecting at a constructible $(n-1)$\nbd molecule. We use mergers to prove that the constructible molecules in a constructible directed complex form an $\omega$\nbd graph: that is, if $U$ is a constructible molecule, then $\bord{}{\alpha}(\bord{}{+} U) = \bord{}{\alpha}(\bord{}{-} U)$ (\emph{globularity} condition). This is necessary for constructible molecules to be interpreted as cells in a higher category. 

Finally, we explore some consequences of globularity, and give a direct characterisation of the $k$\nbd dimensional boundary of a constructible $n$\nbd molecule, for each $k < n$.

\begin{cons} \label{cons:substitution}
Let $U$ be a constructible $n$\nbd molecule, $V \sqsubseteq U$ a constructible submolecule, and $W$ another constructible $n$\nbd molecule such that there is an isomorphism $\imath: \bord{}{}W \hookrightarrow \bord{}{}V$ restricting to isomorphisms $\imath^-: \bord{}{-}W \hookrightarrow \bord{}{-}V$ and $\imath^+: \bord{}{+}W \hookrightarrow \bord{}{+}V$ of constructible $(n-1)$\nbd molecules. 

Then, we define $U[W/V]$ to be the result of replacing $V$ with $W$ in $U$, identifying $\bord{}{}V$ and $\bord{}{}W$ through $\imath$: that is, $U[W/V]$ is the set $(U\setminus V)+ W$, with $x \leq y$ if and only if
\begin{itemize}
	\item $x, y \in U\setminus V$ and $x \leq y$ in $U$, or $x, y \in W$ and $x \leq y$ in $W$, or
	\item $x \in U\setminus V$, $y \in W$, and there exists $z \in \bord{}{}W$ such that $x \leq \imath(z)$ in $U$ and $z \leq y$ in $W$, or
	\item $x \in W$, $y \in U\setminus V$, and there exists $z \in \bord{}{}W$ such that $x \leq z$ in $W$ and $\imath(z) \leq y$ in $U$.
\end{itemize}
Then $U[W/V]$ is still a pure graded poset, and inherits an orientation from those of $U$ and $W$. Notice that $\bord{}{\alpha}U[W/V]$ is isomorphic to $\bord{}{\alpha}U$.

We call this a \emph{substitution} of $W$ for the constructible submolecule $V$ of $U$.
\end{cons}

\begin{remark}
By Lemma \ref{lem:noautomorph}, there can be at most one isomorphism between any pair of constructible molecules, so the notation $U[W/V]$ is unambiguous even though it does not specify an isomorphism.
\end{remark}

\begin{exm}
The following depicts a substitution on a constructible 2-molecule:
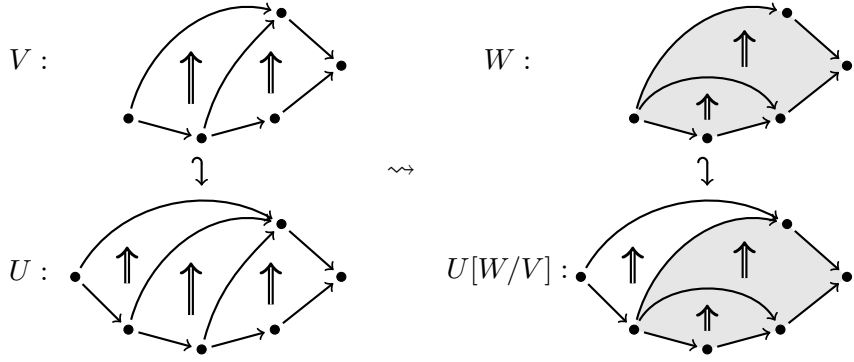
\begin{equation*}
\begin{tikzpicture}[baseline={([yshift=-.5ex]current bounding box.center)},scale=.7]
\begin{scope}[shift={(0,4)}]
	\node[0c] (b) at (-1.375,-1.375) {};
	\node[0c] (c) at (0,-1.75) {};
	\node[0c] (d) at (1.375,-1.375) {};
	\node[0c] (e) at (2.625,-.375) {};
	\node[0c] (x) at (1.5,.625) {};
	\draw[1c] (b) to (c);
	\draw[1c] (c) to (d);
	\draw[1c] (d) to (e);
	\draw[1c] (x) to (e);
	\draw[1c, out=75, in=165] (b) to (x);
	\draw[1c, out=75, in=-135] (c) to (x);
	\draw[2c] (-.2,-1.4) to (-.2,.2);
	\draw[2c] (1.3,-1.2) to (1.3,.2);
	\node[scale=1.25] at (-3.25,-.25) {$V :$};
	\draw[1cincl] (0,-2) to (0,-2.75);
	\node[scale=1.25] at (3.75,-2.375) {$\leadsto$};
\end{scope}
\begin{scope}
	\node[0c] (a) at (-2.375,-.375) {};
	\node[0c] (b) at (-1.375,-1.375) {};
	\node[0c] (c) at (0,-1.75) {};
	\node[0c] (d) at (1.375,-1.375) {};
	\node[0c] (e) at (2.625,-.375) {};
	\node[0c] (x) at (1.5,.625) {};
	\draw[1c] (a) to (b);
	\draw[1c] (b) to (c);
	\draw[1c] (c) to (d);
	\draw[1c] (d) to (e);
	\draw[1c, out=60, in=150] (a) to (x);
	\draw[1c] (x) to (e);
	\draw[1c, out=75, in=165] (b) to (x);
	\draw[1c, out=75, in=-135] (c) to (x);
	\draw[2c] (-.2,-1.4) to (-.2,.2);
	\draw[2c] (-1.4,-.8) to (-1.4,.5);
	\draw[2c] (1.3,-1.2) to (1.3,.2);
	\node[scale=1.25] at (-3.25,-.25) {$U :$};
\end{scope}
\begin{scope}[shift={(9.5,4)}]
	\path[fill, color=gray!20] (-1.375,-1.375) to [out=75,in=165,looseness=1.1] (1.5,.625) to (2.625,-.375) to (1.375,-1.375) to (0,-1.75) to (-1.375,-1.375);
	\node[0c] (b) at (-1.375,-1.375) {};
	\node[0c] (c) at (0,-1.75) {};
	\node[0c] (d) at (1.375,-1.375) {};
	\node[0c] (e) at (2.625,-.375) {};
	\node[0c] (x) at (1.5,.625) {};
	\draw[1c] (b) to (c);
	\draw[1c] (c) to (d);
	\draw[1c] (d) to (e);
	\draw[1c] (x) to (e);
	\draw[1c, out=75, in=165] (b) to (x);
	\draw[1c, out=60, in=120] (b) to (d);
	\draw[2c] (0,-1.7) to (-0,-.6);
	\draw[2c] (.7,-.7) to (.7,.6);
	\node[scale=1.25] at (-3.75,-.25) {$W :$};
	\draw[1cincl] (0,-2) to (0,-2.75);
\end{scope}
\begin{scope}[shift={(9.5,0)}]
	\path[fill, color=gray!20] (-1.375,-1.375) to [out=75,in=165,looseness=1.1] (1.5,.625) to (2.625,-.375) to (1.375,-1.375) to (0,-1.75) to (-1.375,-1.375);
	\node[0c] (a) at (-2.375,-.375) {};
	\node[0c] (b) at (-1.375,-1.375) {};
	\node[0c] (c) at (0,-1.75) {};
	\node[0c] (d) at (1.375,-1.375) {};
	\node[0c] (e) at (2.625,-.375) {};
	\node[0c] (x) at (1.5,.625) {};
	\draw[1c] (a) to (b);
	\draw[1c] (b) to (c);
	\draw[1c] (c) to (d);
	\draw[1c] (d) to (e);
	\draw[1c, out=60, in=150] (a) to (x);
	\draw[1c] (x) to (e);
	\draw[1c, out=75, in=165] (b) to (x);
	\draw[1c, out=60, in=120] (b) to (d);
	\draw[2c] (0,-1.7) to (-0,-.6);
	\draw[2c] (-1.4,-.8) to (-1.4,.5);
	\draw[2c] (.7,-.7) to (.7,.6);
	\node[scale=1.25] at (-3.75,-.25) {$U[W/V] :$};
	\node[scale=1.25] at (3.25,-1.25) {.};
\end{scope}
\end{tikzpicture}
\end{equation*}
\end{exm}

\begin{lem}[Substitution] \label{lem:substitution}
$U[W/V]$ is a constructible $n$\nbd molecule. If $V \sqsubseteq V' \sqsubseteq U$, then $W \sqsubseteq V'[W/V] \sqsubseteq U[W/V]$.
\end{lem}
\begin{proof}
Let $[x,y]$ be an interval in $U[W/V]$. If $y \in W$, the interval is entirely contained in $W$; if $y \in U \setminus V$, the interval is entirely contained in the closed subset $(U \setminus V) + \bord{}{}W$, which is isomorphic to $U \setminus (V \setminus \bord{}{}V)$, a closed subset of $U$. Oriented thinness of $U[W/V]$ follows from the oriented thinness of $W$ and of $U \setminus (V \setminus \bord{}{}V)$.

Next, we proceed by induction on increasing $V'$ with $V \sqsubseteq V' \sqsubseteq U$. If $V' = V$, then $V'[W/V] = W$. Suppose that $V'$ splits as $V'_1 \cup V'_2$, with $V \sqsubseteq V'_i$; without loss of generality, suppose $i = 1$. Then $V'[W/V] = V'_1[W/V] \cup V'_2$, where $V'_1[W/V]$ is a constructible molecule by the inductive hypothesis and $V'_2$ by assumption. Moreover, $\bord{}{\alpha}V'_1[W/V] = \bord{}{\alpha}V'_1$, so the constructible submolecule conditions satisfied by 
$V'_1 \cup V'_2$ are also satisfied by $V'_1[W/V] \cup V'_2$. 

This proves that $V'[W/V]$ is a constructible molecule with $W \sqsubseteq V'[W/V]$. The statement follows from the fact that any chain $V \sqsubset \ldots \sqsubset U$ is finite.
\end{proof}

\begin{cons} \label{cons:mergers}
Let $P$ be an oriented thin poset, and $U \subseteq P$ a constructible $n$\nbd molecule with the following properties:
\begin{enumerate}
	\item $U$ has exactly two maximal elements, $x_1$ and $x_2$, and
	\item $\clos{\{x_1\}} \cap \clos{\{x_2\}} = \bord{}{+}x_1 \cap \bord{}{-}x_2 = \clos{\{y\}}$ for an $(n-1)$\nbd dimensional element $y$, which is only covered by $x_1$ and $x_2$ in $P$.
\end{enumerate}
Let $P'$ be the poset obtained from $P$ by removing $y$, and identifying $x_1$ and $x_2$; we write $\tilde{x}$ for the result of the identification. We write $m_U: P \leadsto P'$ for the partial function that is undefined on $y$ and sends $x_1$ and $x_2$ to $\tilde{x}$. We say that $m_U$ is a \emph{simple merger} from $P$ to $P'$, or that $m_U$ \emph{merges} $x_1$ and $x_2$.

First of all, $P'$ is graded. Let $z \in P'$, and consider the set of all paths from $z$ to $\bot$ in $\hass{P'}_\bot$. Suppose first $z = \tilde{x}$: then any path from $z$ to $\bot$ corresponds to a path from $x_1$ or $x_2$ to $\bot$ in $\hass{P}_\bot$ that does not pass through $y$. By Lemma \ref{lem:basic}, both $x_1$ and $x_2$ cover at least another element besides $y$, so there is at least one such path, and they are all of the same length $n+1$.

Suppose $z \neq \tilde{x}$, so $z$ corresponds to a unique element of $P$.
\begin{itemize}
	\item If $z < y$, or if $z$ and $y$ are incomparable in $P$, then any path from $z$ to $\bot$ in $\hass{P'}_\bot$ corresponds to a path of the same length $\dmn{z} + 1$ in $\hass{P}_\bot$.
	\item If $y < z$ in $P$, because $y$ is only covered by $x_1$ and $x_2$ in $P$, then also $x_1 \leq z$ or $x_2 \leq z$ in $P$, hence $\tilde{x} \leq z$ in $P'$. Therefore paths from $z$ to $\bot$ in $\hass{P'}_\bot$ either do not pass through $\tilde{x}$, in which case they correspond to paths of the same length $\dmn{z} + 1$ in $\hass{P}_\bot$, or they pass through $\tilde{x}$, in which case they correspond to paths from $z$ to $x_i$ in $\hass{P}_\bot$ (of length $\dmn{z} - n$), followed by paths from $x_j$ to $\bot$ in $\hass{P}_\bot$ (of length $n + 1$), for some $i, j \in \{1,2\}$.
\end{itemize}
In each case, $z$ has a well-defined dimension in $P'$, the same it had in $P$.

Next, we define an orientation $o'$ on $P'$ as follows: $o'(c_{w,z}) := o(c_{w,z})$ if $w, z \neq \tilde{x}$; $o'(c_{w,\tilde{x}}) := o(c_{w,x_i})$ if $w$ covers $x_i$ in $P$, and similarly $o'(c_{\tilde{x},z}) := o(c_{x_i,z})$ if $x_i$ covers $z$ in $P$, for some $i \in \{1,2\}$. Only $y$ is covered both by $x_1$ and by $x_2$ in $P$, so $o'(c_{\tilde{x},z})$ is well-defined for all $z \in P'$. Suppose that $w$ covers both $x_1$ and $x_2$ in $P$; then, by the oriented thinness of $P$, we have
\begin{equation} \label{eq:merger-nplus1}
\begin{tikzpicture}[baseline={([yshift=-.5ex]current bounding box.center)}]
	\node[scale=1.25] (0) at (0,2) {$w$};
	\node[scale=1.25] (1) at (-1,1) {$x_1$};
	\node[scale=1.25] (1b) at (1,1) {$x_2$};
	\node[scale=1.25] (2) at (0,0) {$y$};
	\draw[1c] (0) to node[auto,swap] {$\alpha_1$} (1);
	\draw[1c] (0) to node[auto] {$\alpha_2$} (1b);
	\draw[1c] (1) to node[auto,swap] {$+$} (2);
	\draw[1c] (1b) to node[auto] {$-$} (2);
	\node[scale=1.25] at (1.25,0) {,};
\end{tikzpicture}
\end{equation}
in $\hass{P}$, so $o(c_{w,x_1}) = \alpha_1 = \alpha_2 = o(c_{w,x_2})$, which proves that $o'(c_{w,\tilde{x}})$ is also well-defined. This makes $P'$ an oriented graded poset.
\end{cons}

\begin{remark} \label{rmk:mergerglobe}
If $P$ is a constructible $n$\nbd molecule and $m_U: P \leadsto P'$ a simple merger where $U \sqsubseteq P$, then $P'$ is isomorphic to $P[U'/U]$, where $U'$ is a constructible $n$\nbd atom with $\bord{}{\alpha}U' = \bord{}{\alpha}U$. By Lemma \ref{lem:substitution}, $P'$ is also a constructible $n$\nbd molecule.
\end{remark}

\begin{exm} \label{exm:merger}
The following depicts a sequence of two simple mergers on a constructible 2-molecule:
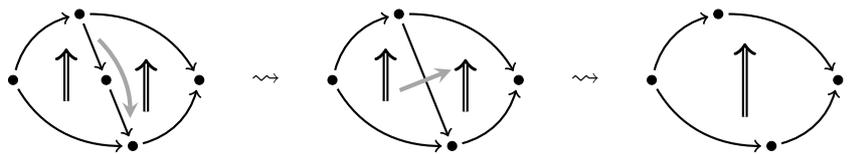
\begin{equation*}
\begin{tikzpicture}[baseline={([yshift=-.5ex]current bounding box.center)},scale=.7]
\begin{scope}[shift={(-6,0)}]
	\node[0c] (a) at (-1.75,0) {};
	\node[0c] (b) at (.5,-1.25) {};
	\node[0c] (c) at (-.5,1.25) {};
	\node[0c] (d) at (1.75,0) {};
	\node[0c] (m) at (0,0) {};
	\draw[1c, out=-60, in=180] (a) to (b);
	\draw[1c, out=0, in=120] (c) to (d);
	\draw[1c, out=15, in=-105] (b) to (d);
	\draw[1c, out=75, in=-165] (a) to (c);
	\draw[1c] (c) to (m);
	\draw[1c] (m) to (b);
	\draw[2c] (-.75,-.7) to (-.75,.9);
	\draw[2c] (.75,-.9) to (.75,.7);
	\node[scale=1.25] at (3,0) {$\leadsto$};
	\draw[follow, out=-45, in=90] (-.25,.875) to (.45,-.875);
\end{scope}
\begin{scope}
	\node[0c] (a) at (-1.75,0) {};
	\node[0c] (b) at (.5,-1.25) {};
	\node[0c] (c) at (-.5,1.25) {};
	\node[0c] (d) at (1.75,0) {};
	\draw[1c, out=-60, in=180] (a) to (b);
	\draw[1c, out=0, in=120] (c) to (d);
	\draw[1c, out=15, in=-105] (b) to (d);
	\draw[1c, out=75, in=-165] (a) to (c);
	\draw[1c] (c) to (b);
	\draw[2c] (-.75,-.7) to (-.75,.9);
	\draw[2c] (.75,-.9) to (.75,.7);
	\node[scale=1.25] at (3,0) {$\leadsto$};
	\draw[follow] (-.625,-.25) to (.625,.25);
\end{scope}
\begin{scope}[shift={(6,0)}]
	\node[0c] (a) at (-1.75,0) {};
	\node[0c] (b) at (.5,-1.25) {};
	\node[0c] (c) at (-.5,1.25) {};
	\node[0c] (d) at (1.75,0) {};
	\draw[1c, out=-60, in=180] (a) to (b);
	\draw[1c, out=0, in=120] (c) to (d);
	\draw[1c, out=15, in=-105] (b) to (d);
	\draw[1c, out=75, in=-165] (a) to (c);
	\draw[2c] (0,-1) to (0,1);
	\node[scale=1.25] at (2.25,-.9) {$.$};
\end{scope}
\end{tikzpicture}
\end{equation*}
\end{exm}

\begin{prop}
Let $P$ be an oriented thin poset, and $m_U: P \leadsto P'$ a simple merger. Then $P'$ is an oriented thin poset.
\end{prop}
\begin{proof}
We only need to look at intervals of the form $[w,\tilde{x}]$; the property is immediate from oriented thinness of $P$ for all other intervals. Any path $\tilde{x} \to z \to w$ in $\hass{P}'_\bot$ must come from a path $x_i \to z \to w$ in $\hass{P}_\bot$, say $i = 1$. This can be completed to
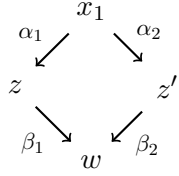
\begin{equation*}
\begin{tikzpicture}[baseline={([yshift=-.5ex]current bounding box.center)}]
	\node[scale=1.25] (0) at (0,2) {$x_1$};
	\node[scale=1.25] (1) at (-1,1) {$z$};
	\node[scale=1.25] (1b) at (1,1) {$z'$};
	\node[scale=1.25] (2) at (0,0) {$w$};
	\draw[1c] (0) to node[auto,swap] {$\alpha_1$} (1);
	\draw[1c] (0) to node[auto] {$\alpha_2$} (1b);
	\draw[1c] (1) to node[auto,swap] {$\beta_1$} (2);
	\draw[1c] (1b) to node[auto] {$\beta_2$} (2);
\end{tikzpicture}
\end{equation*}
for a unique $z'$ in $P$, with $\alpha_1\beta_1 = -\alpha_2\beta_2$. If $z' \neq y$, the diagram translates faithfully to $P'$. If $z' = y$, by oriented thinness applied to the interval $[w, x_2]$ of $P$, there exists a unique $z'' \neq z, y$ such that 
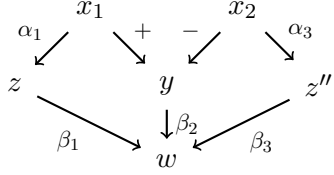
\begin{equation*}
\begin{tikzpicture}[baseline={([yshift=-.5ex]current bounding box.center)}]
	\node[scale=1.25] (0) at (0,2) {$x_1$};
	\node[scale=1.25] (0b) at (2,2) {$x_2$};
	\node[scale=1.25] (1) at (-1,1) {$z$};
	\node[scale=1.25] (1b) at (1,1) {$y$};
	\node[scale=1.25] (1c) at (3,1) {$z''$};
	\node[scale=1.25] (2) at (1,0) {$w$};
	\draw[1c] (0) to node[auto,swap] {$\alpha_1$} (1);
	\draw[1c] (0) to node[auto] {$\!+$} (1b);
	\draw[1c] (0b) to node[auto,swap] {$-\!$} (1b);
	\draw[1c] (0b) to node[auto] {$\alpha_3$} (1c);
	\draw[1c] (1) to node[auto,swap] {$\beta_1$} (2);
	\draw[1c] (1b) to node[auto] {$\beta_2$} (2);
	\draw[1c] (1c) to node[auto] {$\beta_3$} (2);
\end{tikzpicture}
\end{equation*}
in $\hass{P}_\bot$, and $\alpha_1\beta_1 = -\beta_2 = -\alpha_3 \beta_3$. This becomes
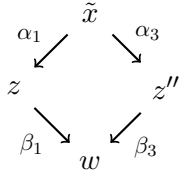
\begin{equation*}
\begin{tikzpicture}[baseline={([yshift=-.5ex]current bounding box.center)}]
	\node[scale=1.25] (0) at (0,2) {$\tilde{x}$};
	\node[scale=1.25] (1) at (-1,1) {$z$};
	\node[scale=1.25] (1b) at (1,1) {$z''$};
	\node[scale=1.25] (2) at (0,0) {$w$};
	\draw[1c] (0) to node[auto,swap] {$\alpha_1$} (1);
	\draw[1c] (0) to node[auto] {$\alpha_3$} (1b);
	\draw[1c] (1) to node[auto,swap] {$\beta_1$} (2);
	\draw[1c] (1b) to node[auto] {$\beta_3$} (2);
\end{tikzpicture}
\end{equation*}
in $\hass{P'}_\bot$, which completes the proof that $P'$ is an oriented thin poset.
\end{proof}

Next, we show that simple mergers ``reflect'' constructible molecules in oriented thin posets. In the following statement, $\invrs{m_U}(V')$ is the inverse image of $V'$ through $m_U$.

\begin{lem} \label{lem:reflectglobe}
Let $P, P'$ be oriented thin posets, $V' \subseteq P'$ a constructible $k$\nbd molecule, and $m_U: P \leadsto P'$ a simple merger. Then $V := \clos{(\invrs{m_U}(V'))}$ is a constructible $k$\nbd molecule in $P$, and $\sbord{}{\alpha}V = \invrs{m_U}(\sbord{}{\alpha}V')$.
\end{lem}
\begin{proof}
Let $U = \clos\{x_1,x_2\}$, a constructible $n$\nbd molecule in $P$, and let $\tilde{x} \in P'$ be the result of the identification of $x_1$ and $x_2$. If $V'$ does not contain $\tilde{x}$, which is always the case if $k < n$, then $V'$ is the identical image of $V$, hence a constructible $k$\nbd molecule in $P$, with the same boundary.

Suppose $k = n$ and $\tilde{x} \in V'$. We prove that $V$ is a constructible molecule in $P$ by induction on increasing constructible molecules $W' \sqsubseteq V'$ with $\tilde{x} \in W'$, letting $W := \clos{(\invrs{m_U}(W'))}$. If $W' = \clos\{\tilde{x}\}$, then $W = U$, and $\bord{}{\alpha}W' = \bord{}{\alpha}W$. Suppose $W'$ splits as $W'_1 \cup W'_2$, with $\tilde{x} \in W'_i$. Then $W = W_1 \cup W_2$, and by the inductive hypothesis both are constructible molecules, whose boundaries are unchanged by the merger. 

This proves that $V'$ is a constructible $n$\nbd molecule, and also that, if $W' \sqsubseteq V'$ in $P'$, then $W \sqsubseteq V$ in $P$. From here, we can prove by a straightforward induction that if $V'$ is a constructible $k$\nbd molecule in $P'$ with $k > n$, then $V$ is a constructible $k$\nbd molecule in $P$ with the desired properties.
\end{proof}

\begin{dfn}
Let $P, P'$ be oriented thin posets. A \emph{merger} $m: P \mrg P'$ is a finite sequence $P \leadsto P_1 \leadsto \ldots \leadsto P_n \leadsto P'$ of simple mergers. We declare two mergers $m, m': P \mrg P'$ to be equal if they are equal as composites of partial functions.
\end{dfn}
\begin{remark}
We allow the sequence to be of length zero: on each oriented thin poset there is an identity merger $P \mrg P$.
\end{remark}

In Example \ref{exm:merger}, we reduced a constructible 2-molecule to an atom with a sequence of simple mergers. The following lemma shows that this is possible in any dimension.

\begin{lem} \label{lem:atomicmerge}
Let $U$ be a constructible $n$\nbd molecule. Then there exist a constructible $n$\nbd atom $\tilde{U}$ and a sequence of simple mergers $U \leadsto U_1 \leadsto \ldots \leadsto U_m \leadsto \tilde{U}$ that restricts to the identity on $\bord{}{}U$, and such that the $U_i$ are all constructible $n$\nbd molecules.

If $U \sqsubseteq P$ for a constructible $n$\nbd molecule $P$, this extends to a merger $P \mrg P[\tilde{U}/U]$.
\end{lem}
\begin{proof}
If $U$ is an atom, the identity merger will do; in particular, the statement is trivially true for $n = 0$. From there, we proceed by induction on $n$ and on the number of maximal elements of $U$.

Suppose that $n > 0$, and $U$ is non-atomic; then there must exist $V \sqsubseteq U$ with only two maximal elements $x_1$, $x_2$. If $\clos{\{x_1\}} \cap \clos{\{x_2\}} = \bord{}{+}x_1 \cap \bord{}{-}x_2 = \clos{\{y\}}$ for an $(n-1)$\nbd dimensional element $y$, then $V$ satisfies the conditions for the existence of a simple merger $m_V: U \leadsto U'$. By Remark \ref{rmk:mergerglobe}, $U'$ is a constructible $n$\nbd molecule with one less maximal element, and we can apply the inductive hypothesis to obtain a merger $U \leadsto U' \mrg \tilde{U}$.

Next, notice that the general case is a consequence of the following statement:
\begin{sublem}
Let $V \sqsubseteq U$ be a constructible submolecule with two maximal elements $x_1, x_2$, and let $W$ be the constructible $(n-1)$\nbd molecule $\clos{\{x_1\}} \cap \clos{\{x_2\}}$. Then there is a merger $m: W \mrg \tilde{W}$ such that
\begin{enumerate}
	\item $\tilde{W}$ is a constructible $(n-1)$\nbd atom, 
	\item $m$ restricts to the identity on $\bord{}{}W$, and 
	\item extending $m$ as the identity on $U \setminus W$ determines a merger $m': U \mrg U'$, where $U'$ is a constructible $n$\nbd molecule.
\end{enumerate}
\end{sublem}
This is because $\clos{\{x_1\}} \cap \clos{\{x_2\}} = \tilde{W}$ is a constructible $(n-1)$\nbd atom in $U'$, and $\bord{}{}V$ is not affected by $m'$, so we end up in the previous situation and can merge $x_1$ and $x_2$ in $U'$.

\begin{proof}[Proof of the Sub-Lemma]
We proceed again by induction, now on the dimension and number of maximal elements of $W$. If $W$ is an atom, which is necessarily true if $n = 1$, the identity merger $W \mrg W$ extends to the identity merger on $P$. 

Suppose that $n > 1$, and $W$ is non-atomic. Proceeding as before, take a constructible submolecule $Z \sqsubseteq W$ with two maximal elements $y_1$, $y_2$. By the inductive hypothesis, we can now assume that $\clos{\{y_1\}} \cap \clos{\{y_2\}}$ is a constructible $(n-2)$\nbd molecule $\clos\{z\}$; otherwise, we take a merger $Z \mrg \tilde{Z}$ and extend it to $W$. Then there is a simple merger $W \leadsto W'$ that merges $y_1$ and $y_2$, such that $W'$ is a constructible $(n-1)$\nbd molecule with one less maximal element. It suffices to prove that this extends to a simple merger $U \leadsto U'$, such that $U'$ is a constructible $n$\nbd molecule.

Let $A := U^{(n)}\setminus\{x_1,x_2\}$. By definition, $z$ is only covered by $y_1$ and $y_2$ in $W$. Suppose that there is another $(n-1)$\nbd dimensional $\tilde{y}$ covering $z$ in $U$; since $U$ is pure, $\tilde{y}$ is covered by at least one $\tilde{x} \in U^{(n)}$, which, by thinness applied to the intervals $[z,x_1]$ and $[z,x_2]$, is different from $x_1$ and $x_2$, that is, $\tilde{x} \in A$. Then $z \in V \cap \clos{A} \subseteq \bord{}{}V$; but any element $y$ which covers $z$ is either $y_1$ or $y_2$, in which case $y \notin \sbord{}{}V$, or, by thinness, is covered neither by $x_1$ nor by $x_2$, and again $y \notin \sbord{}{}V$. It follows that $z \notin \clos{(\sbord{}{}V)} = \bord{}{}V$, a contradiction. Therefore $z$ is only covered by $y_1$ and $y_2$ in $U$, and $U \leadsto U'$ is a simple merger. 

Moreover, any merger tree for $U$ with a vertex labelled $\clos\{x_1,x_2\}$ is still a valid merger tree for $U'$: $\bord{}{+}x_1$, $\bord{}{-}x_2$, and their intersection only have $W'$ substituted for $W$ in $U'$, so by Lemma \ref{lem:substitution} they are constructible $(n-1)$\nbd molecules, and all the relevant submolecule conditions hold.
\end{proof}
The first part of the main statement follows. For the second part, it suffices to observe that our procedure can be relativised to any constructible submolecule $U$ of a constructible $n$\nbd molecule $P$, and produces a merger $P \mrg P[\tilde{U}/U]$.
\end{proof}

\begin{thm}[Globularity] \label{thm:globularity}
Let $U$ be a constructible $n$\nbd molecule. Then
\begin{equation*} 
	\bord{}{\alpha}(\bord{}{+}U) = \bord{}{\alpha}(\bord{}{-}U) \quad\quad \text{and} \quad\quad
\bord{}{+}U \cap \bord{}{-}U = \bord{}{}(\bord{}{+}U) = \bord{}{}(\bord{}{-}U).
\end{equation*}
\end{thm}
\begin{proof}
The statement is vacuously true when $n = 0$ or $n = 1$, so assume $n > 1$. Suppose $U$ is an atom with greatest element $x$. By Lemma \ref{lem:basic} applied to the constructible $(n-1)$\nbd molecule $\bord{}{+}U$, there is at least one $z \in \sbord{}{\alpha}(\bord{}{+}U)$ for each of $\alpha \in \{+,-\}$. The interval $[z,x]$ is of the form 
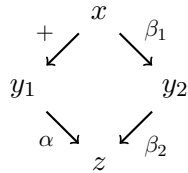
\begin{equation*}
\begin{tikzpicture}[baseline={([yshift=-.5ex]current bounding box.center)}]
	\node[scale=1.25] (0) at (0,2) {$x$};
	\node[scale=1.25] (1) at (-1,1) {$y_1$};
	\node[scale=1.25] (1b) at (1,1) {$y_2$};
	\node[scale=1.25] (2) at (0,0) {$z$};
	\draw[1c] (0) to node[auto,swap] {$+$} (1);
	\draw[1c] (0) to node[auto] {$\beta_1$} (1b);
	\draw[1c] (1) to node[auto,swap] {$\alpha$} (2);
	\draw[1c] (1b) to node[auto] {$\beta_2$} (2);
\end{tikzpicture}
\end{equation*}
for some $y_1, y_2$. Suppose that $\beta_1 = +$; then $y_2 \in \sbord{}{+}U$, yet necessarily $\beta_2 = -\alpha$, contradicting $z \in \sbord{}{\alpha}(\bord{}{+}U)$. Therefore $\beta_1 = -$ and $\beta_2 = \alpha$; because $z$ is only covered by $y_1$ and $y_2$, it follows that $z \in \sbord{}{\alpha}(\bord{}{-}U)$. The converse is symmetrical, and we have proved $\bord{}{\alpha}(\bord{}{+}U) = \bord{}{\alpha}(\bord{}{-}U)$.

This also implies that $\bord{}{}(\bord{}{\alpha}U) \subseteq \bord{}{+}U \cap \bord{}{-}U$. Conversely, if $z \in \bord{}{+}U \cap \bord{}{-}U$, the interval $[z,x]$ is of the form 
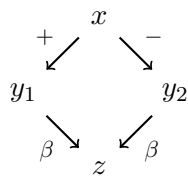
\begin{equation*}
\begin{tikzpicture}[baseline={([yshift=-.5ex]current bounding box.center)}]
	\node[scale=1.25] (0) at (0,2) {$x$};
	\node[scale=1.25] (1) at (-1,1) {$y_1$};
	\node[scale=1.25] (1b) at (1,1) {$y_2$};
	\node[scale=1.25] (2) at (0,0) {$z$};
	\draw[1c] (0) to node[auto,swap] {$+$} (1);
	\draw[1c] (0) to node[auto] {$-$} (1b);
	\draw[1c] (1) to node[auto,swap] {$\beta$} (2);
	\draw[1c] (1b) to node[auto] {$\beta$} (2);
\end{tikzpicture}
\end{equation*}
for some $\beta$, so $z \in \sbord{}{\beta}(\bord{}{+}U)$ and $z \in \sbord{}{\beta}(\bord{}{-}U)$. 

Now, suppose $U$ is non-atomic, and let $U \mrg \tilde{U}$ be a merger as in Lemma \ref{lem:atomicmerge}, restricting to the identity on $\bord{}{}U$. Applying the first part of the proof to $\tilde{U}$, we obtain the statement for $U$.
\end{proof}

\begin{cons}
By Theorem \ref{thm:globularity}, for any constructible directed complex $P$, the diagram
\begin{equation*}
\begin{tikzpicture}
	\node[scale=1.25] (0) at (-.5,0) {$\glob{0}{P}$};
	\node[scale=1.25] (1) at (2,0) {$\glob{1}{P}$};
	\node[scale=1.25] (2) at (4,0) {$\ldots$};
	\node[scale=1.25] (3) at (6,0) {$\glob{n}{P}$};
	\node[scale=1.25] (4) at (8,0) {$\ldots$};
	\draw[1c] (1.west |- 0,.15) to node[auto,swap] {$\bord{}{+}$} (0.east |- 0,.15);
	\draw[1c] (1.west |- 0,-.15) to node[auto] {$\bord{}{-}$} (0.east |- 0,-.15);
	\draw[1c] (2.west |- 0,.15) to node[auto,swap] {$\bord{}{+}$} (1.east |- 0,.15);
	\draw[1c] (2.west |- 0,-.15) to node[auto] {$\bord{}{-}$} (1.east |- 0,-.15);
	\draw[1c] (3.west |- 0,.15) to node[auto,swap] {$\bord{}{+}$} (2.east |- 0,.15);
	\draw[1c] (3.west |- 0,-.15) to node[auto] {$\bord{}{-}$} (2.east |- 0,-.15);
	\draw[1c] (4.west |- 0,.15) to node[auto,swap] {$\bord{}{+}$} (3.east |- 0,.15);
	\draw[1c] (4.west |- 0,-.15) to node[auto] {$\bord{}{-}$} (3.east |- 0,-.15);
\end{tikzpicture}
\end{equation*}
is an $\omega$\nbd graph. Any inclusion $P \hookrightarrow Q$ induces inclusions $\glob{n}{P} \hookrightarrow \glob{n}{Q}$ of sets of constructible $n$\nbd molecules, and these obviously commute with boundaries, so this determines a functor $\glob{}{}: \globpos \to \globset$. 
\end{cons}

With little more effort, we can get a more refined functorial invariant than an $\omega$\nbd graph. The following definition is Steiner's \cite{steiner2004omega}.
\begin{dfn}
An \emph{augmented directed complex} $(K, K^*)$ is an augmented chain complex $K$ of abelian groups, concentrated in positive degree,
\begin{equation*}
\begin{tikzpicture}[baseline={([yshift=-.5ex]current bounding box.center)}]
	\node[scale=1.25] (0) at (0,0) {$\mathbb{Z}$};
	\node[scale=1.25] (1) at (2,0) {$K_0$};
	\node[scale=1.25] (2) at (4,0) {$\ldots$};
	\node[scale=1.25] (3) at (6,0) {$K_n$};
	\node[scale=1.25] (4) at (8,0) {$\ldots$};
	\draw[1c] (1.west) to node[auto,swap] {$e$} (0.east);
	\draw[1c] (2.west) to node[auto,swap] {$d$} (1.east);
	\draw[1c] (3.west) to node[auto,swap] {$d$} (2.east);
	\draw[1c] (4.west) to node[auto,swap] {$d$} (3.east);
\end{tikzpicture}
\end{equation*}
together with, for each $n \in \mathbb{N}$, a distinguished submonoid $K^*_n$ of $K_n$. 

A morphism of augmented directed complexes $(K, K^*) \to (L, L^*)$ is a morphism of augmented chain complexes $f: K \to L$ such that $f(K^*_n) \subseteq L^*_n$ for all $n$. Augmented directed complexes and their morphisms form a category $\adc$.
\end{dfn}

\begin{cons} \label{cons:adc}
Let $P$ be a constructible directed complex. For all $n$, let $KP_n$ be $\mathbb{Z}P^{(n)}$, the free abelian group on the set of $n$\nbd dimensional elements of $P$, and $KP^*_n$ be its submonoid $\mathbb{N}P^{(n)}$. Then, we define $d: KP_{n+1} \to KP_n$ and $e: KP_0 \to \mathbb{Z}$ by freely extending
\begin{align*}
	dx & := \sum_{y \in \sbord{}{+}x} y \; - \sum_{y' \in \sbord{}{-}x} y', \\
	ex & := 1,
\end{align*}
for each $n$ and $x \in P^{(n)}$. We claim that $(KP, KP^*)$ so defined is an augmented directed complex. Writing 
\begin{equation*}
	d^+x := \sum_{y \in \sbord{}{+}x} y, \quad\quad d^-x := \sum_{y' \in \sbord{}{-}x} y',
\end{equation*}
we have that, for all $n$\nbd dimensional $x$, with $n > 1$,
\begin{equation*}
	dd^\alpha x = \sum_{z \in \sbord{}{+}(\bord{}{\alpha} x)} z - \sum_{z' \in \sbord{}{-}(\bord{}{\alpha} x)} z'
\end{equation*}
because $\bord{}{\alpha}x$ is a constructible $(n-1)$\nbd molecule, so by Lemma \ref{lem:basic} any element $z$ covered by some $y \in \sbord{}{\alpha}x$ is either covered by another $y'$ with opposite orientation, in which case there is a cancellation $z - z = 0$ in $dd^\alpha x$, or is only covered by $y$, and belongs to $\sbord{}{\beta}(\bord{}{\alpha} x)$.

By Theorem \ref{thm:globularity}, then, $dd^+ x = dd^- x$, so $dd x = d(d^+ x - d^- x) = dd^+x - dd^-x = 0$. This proves that $dd = 0$, and $ed = 0$ is proved similarly. 

Any inclusion of constructible directed complexes induces inclusions of the free abelian groups on their $n$\nbd dimensional elements, and it is straightforward to check that this determines a functor $K: \globpos \to \adc$.
\end{cons}

Using globularity, we can prove that constructible $n$\nbd atoms are classified by pairs of constructible $(n-1)$\nbd molecules with isomorphic boundaries. 

\begin{prop} \label{prop:atomicnglobe}
Let $U$ be an oriented graded poset with a greatest element, and $\dmn{U} = n > 0$. The following are equivalent:
\begin{enumerate}[label=(\alph*)]
	\item $U$ is a constructible $n$\nbd atom;
	\item $\bord{}{+}U$ and $\bord{}{-}U$ are constructible $(n-1)$\nbd molecules, and satisfy $\bord{}{\alpha}(\bord{}{+}U) = \bord{}{\alpha}(\bord{}{-}U)$ and $\bord{}{+}U \cap \bord{}{-}U = \bord{}{}(\bord{}{+}U) = \bord{}{}(\bord{}{-}U)$.
\end{enumerate}
\end{prop}
\begin{proof}
The implication from $(a)$ to $(b)$ is immediate from the definition and from Theorem \ref{thm:globularity}. Conversely, we only need to check that $U$ is oriented thin. For all intervals $[z,x]$ of length 2 in $U_\bot$, if $\dmn{x} < n$ the interval is entirely contained in $(\bord{}{\alpha}U)_\bot$, and we can use oriented thinness of $\bord{}{\alpha}U$.

Suppose $\dmn{x} = n$, that is, $x$ is the greatest element of $U$. First, suppose that $z$ is only covered by elements of $\bord{}{+}U$: then $z \notin \sbord{}{}(\bord{}{-}U) = \sbord{}{}(\bord{}{+}U)$, and by Lemma \ref{lem:basic} $z$ is covered by exactly two elements $y_1, y_2$ of $\bord{}{+}U$ with opposite orientations. In this case, the interval $[z,x]$ is 
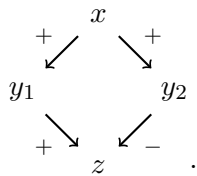
\begin{equation*}
\begin{tikzpicture}[baseline={([yshift=-.5ex]current bounding box.center)}]
	\node[scale=1.25] (0) at (0,2) {$x$};
	\node[scale=1.25] (1) at (-1,1) {$y_1$};
	\node[scale=1.25] (1b) at (1,1) {$y_2$};
	\node[scale=1.25] (2) at (0,0) {$z$};
	\draw[1c] (0) to node[auto,swap] {$+$} (1);
	\draw[1c] (0) to node[auto] {$+$} (1b);
	\draw[1c] (1) to node[auto,swap] {$+$} (2);
	\draw[1c] (1b) to node[auto] {$-$} (2);
	\node[scale=1.25] at (1.25,0) {.};
\end{tikzpicture}
\end{equation*}
The case where $z$ is only covered by elements of $\bord{}{-}U$ is analogous. 

Finally, suppose $z$ is covered both by elements of $\bord{}{-}U$ and of $\bord{}{+}U$. Then $z$ belongs to  $\sbord{}{\beta}(\bord{}{+}U) = \sbord{}{\beta}(\bord{}{-}U)$ for some $\beta$, and again by Lemma \ref{lem:basic} $z$ is covered by a single element $y_1$ of $\bord{}{+}U$ and by a single element $y_2$ of $\bord{}{-}U$. The interval $[z,x]$ is
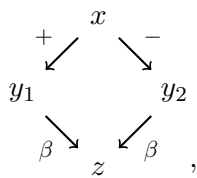
\begin{equation*}
\begin{tikzpicture}[baseline={([yshift=-.5ex]current bounding box.center)}]
	\node[scale=1.25] (0) at (0,2) {$x$};
	\node[scale=1.25] (1) at (-1,1) {$y_1$};
	\node[scale=1.25] (1b) at (1,1) {$y_2$};
	\node[scale=1.25] (2) at (0,0) {$z$};
	\draw[1c] (0) to node[auto,swap] {$+$} (1);
	\draw[1c] (0) to node[auto] {$-$} (1b);
	\draw[1c] (1) to node[auto,swap] {$\beta$} (2);
	\draw[1c] (1b) to node[auto] {$\beta$} (2);
	\node[scale=1.25] at (1.25,0) {,};
\end{tikzpicture}
\end{equation*}
which completes the proof.
\end{proof}

In particular, we have the following construction.

\begin{cons} \label{cons:ou}
Let $U$ and $V$ be constructible $n$\nbd molecules with (unique) isomorphisms $\bord{}{-}U \incliso \bord{}{-}V$ and $\bord{}{+}U \incliso \bord{}{+}V$.

Form the pushout of $\bord{}{}U \subseteq U$ and $\bord{}{}U \incliso \bord{}{}V \subseteq V$ in $\globpos$; then, let $U \celto V$ be the oriented graded poset obtained by adjoining a single $(n+1)$\nbd dimensional element $\top$ with $\bord{}{-}\top := U$ and $\bord{}{+}\top := V$. By Proposition \ref{prop:atomicnglobe}, $U \celto V$ is a constructible $(n+1)$\nbd atom.

In particular, if $\tilde{U}$ is the constructible $n$\nbd atom produced from $U$ by Lemma \ref{lem:atomicmerge}, we have constructible $(n+1)$\nbd atoms $U \celto \tilde{U}$ and $\tilde{U} \celto U$.
\end{cons}

\begin{lem} \label{lem:standardmerge}
Let $U$ be a constructible $n$\nbd molecule. Then there exists a unique merger $m: U \mrg O^n$, which restricts to a merger $\bord{}{}U \mrg \bord{}{}O^n$, and is a composite of simple mergers of constructible $n$\nbd molecules.
\end{lem}
\begin{proof}
If $n = 0$, we can take the identity merger. Suppose $n > 0$; by Lemma \ref{lem:atomicmerge}, there exist a constructible $n$\nbd atom $\tilde{U}$ and a merger $U \mrg \tilde{U}$ which is a composite of simple mergers of constructible $n$\nbd molecules and restricts to the identity on $\bord{}{}U$. Then, by the inductive hypothesis, there are mergers $\bord{}{+}\tilde{U} \mrg O^{n-1}$ and $\bord{}{-}\tilde{U} \mrg O^{n-1}$ that restrict to the identity on $\bord{}{}(\bord{}{+}\tilde{U}) = \bord{}{}(\bord{}{-}\tilde{U})$, and are composites of simple mergers of constructible $(n-1)$\nbd molecules. Extending them to $\tilde{U}$ in any sequence, we obtain a merger $\tilde{U} \mrg O^n$; because at any intermediate point the boundaries of the greatest element are constructible $(n-1)$\nbd molecules, this is a composite of simple mergers of constructible $n$\nbd molecules.

Uniqueness of mergers $U \mrg O^n$ follows from the fact that no simple merger can be undefined on elements of $U^{(n)}$, or $\sbord{}{}U$, or $\sbord{}{}(\bord{}{\alpha}U)$, and so on; therefore, any merger to $O^n$ must map all maximal elements of $U$ to $\underline{n}$, all elements of $\sbord{}{\alpha}U$ to $\underline{n-1}^\alpha$, and so on, and be undefined on all other elements.
\end{proof}

The following is a useful lemma making use of substitutions and globularity.
\begin{lem} \label{lem:boundarymove}
Suppose $U$ is a constructible $n$\nbd molecule splitting as $U_1 \cup U_2$. Then $\bord{}{+}U_1 \cup \bord{}{-}U_2$ is a constructible $(n-1)$\nbd molecule, and $\bord{}{+}U_1, \bord{}{-}U_2 \sqsubseteq \bord{}{+}U_1 \cup \bord{}{-}U_2$.
\end{lem}
\begin{proof}
By assumption, $\bord{}{-}U_1 \sqsubseteq \bord{}{-}U$, and by globularity $\bord{}{\alpha}(\bord{}{-}U_1) = \bord{}{\alpha}(\bord{}{+}U_1)$. Then $V := \bord{}{-}U[\bord{}{+}U_1/\bord{}{-}U_1]$ is well-defined, and there is an obvious inclusion $V \hookrightarrow U$ whose image is $\bord{}{+}U_1 \cup \bord{}{-}U_2$. By the substitution lemma, $\bord{}{+}U_1 \cup \bord{}{-}U_2$ is a constructible $(n-1)$\nbd molecule and $\bord{}{+}U_1 \sqsubseteq \bord{}{+}U_1 \cup \bord{}{-}U_2$. The dual argument using $\bord{}{+}U[\bord{}{-}U_2/\bord{}{+}U_2]$ also proves that $\bord{}{-}U_2 \sqsubseteq \bord{}{+}U_1 \cup \bord{}{-}U_2$.
\end{proof}

Another easy consequence of globularity is that, for all constructible $n$\nbd molecules $U$, the iterated boundary $\bord{}{\alpha_1}(\ldots(\bord{}{\alpha_k} U))$ only depends on $\alpha_1$, and not on any of the $\alpha_2,\ldots,\alpha_k$. There is, in fact, a direct characterisation of the same constructible molecule, which we now describe.

\begin{dfn}
Let $U$ be a closed subset of an oriented graded poset. For $\alpha \in \{+,-\}$ and $n \in \mathbb{N}$, let
\begin{align*}
	\sbord{n}{\alpha} U & := \{x \in U \,|\, \dmn{x} = n \text{ and, for all $y \in U$, if $y$ covers $x$, then $o(c_{y,x}) = \alpha$} \}, \\
	\bord{n}{\alpha} U & := \clos{(\sbord{n}{\alpha} U)} \cup \{ x \in U \,|\, \text{for all $y \in U$, if $x \leq y$, then $\dmn{y} \leq n$} \}, \\
	\sbord{n}{} U & := \sbord{n}{+}U \cup \sbord{n}{-}U, \quad \qquad \quad \bord{n}{}U := \bord{n}{+}U \cup \bord{n}{-} U.
\end{align*}
In particular, if $U$ is $n$\nbd dimensional, then $\bord{n-1}{\alpha}U = \bord{}{\alpha}U$, and $\bord{m}{\alpha}U = U$ for $m \geq n$. If $U$ is pure, for all $k < n$, $\bord{k}{\alpha}U = \clos{(\sbord{k}{\alpha} U)}$.

We call $\bord{n}{-}U$ the \emph{input $n$\nbd boundary}, and $\bord{n}{+}U$ the \emph{output $n$\nbd boundary} of $U$. For all $x \in P$, we will use the short-hand notation $\sbord{n}{\alpha}x := \sbord{n}{\alpha}\clos{\{x\}}$ and $\bord{n}{\alpha}x := \bord{n}{\alpha}\clos{\{x\}}$.
\end{dfn}

\begin{lem} \label{lem:intersection}
Let $U$ be a constructible $n$\nbd molecule, and suppose $V, W \subseteq U$ and $V \cap W$ are all constructible $n$\nbd molecules. Then $\sbord{}{}(V \cap W) \subseteq \sbord{}{}V \cup \sbord{}{}W$.
\end{lem}
\begin{proof}
Let $x \in \sbord{}{}(V \cap W)$; by Lemma \ref{lem:basic}, $x$ is covered by exactly one $n$\nbd dimensional $y \in V \cap W$. Suppose by way of contradiction that $x \notin \sbord{}{}V \cup \sbord{}{}W$: then $x \in V^{(n-1)}\setminus \sbord{}{}V$, and $x$ is covered by a single other element $y' \in V \setminus W$, but also $x \in W^{(n-1)}\setminus \sbord{}{}W$, and $x$ is covered by a single other element $y'' \in W \setminus V$. Thus there are at least three distinct $n$\nbd dimensional elements of $U$ covering $x$, contradicting Lemma \ref{lem:basic} applied to $U$.
\end{proof}

\begin{prop} \label{prop:globelike}
Let $U$ be a constructible $n$\nbd molecule. Then, for all $k < n$,
\begin{equation*}
	\bord{k}{\alpha}U = \underbrace{\bord{}{\alpha}(\ldots(\bord{}{\alpha}}_{n-k} U)).
\end{equation*} 
\end{prop}
\begin{remark}
This also implies that $\bord{k}{\alpha}U$ is a constructible $k$\nbd molecule, because we know that the right-hand side is.
\end{remark}
\begin{proof}
The statement clearly holds for $U = O^n$. We will prove that, for all simple mergers $m_V: U \leadsto U'$ of constructible $n$\nbd molecules, if the statement is true for $U'$, then it is true for $U$; the general statement will then follow from Lemma \ref{lem:standardmerge}.

From Lemma \ref{lem:reflectglobe}, we can derive that $\bord{}{\alpha}(\ldots(\bord{}{\alpha} U)) = \clos{(\invrs{m_V}(\bord{}{\alpha}(\ldots(\bord{}{\alpha} U'))))}$, so it will suffice to prove that $\bord{k}{\alpha}U = \clos{(\invrs{m_V}(\bord{k}{\alpha}U'))}$, or equivalently $w \in \sbord{k}{\alpha}U$ if and only if $m_V$ is defined on $w$ and $w \in \sbord{k}{\alpha}U'$, for all $w \in U$ and $k < n$. 

Let $V = \clos\{x_1,x_2\}$, a constructible $m$\nbd molecule in $U$, and let $\tilde{x} \in U'$ be the result of the identification of $x_1$ and $x_2$. Given $w \in U$, there are four possibilities.
\begin{itemize}
	\item $w = x_1$ or $x_2$. Any element that covers $x_1$ also covers $x_2$ (and vice versa) in $U$, and it covers $\tilde{x}$ in $U'$ with the same orientation. Thus, $x_1, x_2 \in \sbord{m}{\alpha}U$ if and only if $\tilde{x} \in \sbord{m}{\alpha}U'$.
	\item $w = y$. Then $y$ is covered by two elements with opposite orientations, so $y$ is not in $\sbord{m-1}{\alpha}U$, and $m_V$ is undefined on $y$.
	\item $w$ is covered by $y$ with orientation $\beta$. In $\hass{U}_\bot$, we have 
	\begin{equation} \label{eq:kboundary}
\begin{tikzpicture}[baseline={([yshift=-.5ex]current bounding box.center)}]
	\node[scale=1.25] (0) at (0,2) {$x_1$};
	\node[scale=1.25] (0b) at (2,2) {$x_2$};
	\node[scale=1.25] (1) at (-1,1) {$z$};
	\node[scale=1.25] (1b) at (1,1) {$y$};
	\node[scale=1.25] (1c) at (3,1) {$z'$};
	\node[scale=1.25] (2) at (1,0) {$w$};
	\draw[1c] (0) to node[auto,swap] {$\alpha_1$} (1);
	\draw[1c] (0) to node[auto] {$\!+$} (1b);
	\draw[1c] (0b) to node[auto,swap] {$-\!$} (1b);
	\draw[1c] (0b) to node[auto] {$\alpha_2$} (1c);
	\draw[1c] (1) to node[auto,swap] {$\beta_1$} (2);
	\draw[1c] (1b) to node[auto] {$\beta$} (2);
	\draw[1c] (1c) to node[auto] {$\beta_2$} (2);
\end{tikzpicture}
\end{equation}
	for unique $z$, $z'$; we will show that $\beta_1 = \beta$ or $\beta_2 = \beta$. Consequently, $w \notin \sbord{m-2}{-\beta}U$, and also $w \notin \sbord{m-2}{-\beta}U'$, since both $z$ and $z'$ also cover $w$ in $U'$; while $w \in \sbord{m-2}{\beta}U$ if and only if $w \in \sbord{m-2}{\beta}U'$, because $y$ is the only element that covers $w$ in $U$ but is not in $U'$.
	
	Assume by way of contradiction $\beta_1 = \beta_2 = -\beta$, and suppose, without loss of generality, that $\beta = -$. By oriented thinness, diagram (\ref{eq:kboundary}) becomes
	\begin{equation*}
	\begin{tikzpicture}[baseline={([yshift=-.5ex]current bounding box.center)}]
	\node[scale=1.25] (0) at (0,2) {$x_1$};
	\node[scale=1.25] (0b) at (2,2) {$x_2$};
	\node[scale=1.25] (1) at (-1,1) {$z$};
	\node[scale=1.25] (1b) at (1,1) {$y$};
	\node[scale=1.25] (1c) at (3,1) {$z'$};
	\node[scale=1.25] (2) at (1,0) {$w$};
	\draw[1c] (0) to node[auto,swap] {$+$} (1);
	\draw[1c] (0) to node[auto] {$\!+$} (1b);
	\draw[1c] (0b) to node[auto,swap] {$-\!$} (1b);
	\draw[1c] (0b) to node[auto] {$-$} (1c);
	\draw[1c] (1) to node[auto,swap] {$+$} (2);
	\draw[1c] (1b) to node[auto] {$-$} (2);
	\draw[1c] (1c) to node[auto] {$+$} (2);
	\node[scale=1.25] at (3.25,0) {.};
\end{tikzpicture}
\end{equation*}
Then $w \in \bord{}{+}x_1$, and because $w$ is covered by two elements of $\bord{}{+}x_1$ with opposite orientations, $w \notin \sbord{}{}(\bord{}{+}x_1)$; similarly, $w \in \bord{}{-}x_2$, but $w \notin \sbord{}{}(\bord{}{-}x_2)$. 

However, $w \in \sbord{}{}y = \sbord{}{}(\bord{}{+}x_1 \cap \bord{}{-}x_2)$, which contradicts Lemma \ref{lem:intersection} applied to $\bord{}{+}x_1, \bord{}{-}x_2 \subseteq \bord{}{+}x_1 \cup \bord{}{-}x_2$, all constructible $(m-1)$\nbd molecules by Lemma \ref{lem:boundarymove}.

\item In all other cases, both $w$ and the elements that cover it are left unchanged by the simple merger.
\end{itemize}
This completes the case distinction and the proof.
\end{proof}

\section{Monoidal structures and other operations} \label{sec:monoidal}

In this section, we study operations on constructible directed complexes and their extension to constructible polygraphs. Some of them are the same that Steiner considered in \cite[Theorem 2.19]{steiner1993algebra} for directed complexes.

\begin{cons}
Let $P, Q$ be oriented graded posets. The disjoint union $P + Q$ of $P$ and $Q$ as posets is graded, and inherits an orientation from its components. Clearly, $P + Q$ is an oriented thin poset or a constructible directed complex whenever $P$ and $Q$ are.
\end{cons}

\begin{cons} \label{cons:suspension}
Let $P$ be an oriented graded poset. The \emph{suspension} $\Sigma P$ of $P$ is the oriented graded poset obtained from $P$ by adjoining two minimal elements $\bot^-$ and $\bot^+$, such that $\bot^\alpha < x$ for all $x \in P$, and, for all 0-dimensional $y \in P$, $o(c_{y,\bot^-}) := -$ and $o(c_{y,\bot^+}) := +$. If $\dmn{x} = n$ in $P$, then $\dmn{x} = n+1$ in $\Sigma P$.

This construction extends to an endofunctor $\Sigma$ on $\ogpos$ in the obvious way.
\end{cons}

\begin{lem}
If $P$ is an oriented thin poset, then $\Sigma P$ is an oriented thin poset.
\end{lem}
\begin{proof}
Let $[x,y]$ be an interval of length 2 in $\Sigma P_\bot$.
\begin{itemize}
	\item If $\dmn{x} > 0$ in $\Sigma P$, then $[x,y]$ corresponds to an interval in $P$, which is of the required form (\ref{eq:signed}) by hypothesis.
	\item If $\dmn{x} = 0$ in $\Sigma P$, then $x = \bot^\alpha$, and by construction $[\bot^\alpha,y]$ is isomorphic to the interval $[\bot,y]$ in $P_\bot$ as a poset. By oriented thinness of $P$, it is of the form
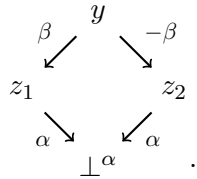
\begin{equation*}
\begin{tikzpicture}[baseline={([yshift=-.5ex]current bounding box.center)}]
	\node[scale=1.25] (0) at (0,2) {$y$};
	\node[scale=1.25] (1) at (-1,1) {$z_1$};
	\node[scale=1.25] (1b) at (1,1) {$z_2$};
	\node[scale=1.25] (2) at (0,0) {$\bot^\alpha$};
	\draw[1c] (0) to node[auto,swap] {$\beta$} (1);
	\draw[1c] (0) to node[auto] {$-\beta$} (1b);
	\draw[1c] (1) to node[auto,swap] {$\alpha$} (2);
	\draw[1c] (1b) to node[auto] {$\alpha$} (2);
	\node[scale=1.25] at (1.25,0) {$.$};
\end{tikzpicture}
\end{equation*}
	\item If $x = \bot$, then $y$ is a 0-dimensional element of $P$, and by construction $[\bot,y]$ is
	
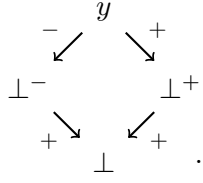
\begin{equation*}
\begin{tikzpicture}[baseline={([yshift=-.5ex]current bounding box.center)}]
	\node[scale=1.25] (0) at (0,2) {$y$};
	\node[scale=1.25] (1) at (-1,1) {$\bot^-$};
	\node[scale=1.25] (1b) at (1,1) {$\bot^+$};
	\node[scale=1.25] (2) at (0,0) {$\bot$};
	\draw[1c] (0) to node[auto,swap] {$-$} (1);
	\draw[1c] (0) to node[auto] {$+$} (1b);
	\draw[1c] (1) to node[auto,swap] {$+$} (2);
	\draw[1c] (1b) to node[auto] {$+$} (2);
	\node[scale=1.25] at (1.25,0) {$.$};
\end{tikzpicture}
\end{equation*}
\end{itemize}
This proves that $\Sigma P$ is an oriented thin poset.
\end{proof}

\begin{prop} \label{prop:suspensionglob}
Let $P$ be an oriented thin poset and $U$ a closed subset. Then $U$ is a constructible $(n-1)$\nbd molecule in $P$ if and only if $\Sigma U$ is a constructible $n$\nbd molecule in $\Sigma P$. Consequently, if $P$ is a constructible directed complex, then $\Sigma P$ is a constructible directed complex.
\end{prop}
\begin{proof}
For all $0$\nbd dimensional $x \in P$, $\clos{\{x\}}$ is isomorphic to the 1-globe in $\Sigma P$. For all other 0-dimensional $y$ in $P$, it is always the case that $\bord{}{\alpha}x \cap \bord{}{-\alpha}y$ is empty in $\Sigma P$; it follows that there are no non-atomic 1\nbd molecules in $\Sigma P$. Hence, we can establish a bijection between $\glob{1}{(\Sigma P)}$ and $\glob{0}{P}$, and use it as the basis of an inductive proof that $\glob{n}{(\Sigma P)}$ is in bijection with  $\glob{n-1}{P}$, for all $n > 0$: if $U$ is a constructible $(n-1)$\nbd molecule in $P$, then $\Sigma U = U \cup \{\bot^+,\bot^-\}$ is a constructible $n$\nbd molecule in $\Sigma P$.

Because $\Sigma \clos{\{x\}}$ is just $\clos\{x\}$ in $\Sigma P$, it follows that $\Sigma P$ is a constructible directed complex whenever $P$ is.
\end{proof}

\begin{exm}
For all $n > 0$, $\Sigma (O^{n-1})$ is isomorphic to $O^n$.
\end{exm}

\begin{cons} \label{cons:laxgray}
Let $P, Q$ be oriented graded posets. The \emph{lax Gray product} $P \tensor Q$ of $P$ and $Q$ is the product poset $P \times Q$, which is graded, with the following orientation: write $x \tensor y$ for an element $(x,y)$ of $P \times Q$, seen as an element of $P \tensor Q$; then, for all $x'$ covered by $x$ in $P$, and all $y'$ covered by $y$ in $Q$, let
\begin{align*}
	o(c_{x \tensor y, x' \tensor y}) & := o_P(c_{x,x'}), \\
	o(c_{x \tensor y, x \tensor y'}) & := (-)^{\dmn{x}}o_Q(c_{y,y'}),
\end{align*}
where $o_P$ and $o_Q$ are the orientations of $P$ and $Q$, respectively.

The lax Gray product defines a monoidal structure on $\ogpos$, whose monoidal unit is $1$, the oriented graded poset with a single element.
\end{cons}

\begin{remark}
The choice of orientation comes from tensor products of chain complexes; we will see from Construction \ref{cons:adc} that it ensures that a certain functor from $\globpos$ to a category of chain complexes with additional structure is monoidal.
\end{remark}

\begin{lem} \label{lem:thintensor}
If $P, Q$ are oriented thin posets, then so is $P \tensor Q$.
\end{lem}
\begin{proof}
An interval of length 2 in $P \tensor Q$ has one of the following forms:
\begin{itemize}
	\item $[x', x] \tensor \{y\}$ for some interval $[x', x]$ of length 2 in $P$, and $y \in Q$;
	\item $\{x\} \tensor [y',y]$ for some interval $[y',y]$ of length 2 in $Q$, and $x \in P$;
	\item $[x'\tensor y', x \tensor y]$, where $x$ covers $x'$ in $P$ and $y$ covers $y'$ in $Q$.
\end{itemize}
In the first two cases, oriented thinness follows from the oriented thinness of $P$ and $Q$, respectively. In the last case, let $\alpha := o_P(c_{x,x'})$, $\beta := o_Q(c_{y,y'})$. Then the interval has the form
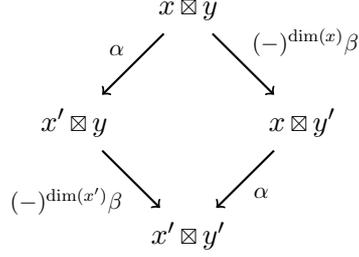
\begin{equation*}
\begin{tikzpicture}[baseline={([yshift=-.5ex]current bounding box.center)}, scale=1.5]
	\node[scale=1.25] (0) at (0,2) {$x\tensor y$};
	\node[scale=1.25] (1) at (-1,1) {$x' \tensor y$};
	\node[scale=1.25] (1b) at (1,1) {$x \tensor y'$};
	\node[scale=1.25] (2) at (0,0) {$x' \tensor y'$};
	\draw[1c] (0) to node[auto,swap] {$\alpha$} (1);
	\draw[1c] (0) to node[auto] {$(-)^{\dmn{x}}\beta$} (1b);
	\draw[1c] (1) to node[auto,swap] {$(-)^{\dmn{x'}}\beta$} (2);
	\draw[1c] (1b) to node[auto] {$\alpha$} (2);
\end{tikzpicture}
\end{equation*}
in the labelled Hasse diagram of $P \tensor Q$, and $(-)^{\dmn{x'}} = (-)^{\dmn{x} - 1} = - (-)^{\dmn{x}}$. The case of intervals $[\bot, x \tensor y]$ in $(P \tensor Q)_\bot$ can easily be handled explicitly.
\end{proof}

In fact, the lax Gray product preserves the class of constructible molecules, and consequently of constructible directed complexes. We give a first sketch of the proof.
\begin{lem} \label{lem:globtensor}
Let $U$ be a constructible $n$\nbd molecule and $V$ a constructible $m$\nbd molecule. Then $U \tensor V$ is a constructible $(n+m)$\nbd . If $U' \sqsubseteq U$ and $V' \sqsubseteq V$, then $U' \tensor V' \sqsubseteq U \tensor V$.
\end{lem}
\begin{proof}[Sketch of the proof]
We proceed by double induction on the dimension and number of maximal elements of $U$ and $V$. If $U$ or $V$ is 0-dimensional, then $U \tensor V$ is isomorphic to $V$ or $U$, respectively, and there is nothing to prove. 

Suppose $n, m > 0$; then $U \tensor V$ is pure and $(n+m)$\nbd dimensional. We first need to show that $\bord{}{\alpha}(U \tensor V)$ is a constructible molecule. We have
\begin{equation*}
	\bord{}{\alpha}(U \tensor V) = (\bord{}{\alpha}U \tensor V) \cup (U \tensor \bord{}{(-)^n\alpha}V);
\end{equation*}
by the inductive hypothesis, both $(\bord{}{\alpha}U \tensor V)$ and $(U \tensor \bord{}{(-)^n\alpha}V)$ are constructible, and their intersection $\bord{}{\alpha}U \tensor \bord{}{(-)^n\alpha}V$ is constructible. We can show that this is, in fact, a decomposition into constructible submolecules.

If $U$ and $V$ are atomic, there is nothing else to prove. Suppose $U$ is non-atomic and splits as $U_1 \cup U_2$, so $U \tensor V = (U_1 \tensor V) \cup (U_2 \tensor V)$. By the inductive hypothesis, $(U_1 \tensor V)$ and $(U_2 \tensor V)$ are both constructible $(n+m)$\nbd molecules, and their intersection $(U_1 \tensor V) \cap (U_2 \tensor V) = (U_1 \cap U_2) \tensor V$ is a constructible $(n+m-1)$\nbd molecule. We can show that this is a decomposition into constructible submolecules.

In the same way, if $V$ splits as $V_1 \cup V_2$, then $U \tensor V$ splits as $(U \tensor V_1) \cup (U \tensor V_2)$.
\end{proof}

While the structure of the proof is straightforward, checking the various submolecule conditions on the boundaries, or that the $\bord{}{\beta}(\bord{}{\alpha}(U \tensor V))$ are themselves submolecules, requires a series of tedious inductive arguments. We give a more detailed proof in Appendix \ref{sec:appendix}.

\begin{prop} \label{thm:globpostensor}
If $P, Q$ are constructible directed complexes, then so is $P \tensor Q$.
\end{prop}
\begin{proof}
For all $x \in P$ and $y \in Q$, we have $\clos\{x \tensor y\} = \clos\{x\} \tensor \clos\{y\}$, which is a constructible molecule by Lemma \ref{lem:globtensor}.
\end{proof}

Thus, the monoidal structure $(\ogpos, \boxtimes, 1)$ restricts to $\globpos$.

\begin{exm}
The $n$\nbd fold lax Gray product $K^n := \underbrace{\vec{I}\tensor\ldots\tensor\vec{I}}_n$ is the $n$\nbd cube, with the conventional orientation of homological algebra \cite{al2002multiple}.
\end{exm}

\begin{cons} \label{cons:join}
Let $P, Q$ be oriented graded posets, and take the lax Gray product $P_\bot \tensor Q_\bot$. This has a unique 0-dimensional element $\bot \tensor \bot$, and all 1-dimensional elements $x \tensor \bot$ and $\bot \tensor y$ cover it with orientation $+$; thus, $P_\bot \tensor Q_\bot$ is isomorphic to $(P \join Q)_\bot$ for a unique oriented graded poset $P \join Q$, the \emph{join} of $P$ and $Q$. 

We introduce the following notation for elements of $P \join Q$:
\begin{itemize}
	\item for all $x \in P^{(n)}$, we write $x \in P \join Q$ for the $n$\nbd dimensional element corresponding to $x \tensor \bot$ in $P_\bot \tensor Q_\bot$;
	\item for all $y \in Q^{(m)}$, we write $y \in P \join Q$ for the $m$\nbd dimensional element corresponding to $\bot \tensor y$ in $P_\bot \tensor Q_\bot$;
	\item for all $x \in P^{(n)}$, $y \in Q^{(m)}$, we write $x \join y \in P \join Q$ for the $(n+m+1)$\nbd dimensional element corresponding to $x \tensor y$ in $P_\bot \tensor Q_\bot$.
\end{itemize}
The first two define inclusions $P \hookrightarrow P \join Q$ and $Q \hookrightarrow P \join Q$. 

The join is an associative operation because the lax Gray product is, and it extends to a monoidal structure on $\ogpos$, whose unit is the empty oriented graded poset $\emptyset$. This makes $(-)_\bot$ a monoidal functor from $(\ogpos, \join, \emptyset)$ to $(\ogpos, \tensor, 1)$.
\end{cons}

\begin{remark}
By monoidal functor, we always mean a \emph{strong} monoidal functor.
\end{remark}

\begin{lem}
If $P, Q$ are oriented thin posets, then so is $P \join Q$.
\end{lem}
\begin{proof}
An interval of length 2 in $(P \join Q)_\bot = P_\bot \tensor Q_\bot$ is either of the form $[x',x]\tensor\{y\}$, or of the form $\{x\}\tensor [y',y]$, or of the form $[x'\tensor y',x \tensor y]$ for some $x', x \in P_\bot$ and $y', y \in Q_\bot$. In the first two cases, we can use oriented thinness of $P$ and $Q$. In the third case, if $x', y' \neq \bot$ we can use oriented thinness of $P \tensor Q$. The remaining few cases are easily checked explicitly.
\end{proof}

\begin{prop} \label{lem:globjoin}
Let $U, V$ be constructible molecules, $\dmn{U} = n, \dmn{V} = m$. Then $U \join V$ is a constructible $(n+m+1)$\nbd molecule. If $U' \sqsubseteq U$ and $V' \sqsubseteq V$, then $U' \join V' \sqsubseteq U \join V$.
\end{prop}
\begin{proof}
Consider the following assignment $j$ of elements of $\Sigma U \tensor \Sigma V$ to elements of $U \join V$ (not an inclusion of oriented graded posets):
\begin{equation*}
	j(z) := \begin{cases} x \tensor \bot^+ & \text{if } z = x \in P, \\
		\bot^+ \tensor y & \text{if } z = y \in Q, \\
		x \tensor y & \text{if } z = x \join y, x \in P, y \in Q.
		\end{cases}
\end{equation*}
This is induced by $(U \join V)_\bot = (U_\bot \tensor V_\bot)$ via the functions $U_\bot \to \Sigma U$, $V_\bot \to \Sigma V$ that map $\bot$ to $\bot^+$. 

By Proposition \ref{prop:suspensionglob}, $U' \sqsubseteq U$ and $V' \sqsubseteq V$ imply $\Sigma U' \sqsubseteq \Sigma U$ and $\Sigma V' \sqsubseteq \Sigma V$. Furthermore, $\bord{}{\alpha}(\Sigma U) = \Sigma \bord{}{\alpha} U$ when $U$ is at least 1-dimensional, and $\bord{}{\alpha}(\Sigma U) = \{\bot^\alpha\}$ when it is 0-dimensional. 

Using these facts, it is straightforward to turn the inductive proof that $\Sigma U \tensor \Sigma V$ is a constructible $(n+m+2)$\nbd molecule, as in Lemma \ref{lem:globtensor}, and $\Sigma U' \tensor \Sigma V' \sqsubseteq \Sigma U \tensor \Sigma V$, into a proof that $U \join V = \invrs{j}(\Sigma U \tensor \Sigma V)$ is a constructible $(n+m+1)$\nbd molecule, and $U' \join V' = \invrs{j}(\Sigma U' \tensor \Sigma V') \sqsubseteq U \join V$. If $n = m = 0$, then we can directly check that $U \join V$ is isomorphic to $O^1$. If $n = 0, m > 0$, then
\begin{equation*}
	\bord{}{\alpha}(U \join V) = \invrs{j}\bord{}{\alpha}(\Sigma U \tensor \Sigma V) = \invrs{j}(\{\bot^\alpha\} \tensor \Sigma V \cup \Sigma U \tensor \Sigma \bord{}{-\alpha}V),
\end{equation*}
hence
\begin{equation*}
	\bord{}{+}(U \join V) = V \cup (U \join \bord{}{-}V), \quad \quad \bord{}{-}(U \join V) = U \join \bord{}{+}V,
\end{equation*}
where, $V, U \join \bord{}{-}V$ and $U \join \bord{}{+}V$ are all constructible $m$\nbd molecules, and the intersection $V \cap U \join \bord{}{-}V = \bord{}{-}V$ is a constructible $(m-1)$\nbd molecule; we can show that $V \cup (U \join \bord{}{-}V)$ is a decomposition into constructible submolecules. The case $n >0, m = 0$ is dual, and for $n > 0, m > 0$, or when $U$ or $V$ are not atoms, one proceeds exactly as in Lemma \ref{lem:globtensor}.
\end{proof}

\begin{prop}
If $P, Q$ are constructible directed complexes, then so is $P \join Q$.
\end{prop}
\begin{proof}
For all $x \in P$ and $y \in Q$, that $\clos\{x\}$ and $\clos\{y\}$ are constructible molecules in $P \join Q$ is immediate from the existence of inclusions $P, Q \hookrightarrow P \join Q$, and $\clos\{x \join y\} = \clos\{x\} \join \clos\{y\}$ is a constructible molecule by Proposition \ref{lem:globjoin}.
\end{proof}

It follows that the monoidal structure $(\ogpos, \join, \emptyset)$ also restricts to $\globpos$; notice however that $(-)_\bot$ is not an endofunctor on $\globpos$.

\begin{exm}
The $(n+1)$\nbd fold join $\Delta^n := \underbrace{1\join\ldots\join 1}_{n+1}$ is the oriented $n$\nbd simplex \cite{street1987algebra}.
\end{exm}

The lax Gray product and join of oriented graded posets are related to the tensor product of chain complexes, and Steiner's extension to augmented directed complexes, in the following way.

\begin{dfn}
Let $K$ and $L$ be two chain complexes. The \emph{tensor product} $K \otimes L$ of $K$ and $L$ is the chain complex defined by
\begin{equation*}
	(K \otimes L)_n := \bigoplus_{k=0}^n K_k \otimes L_{n-k}, 
\end{equation*}
\begin{equation*}
	d(x \otimes y) := dx \otimes y + (-1)^{\dmn{x}} x \otimes dy,
\end{equation*}
where $\dmn{x} = n$ if $x \in K_n$. If $K$ and $L$ are augmented, $K \otimes L$ is augmented by $e: x \otimes y \mapsto ex \, ey$.

If $K$ is an augmented chain complex, the \emph{suspension} of $K$ is the chain complex $\Sigma K$ defined by
\begin{equation*}
	(\Sigma K)_n := \begin{cases} \mathbb{Z}, & n = 0, \\ K_{n-1}, & n > 0, \end{cases} 
\end{equation*} 
\begin{equation*}
	(d: \Sigma K_{n+1} \to \Sigma K_n) := \begin{cases} e: K_0 \to \mathbb{Z}, & n = 0, \\ d: K_{n} \to K_{n-1}, & n > 0. \end{cases}
\end{equation*}
Because $(\Sigma K \otimes \Sigma L)_0 = \mathbb{Z} \otimes \mathbb{Z} \simeq \mathbb{Z}$, there is a unique (up to isomorphism) augmented chain complex $K \join L$ such that $\Sigma(K \join L) \simeq \Sigma K \otimes \Sigma L$. This augmented chain complex is the \emph{join} of $K$ and $L$.

Let $(K, K^*)$ and $(L, L^*)$ be augmented directed complexes. The \emph{tensor product} $(K, K^*) \otimes (L, L^*)$ is $K \otimes L$ together with the distinguished submonoids $(K \otimes L)_n^*$ generated by elements $x \otimes y$ of $(K \otimes L)_n$ where $x \in K_k^*$ and $y \in L_{n-k}^*$. 

The \emph{suspension} $\Sigma(K, K^*)$ is $\Sigma K$ with the distinguished submonoids $\mathbb{N}$ of $(\Sigma K)_0$ and $K^*_{n-1}$ of $(\Sigma K)_n$, and the augmentation $0: \mathbb{Z} \to \mathbb{Z}$. 

The \emph{join} of $(K, K^*)$ and $(L, L^*)$ is the unique augmented directed complex satisfying $\Sigma((K, K^*) \join (L, L^*)) \simeq \Sigma(K,K^*)\otimes \Sigma(L,L^*)$.

The tensor product determines a monoidal structure on $\adc$, whose unit is the augmented directed complex $I$ with $I_0 :=\mathbb{Z}$, $I_0^* := \mathbb{N}$, $I_n := 0$ for $n > 0$, and the identity on $\mathbb{Z}$ as augmentation. The join also determines a monoidal structure on $\adc$, whose unit is the augmented directed complex $0$ equal to 0 in every degree.
\end{dfn}

\begin{prop} \label{prop:adcmonoidal}
The functor $K$ is monoidal from $(\globpos, \tensor, 1)$ to $(\adc, \otimes, I)$ and from $(\globpos, \join, \emptyset)$ to $(\adc, \join, 0)$.
\end{prop}
\begin{proof}
A simple comparison of the definitions.
\end{proof}

\begin{remark}
Lax Gray products and joins do not preserve the class of positive opetopes, nor the class of simple constructible molecules. For the first class, we can take the 2-cube $K^2 = \vec{I} \tensor \vec{I}$ or the 2-simplex $\Delta^2 = \vec{I} \join 1$ as counterexamples, since both $\vec{I}$ and $1$ are positive opetopes; note however that the dual $\oppall{(\Delta^2)}$, defined later in this section, is a positive opetope.

For the second class, the 2-cube $K^2$ is a simple constructible molecule, but neither the 3-cube $K^3 = K^2 \tensor \vec{I}$, nor the join $K^2 \join 1$ are simple:
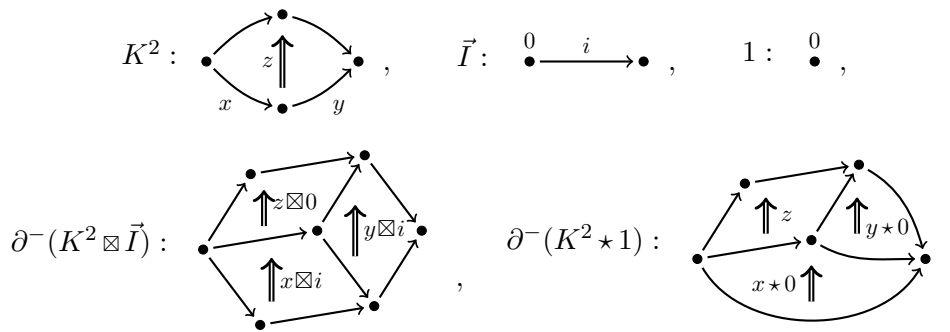
\begin{equation*}
\begin{tikzpicture}[baseline={([yshift=-.5ex]current bounding box.center)}]
\begin{scope}
	\node[0c] (0) at (-1, 0) {};
	\node[0c] (1) at (1,0) {};
	\node[0c] (i) at (0,-.625) {};
	\node[0c] (o) at (0,.625) {};
	\draw[1c,out=-45,in=165] (0) to node[auto,swap] {$x$} (i);
	\draw[1c,out=15,in=-135] (i) to node[auto,swap] {$y$} (1);
	\draw[1c,out=45,in=-165] (0) to (o);
	\draw[1c,out=-15,in=135] (o) to (1);
	\draw[2c] (i) to node[auto] {$z$} (o);
	\node[scale=1.25] at (-1.75,.125) {$K^2 :$};
	\node[scale=1.25] at (1.375,-.125) {,};
\end{scope}
\begin{scope}[shift={(4,0)}]
	\node[0c] (0) at (-.75, 0) [label=above:$0$] {};
	\node[0c] (1) at (.75,0) {};
	\draw[1c] (0) to node[auto] {$i$} (1);
	\node[scale=1.25] at (-1.5,.125) {$\vec{I} :$};
	\node[scale=1.25] at (1.125,-.125) {,};
\end{scope}
\begin{scope}[shift={(7,0)}]
	\node[0c] (0) at (0, 0) [label=above:$0$] {};
	\node[scale=1.25] at (-.75,.125) {$1 :$};
	\node[scale=1.25] at (.375,-.125) {,};
\end{scope}
\end{tikzpicture}
\end{equation*}
\begin{equation*}
\begin{tikzpicture}[baseline={([yshift=-.5ex]current bounding box.center)}]
\begin{scope}
	\node[0c] (0) at (-1.5, 0) {};
	\node[0c] (o1) at (-.875,1) {};
	\node[0c] (o2) at (.625,1.25) {};
	\node[0c] (m) at (0,.25) {};
	\node[0c] (i2) at (.75,-.75) {};
	\node[0c] (i1) at (-.75,-1) {};
	\node[0c] (1) at (1.375,.25) {};
	\draw[1c] (0) to (o1);
	\draw[1c] (o1) to (o2);
	\draw[1c] (0) to (m);
	\draw[1c] (m) to (o2);
	\draw[1c] (0) to (i1);
	\draw[1c] (i1) to (i2);
	\draw[1c] (i2) to (1);
	\draw[1c] (o2) to (1);
	\draw[1c] (m) to (i2);
	\draw[2c] (-.7,.1) to node[auto,swap,pos=.6] {$z\!\boxtimes\!0$} (-.7,1);
	\draw[2c] (-.6,-.9) to node[auto,swap] {$x\!\boxtimes\!i$} (-.6,.1);
	\draw[2c] (.5,-.3) to node[auto,swap] {$y\!\boxtimes\!i$} (.5,.8);
	\node[scale=1.25] at (-3,.125) {$\bord{}{-}(K^2 \tensor \vec{I}) :$};
	\node[scale=1.25] at (1.875,-.5) {,};
\end{scope}
\begin{scope}[shift={(6.5,0)}]
	\node[0c] (0) at (-1.5, -.125) {};
	\node[0c] (1) at (1.5,-.125) {};
	\node[0c] (o1) at (-.875,.875) {};
	\node[0c] (o2) at (.625,1.125) {};
	\node[0c] (m) at (0,.125) {};
	\draw[1c] (0) to (o1);
	\draw[1c] (o1) to (o2);
	\draw[1c] (0) to (m);
	\draw[1c] (m) to (o2);
	\draw[1c, out=-30, in=105] (o2) to (1);
	\draw[1c, out=-30, in=180] (m) to (1);
	\draw[1c,out=-60, in=-120] (0) to (1);
	\draw[2c] (0,-.9) to node[auto] {$x \join 0\;$} (0,0);
	\draw[2c] (.6,0) to node[auto,swap,pos=.3] {$y \join 0$} (.6,.9);
	\draw[2c] (-.6,0) to node[auto,swap] {$\;z$} (-.6,.9);
	\node[scale=1.25] at (-3,.125) {$\bord{}{-}(K^2\join 1) :$};
	\node[scale=1.25] at (1.875,-.5) {.};
\end{scope}
\end{tikzpicture}
\end{equation*}
\end{remark}

\begin{cons} \label{cons:jdual}
Let $P$ be an oriented graded poset, and $J \subseteq \mathbb{N}^+ = \mathbb{N}\setminus\{0\}$. Then $\oppn{J}{P}$, the \emph{$J$\nbd dual} of $P$, is the oriented graded poset with the same underlying poset as $P$, and the orientation $o'$ defined by 
\begin{equation*}
o'(c_{x,y}) := \begin{cases}
		-o(c_{x,y}), & \dmn{x} \in J, \\
		o(c_{x,y}), & \dmn{x} \not\in J,
	\end{cases}
\end{equation*}
for all elements $x, y \in P$ such that $x$ covers $y$. This extends to an involutive endofunctor $\oppn{J}{-}$ on $\ogpos$.
\end{cons}

\begin{prop} \label{prop:oppglobpos}
Let $P$ be a constructible directed complex, $J \subseteq \mathbb{N}^+$. Then $\oppn{J}{P}$ is a constructible directed complex.
\end{prop}
\begin{proof}
The underlying poset of $\oppn{J}{P}$ is the same as the underlying poset of $P$, so it is thin when $P$ is. Moreover, suppose diagram (\ref{eq:signed}) represents an interval of length 2 in $P_\bot$. In $\oppn{J}{P}_\bot$, the orientation $\alpha_1$ is flipped if and only if $\alpha_2$ is, and the same holds of the orientations $\beta_1$ and $\beta_2$. In any case, the product $\alpha_1\beta_1$ is flipped if and only if the product $\alpha_2\beta_2$ is.

Next, we show that for any constructible molecule $U$ in $P$, the corresponding subset $\oppn{J}{U}$ of $\oppn{J}{P}$ is a constructible molecule, and if $V \sqsubseteq U$, then $\oppn{J}{V} \sqsubseteq \oppn{J}{U}$. If $U$ is a 0-globe, this is obvious. Suppose $U$ is a constructible $n$\nbd molecule; then $\bord{}{\alpha}\oppn{J}{U} = \oppn{J}{\bord{}{\pm\alpha}U}$, depending on whether $n \in J$ or not, and in any case a constructible $(n-1)$\nbd molecule by the inductive hypothesis. If $U$ is atomic, this is sufficient. Otherwise, $U$ splits as $U_1 \cup U_2$. It is straightforward to verify that, if $n \in J$, then $\oppn{J}{U}$ splits as $\oppn{J}{U_2} \cup \oppn{J}{U_1}$, and if $n \notin J$, then $\oppn{J}{U}$ splits as $\oppn{J}{U_1} \cup \oppn{J}{U_2}$.

Since the closure of an element $x$ in $\oppn{J}{P}$ is dual to its closure in $P$, this proves that $\oppn{J}{P}$ is a constructible directed complex.
\end{proof}

There are three particularly interesting instances of $J$\nbd duals:
\begin{enumerate}
	\item $\oppall{P} := \oppn{\mathbb{N}^+}{P}$,
	\item $\opp{P} := \oppn{J_\text{odd}}{P}$ where $J_\text{odd} = \{2n - 1 \,|\, n > 0\}$, and
	\item $\coo{P} := \oppn{J_\text{even}}{P}$ where $J_\text{even} = \{2n \,|\, n > 0\}$.
\end{enumerate}
Clearly, $\oppall{P} = \opp{(\coo{P})} = \coo{(\opp{P})}$. 

\begin{prop} \label{prop:opglob}
Let $P, Q$ be two oriented graded posets. Then:
\begin{enumerate}[label=(\alph*)]
	\item $x \tensor y \mapsto y \tensor x$ defines an isomorphism between $\opp{(P \tensor Q)}$ and $\opp{Q} \tensor \opp{P}$ and between $\coo{(P \tensor Q)}$ and $\coo{Q} \tensor \coo{P}$;
	\item $x \join y \mapsto y \join x$ defines an isomorphism between $\opp{(P \join Q)}$ and $\opp{Q} \join \opp{P}$.
\end{enumerate}
Consequently, $x \tensor y \mapsto x \tensor y$ defines an isomorphism between $\oppall{(P \tensor Q)}$ and $\oppall{P} \tensor \oppall{Q}$.
\end{prop}
\begin{proof}
Let $x \tensor y \in \opp{(P \tensor Q)}$, with $\dmn{x} = n$ and $\dmn{y} = m$. Then $x \tensor y$ covers $x' \tensor y$ in $\opp{(P \tensor Q)}$ with orientation $\alpha$ if and only if $x \tensor y$ covers $x' \tensor y$ in $P \tensor Q$ with orientation $(-)^{n+m}\alpha$, if and only $x$ covers $x'$ in $P$ with orientation $(-)^{n+m}\alpha$, if and only if $x$ covers $x'$ in $\opp{P}$ with orientation $(-)^n(-)^{n+m}\alpha = (-)^m\alpha$, if and only if $y \tensor x$ covers $y \tensor x'$ in $\opp{Q} \tensor \opp{P}$ with orientation $\alpha$. 

Similarly, $x \tensor y$ covers $x \tensor y'$ in $\opp{(P \tensor Q)}$ with orientation $\alpha$ if and only if $x \tensor y$ covers $x \tensor y'$ in $P \tensor Q$ with orientation $(-)^{n+m}\alpha$, if and only if $y$ covers $y'$ in $Q$ with orientation $(-)^n(-)^{n+m}\alpha = (-)^m\alpha$, if and only if $y$ covers $y'$ in $\opp{Q}$ with orientation $(-)^m(-)^m\alpha = \alpha$, if and only if $y \tensor x$ covers $y' \tensor x$ in $\opp{Q} \tensor \opp{P}$ with orientation $\alpha$.

The proofs for $\coo{(P \tensor Q)}$ and $\opp{(P \join Q)}$ are analogous.
\end{proof}

\begin{remark}
There is no isomorphism between $\coo{(P \join Q)}$ and $\coo{Q} \join \coo{P}$, unless $P$ or $Q$ is empty: if $x \in P$ and $y \in Q$ are 0-dimensional, then $x \join y$ covers $y$ in $\coo{(P \join Q)}$ with orientation $+$, but $y \join x$ covers $y$ in $\coo{Q} \join \coo{P}$ with orientation $-$.
\end{remark}

Next, we focus on the extension of these operations to constructible polygraphs.

\begin{dfn}
Let $\globe_+$ be the full subcategory of $\globpos$ containing $\globe$ and the empty oriented graded poset $\emptyset$. An \emph{augmented constructible polygraph} is a presheaf on $\globe_+$. Augmented constructible polygraphs and their morphisms form a category $\cpol_+$.

For any constructible polygraph $X$, the \emph{trivial augmentation} $X_+$ of $X$ is the augmented constructible polygraph extending $X$ by $X(\emptyset) = \{*\}$. This determines a full and faithful functor $(-)_+: \cpol \to \cpol_+$ in the obvious way, making $\cpol$ a reflective subcategory of $\cpol_+$.
\end{dfn}

\begin{cons}
By Proposition \ref{thm:globpostensor}, the lax Gray product restricts to a monoidal structure on $\globe$, and by Proposition \ref{lem:globjoin}, the join restricts to a monoidal structure on $\globe_+$. There is a canonical way of extending a monoidal structure on a small category to a monoidal biclosed structure on its category of presheaves, by Day convolution \cite{day1970closed}. 

For the lax Gray product, the monoidal product of two constructible polygraphs $X, Y$ is defined by the coend
\begin{equation*}
	X \tensor Y := \int^{U, V \in \globe} X(U) \times Y(V) \times \homset{\globe}(-,U \tensor V),
\end{equation*}
the left hom $\homoplax{X}{Y}$ by the end
\begin{equation*}
	\homoplax{X}{Y} := \int_{U \in \globe} \homset{\cat{Set}}(X(U), Y(- \tensor U)),
\end{equation*}
and the right hom $\homlax{X}{Y}$ by the end
\begin{equation*}
	\homlax{X}{Y} := \int_{U \in \globe} \homset{\cat{Set}}(X(U), Y(U \tensor -)).
\end{equation*}
The unit is the Yoneda embedding of $1$.

Similarly, for the join, given two augmented constructible polygraphs $X,Y$, we have
\begin{align*}
	X \jointil Y & := \int^{U, V \in \globe_+} X(U) \times Y(V) \times \homset{\globe_+}(-,U \join V), \\
	\joinl{X}{Y} & := \int_{U \in \globe_+} \homset{\cat{Set}}(X(U), Y(- \join U)), \\
	\joinr{X}{Y} & := \int_{U \in \globe_+} \homset{\cat{Set}}(X(U), Y(U \join -)).
\end{align*}
This defines a monoidal biclosed structure on $\cpol_+$, whose unit is the Yoneda embedding of $\emptyset$. This does \emph{not} reflect onto a monoidal biclosed structure on $\cpol$, because $\cpol$ is not an exponential ideal of $\cpol_+$ in the biclosed sense. On the other hand, $\cpol$ is closed under joins in $\cpol_+$, because if $X(\emptyset) = Y(\emptyset) = \{*\}$, then
\begin{equation*}
	X \jointil Y(\emptyset) = X(\emptyset) \times Y(\emptyset) \times \homset{\globe_+}(\emptyset,\emptyset) \simeq \{*\}.
\end{equation*}
Thus if $r: \cpol_+ \to \cpol$ is the restriction functor, we can define $X \join Y := r(X_+ \jointil Y_+)$ for two constructible polygraphs, and obtain a monoidal structure on $\cpol$.
\end{cons}

\begin{remark}
By restriction to full subcategories of $\globe$ and $\globe_+$, we obtain the category of pre-cubical sets with the lax Gray product, and the category of semi-simplicial sets with the join, as full monoidal subcategories of $(\cpol, \tensor, 1)$ and of $(\cpol, \join, \emptyset)$, respectively.
\end{remark}

\begin{remark}
Because $\emptyset$ is initial in $\cpol$, we have natural inclusions
\begin{equation*}
	\begin{tikzpicture}[baseline={([yshift=-.5ex]current bounding box.center)}]
	\node[scale=1.25] (0) at (-2,0) {$X$};
	\node[scale=1.25] (1) at (-.5,0) {$X \join \emptyset$};
	\node[scale=1.25] (2) at (2,0) {$X \join Y,$};
	\draw[1cinc] (0) to node[auto] {$\sim$} (1);
	\draw[1cinc] (1) to node[auto] {$\mathrm{id}_X \join !$} (2);
\end{tikzpicture}
\end{equation*}
\begin{equation*}
	\begin{tikzpicture}[baseline={([yshift=-.5ex]current bounding box.center)}]
	\node[scale=1.25] (0) at (-2,0) {$Y$};
	\node[scale=1.25] (1) at (-.5,0) {$\emptyset \join Y$};
	\node[scale=1.25] (2) at (2,0) {$X \join Y,$};
	\draw[1cinc] (0) to node[auto] {$\sim$} (1);
	\draw[1cinc] (1) to node[auto] {$! \join \mathrm{id}_Y$} (2);
\end{tikzpicture}
\end{equation*}
for all $X$ and $Y$.
\end{remark}

\begin{cons}
Let $X$ be a constructible polygraph, and $J \subseteq \mathbb{N}^+$. The \emph{$J$\nbd dual} $\oppn{J}{X}$ of $X$ is the constructible polygraph defined by $\oppn{J}{X}(-) := X(\oppn{J}{-})$. For each $J$, this defines an endofunctor $\oppn{J}{-}$ on $\cpol$.

In particular, we write $\oppall{X}$, $\opp{X}$, and $\coo{X}$ for $X(\oppall{(-)})$, $X(\opp{(-)})$, and $X(\coo{(-)})$, respectively.
\end{cons}

\begin{prop}
There exist canonical isomorphisms of constructible polygraphs
\begin{equation*}
	\opp{(X \tensor Y)} \simeq \opp{Y} \tensor \opp{X}, \quad \quad 	\coo{(X \tensor Y)} \simeq \coo{Y} \tensor \coo{X}, \quad \quad \oppall{(X \tensor Y)} \simeq \oppall{X} \tensor \oppall{Y},
\end{equation*}
\begin{equation*}
	\opp{(X \join Y)} \simeq \opp{Y} \join \opp{X},
\end{equation*}
natural in $X$ and $Y$.
\end{prop}
\begin{proof}
We have
\begin{align*}
	\opp{(X \tensor Y)} = & \int^{U, V \in \globe} X(U) \times Y(V) \times \mathrm{Hom}_\globe(\opp{(-)},U \tensor V) = \\
	= & \int^{U, V \in \globe} X(\opp{U}) \times Y(\opp{V}) \times \mathrm{Hom}_\globe(\opp{(-)},\opp{U} \tensor \opp{V})
\end{align*}
because the coend is over all atoms, and any atom $U$ is equal to $\opp{(\opp{U})}$. By Proposition \ref{prop:opglob}, this is isomorphic to 
\begin{align*}
	& \int^{U, V \in \globe} X(\opp{U}) \times Y(\opp{V}) \times \mathrm{Hom}_\globe(\opp{(-)},\opp{(V \tensor U)}) \simeq \\
	\simeq & \int^{U, V \in \globe} \opp{X}(U) \times \opp{Y}(V) \times \mathrm{Hom}_\globe(-,V \tensor U) = \opp{Y} \tensor \opp{X}.
\end{align*}
The other natural isomorphisms are constructed in the same way, using the appropriate parts of Proposition \ref{prop:opglob}.
\end{proof}

\begin{cons} \label{cons:slices}
Like the category of $\omega$\nbd categories with the join defined by Ara and Maltsiniotis \cite{ara2016joint}, $\cpol$ with the join is not biclosed, but it is locally biclosed, in the following sense. For all constructible polygraphs $X, Y$, we have an augmented constructible polygraph $\joinr{X_+}{Y_+}$, and
\begin{align*} 
	\joinr{X_+}{Y_+}(\emptyset) & \simeq \homset{\cpol_+}(\emptyset, \joinr{X_+}{Y_+}) \simeq \\
		& \simeq \homset{\cpol_+}(X_+\jointil\emptyset, Y_+) \simeq \homset{\cpol_+}(X_+, Y_+),
\end{align*}
which is naturally isomorphic to $\homset{\cpol}(X, Y)$ since $\cpol$ is a full subcategory of $\cpol_+$.

For each atom $U$, let $!: \emptyset \incl U$ be the unique inclusion of $\emptyset$ into $U$. For each morphism $f: X \to Y$ of constructible polygraphs, we define a sub-presheaf $\slice{f}{Y}$ of $\joinr{X_+}{Y_+}$ by
\begin{equation*}
	\slice{f}{Y}(U) := \{x \in \joinr{X_+}{Y_+}(U) \,|\, !^*x = f\},
\end{equation*}
where the identity is interpreted via $\joinr{X_+}{Y_+}(\emptyset) \simeq \homset{\cpol}(X, Y)$.

This is well-defined as a sub-presheaf, and $\slice{f}{Y}(\emptyset) = \{f\}$, so we can identify it with a constructible polygraph.

\begin{remark} \label{rmk:slice_diagram}
Via the natural isomorphisms
\begin{equation*}
	\joinr{X_+}{Y_+}(U) \simeq \homset{\cpol_+}(U, \joinr{X_+}{Y_+}) \simeq \homset{\cpol}(X \join U, Y),
\end{equation*}
the elements of $\slice{f}{Y}(U)$ correspond to morphisms $h: X\join U \to Y$ such that 
\begin{equation} \label{eq:slice_diagram}
	\begin{tikzpicture}[baseline={([yshift=-.5ex]current bounding box.center)}]
	\node[scale=1.25] (0) at (-1.25,-1.25) {$X \join U$};
	\node[scale=1.25] (1) at (0,0) {$X$};
	\node[scale=1.25] (2) at (1.25,-1.25) {$Y$};
	\draw[1cincl] (1) to (0);
	\draw[1c] (1) to node[auto] {$f$} (2);
	\draw[1c] (0) to node[auto] {$h$} (2);
\end{tikzpicture}
\end{equation}
commutes.
\end{remark}

Similarly, there is a sub-presheaf $\sliceco{Y}{f}$ of $\joinl{X_+}{Y_+}$, defined by
\begin{equation*}
	\sliceco{Y}{f}(U) := \{x \in \joinl{X_+}{Y_+}(U) \,|\, !^*x = f\}.
\end{equation*}
\end{cons}

\begin{prop} \label{prop:locallybiclosed}
The assignments $(f: X \to Y) \mapsto \slice{f}{Y}$ and $\sliceco{Y}{f}$ of Construction \ref{cons:slices} extend to functors $\slice{X}{\cpol} \to \cpol$. These functors are right adjoints to the functors $\cpol \to \slice{X}{\cpol}$ defined, respectively, by
\begin{equation} \label{eq:join_first_incl}
	f: Y \to Z \quad \quad \mapsto \quad \quad 
	\begin{tikzpicture}[baseline={([yshift=-.5ex]current bounding box.center)}]
	\node[scale=1.25] (0) at (-1.25,-1.25) {$X \join Y$};
	\node[scale=1.25] (1) at (0,0) {$X$};
	\node[scale=1.25] (2) at (1.25,-1.25) {$X \join Z$};
	\draw[1cincl] (1) to (0);
	\draw[1cinc] (1) to (2);
	\draw[1c] (0) to node[auto] {$\mathrm{id}_X \join f$} (2);
	\node[scale=1.25] at (2,-1.4) {,};
\end{tikzpicture}
\end{equation}
\begin{equation*} 
	f: Y \to Z \quad \quad \mapsto \quad \quad 
	\begin{tikzpicture}[baseline={([yshift=-.5ex]current bounding box.center)}]
	\node[scale=1.25] (0) at (-1.25,-1.25) {$Y \join X$};
	\node[scale=1.25] (1) at (0,0) {$X$};
	\node[scale=1.25] (2) at (1.25,-1.25) {$Z \join X$};
	\draw[1cincl] (1) to (0);
	\draw[1cinc] (1) to (2);
	\draw[1c] (0) to node[auto] {$f \join \mathrm{id}_X$} (2);
	\node[scale=1.25] at (2,-1.4) {.};
\end{tikzpicture}
\end{equation*}
\end{prop}
\begin{proof}
Given a morphism of constructible polygraphs $g: Y \to Z$, the morphisms 
\begin{equation*}
	\joinr{X_+}{g}: \joinr{X_+}{Y_+} \to \joinr{X_+}{Z_+}, \quad \quad \quad \joinl{X_+}{g}: \joinl{X_+}{Y_+} \to \joinl{X_+}{Z_+}
\end{equation*}
restrict to morphisms 
\begin{equation*}
	g_*: \slice{f}{Y} \to \slice{(f;g)}{Z}, \quad \quad \quad g_*: \sliceco{Y}{f} \to \sliceco{Z}{(f;g)}.
\end{equation*}
This is clear from the reformulation of Remark \ref{rmk:slice_diagram}: if $(\ref{eq:slice_diagram})$ commutes, so does the post-composition of the triangle with $g$. This gives functoriality of $\slice{f}{Y}$.

The adjunction with $(\ref{eq:join_first_incl})$ is also almost obvious from Remark \ref{rmk:slice_diagram}: a morphism $(X \incl X \join Y) \to (h: X \to Z)$ is a commutative triangle
\begin{equation*}
	\begin{tikzpicture}[baseline={([yshift=-.5ex]current bounding box.center)}]
	\node[scale=1.25] (0) at (-1.25,-1.25) {$X \join Y$};
	\node[scale=1.25] (1) at (0,0) {$X$};
	\node[scale=1.25] (2) at (1.25,-1.25) {$Z$};
	\draw[1cincl] (1) to (0);
	\draw[1c] (1) to node[auto] {$h$} (2);
	\draw[1c] (0) to node[auto] {$k$} (2);
	\node[scale=1.25] at (1.75,-1.4) {,};
\end{tikzpicture}
\end{equation*}
which via the join-right hom adjunction corresponds to a morphism $\widehat{k}: Y_+ \to \joinr{X_+}{Z_+}$ with the property that $\widehat{k}(y) \in \slice{f}{Y}$ for all $y \in Y$.

The case of $\sliceco{Y}{f}$ is completely analogous.
\end{proof}

In \cite{ara2016joint}, the analogue of $\slice{f}{Y}$ for $\omega$\nbd categories is called the \emph{generalised slice} of $Y$ under $f$. Ara and Maltsiniotis reserve the name of generalised slice of $Y$ over $f$ for $\oppall{(\slice{\oppall{f}}{\oppall{Y}})}$, rather than $\sliceco{Y}{f}$.

\section{Realisation as $\omega$-categories} \label{sec:omegacat}

In this section, we look at how we can interpret constructible polygraphs as strict $\omega$\nbd categories. To this aim, we will use Steiner's theory of directed complexes \cite{steiner1993algebra} as a bridge. As we mentioned in Remark \ref{rmk:precomplex}, directed complexes have an underlying structure --- a directed precomplex --- which is only slightly more general than an oriented graded poset, and, similarly to constructible directed complexes, satisfy a property relative to a class of closed subsets: the \emph{molecules}. 

The characteristic property of the molecules is that they are ``globular composites'' of other molecules; compare this with constructible molecules, which can be decomposed as ``mergers''. We will show that constructible molecules are molecules, and derive that constructible directed complexes are directed complexes, which \emph{a posteriori} justifies their name. This will allow us to realise constructible atoms as $\omega$\nbd categories, which we can then extend to realise constructible polygraphs as $\omega$\nbd categories. 

First of all, we need to fix the notation for $\omega$\nbd categories; we will use a mixture of Steiner's notation and the one prevalent in the more recent literature on polygraphs \cite{metayer2008cofibrant,lafont2010folk}.

\begin{dfn}
Let $X$ be an $\omega$\nbd graph, and for all $x \in X_n$ and $k < n$, let
\begin{equation*}
	\bord{k}{\alpha}x = \underbrace{\bord{}{\alpha}(\ldots(\bord{}{\alpha}}_{n-k} x)).
\end{equation*}  
We call the elements $x \in X_n$ the \emph{$n$\nbd cells} of $X$. Given two $n$\nbd cells $x$, $y$ of $X$, and $k < n$, we say that $x$ and $y$ are \emph{$k$\nbd composable}, and write $x \comp{k} y$, if $\bord{k}{+}x = \bord{k}{-}y$.

We write $X_n \comp{k} X_n \subseteq X_n \times X_n$ for the set of pairs of $k$\nbd composable $n$\nbd cells of $X$.
\end{dfn}

\begin{dfn}
A \emph{partial $\omega$\nbd category} is an $\omega$\nbd graph $X$ together with \emph{unit} and \emph{$k$\nbd composition} operations
\begin{equation*}
	\idd{}: X_n \to X_{n+1}, \qquad \cp{k}: X_n \comp{k} X_n \pfun X_n
\end{equation*}
for all $n \in \mathbb{N},$ and $k < n$, where $\idd{}$ is a total function and the $\cp{k}$ are partial functions. For all $k$\nbd cells $x$, and $n > k$, let
\begin{equation*}
	\idd{n}x := \underbrace{(\idd{}\ldots\idd{})}_{n-k} x,
\end{equation*}
an $n$\nbd cell of $X$. The operations are required to satisfy the following conditions:
\begin{enumerate}
	\item for all $n$\nbd cells $x$, and all $k < n$, 
	\begin{align*}
	 	\bord{}{\alpha}(\varepsilon x) & = x, \\
		x \cp{k} \idd{n}(\bord{k}{+}x) & = x = \idd{n}(\bord{k}{-}x) \cp{k} x,
	\end{align*}
	where the two $k$\nbd compositions are always defined;
	\item for all $(n+1)$\nbd cells $x, y$, and all $k < n$, whenever the left-hand side is defined,
	\begin{align*}
		\bord{}{-}(x \cp{n} y) & = \bord{}{-} x, \\
		\bord{}{+}(x \cp{n} y) & = \bord{}{+} y, \\
		\bord{}{\alpha}(x \cp{k} y) & = \bord{}{\alpha} x \cp{k} \bord{}{\alpha} y;
	\end{align*}
	\item for all cells $x, y, x', y'$, and all $n$ and $k < n$, whenever the left-hand side is defined, 
	\begin{align*}
		\idd{}(x \cp{n} y) & = \idd{}x \cp{n} \idd{}y; \\
		(x \cp{n} x') \cp{k} (y \cp{n} y') & = (x \cp{k} y) \cp{n} (x' \cp{k} y');
	\end{align*}
	\item for all cells $x, y, z$, and all $n$, whenever either side is defined,
	\begin{equation*}
		(x \cp{n} y) \cp{n} z = x \cp{n} (y \cp{n} z).
	\end{equation*}
\end{enumerate} 
A partial $\omega$\nbd category is an \emph{$\omega$\nbd category} if the $\cp{k}$ are total functions.

A \emph{functor} of partial $\omega$\nbd categories is a morphism of the underlying $\omega$\nbd graphs that commutes with units and compositions. A functor is an \emph{inclusion} if it is injective on cells of each dimension. Partial $\omega$\nbd categories and functors form a category $\pomegacat$, with a full subcategory $\omegacat$ on $\omega$\nbd categories.

The inclusion of $\omegacat$ into $\pomegacat$ has a left adjoint $(-)^*: \pomegacat \to \omegacat$. Given a partial $\omega$\nbd category $X$, we call $X^*$ the $\omega$\nbd category \emph{generated} by $X$.
\end{dfn}

\begin{dfn}
Let $X$ be a partial $\omega$\nbd category, $n \in \mathbb{N}$, and $x \in X_n$. We define a number $\dmn{x}$, the \emph{dimension} of $x$, by induction on $n$:
\begin{itemize}
	\item if $x \in X_0$, $\dmn{x} := 0$;
	\item if $x \in X_n$, if $x = \idd{}y$ for some $y \in X_{n-1}$, then $\dmn{x} := \dmn{y}$, else $\dmn{x} := n$.
\end{itemize}
\end{dfn}

\begin{remark}
Given a cell $x$ in a partial $\omega$\nbd category, when we say that ``$x$ is an $\emph{n}$\nbd cell'' we mean $x \in X_n$, and when we say that ``$x$ is $n$\nbd dimensional'' we mean $\dmn{x} = n$.
\end{remark}

\begin{remark}
Steiner defines a partial $\omega$\nbd category as a single set $X$ together with boundary operations $\bord{n}{\alpha}: X \to X$ and partial composition operations $\cp{n}: X \times X \pfun X$ for all $n$. 

We recover this picture from our definition by taking $X := (\coprod_n X_n)/\!\!\sim$, where $\sim$ is the equivalence relation generated by $x \sim \idd{n}x$ for all $n$; notice that $X$ is isomorphic to $\coprod_n \{x \in X_n \,|\, \dmn{x} = n\}$ as a set.

For all cells $x$ and $k \in \mathbb{N}$, there is some $n > k$ such that $\bord{k}{\alpha}(\idd{n}x)$ is defined, and the result belongs to the same equivalence class for all such $n$, so the boundary operations descend to $X$. Similarly, given two cells $x$ and $y$ belonging to equivalence classes $[x]$, $[y]$, we let $[x] \cp{k} [y]$ be defined when $\idd{n}x \cp{k} \idd{n}y$ is defined for some $n > k$, and in that case be equal to the latter's equivalence class.
\end{remark}

\begin{cons}
Let $X$ be a partial $\omega$\nbd category. The \emph{$J$\nbd dual} $\oppn{J}{X}$ of $X$ is the partial $\omega$\nbd category with an $n$\nbd cell $\oppn{J}{x}$ for each $n$\nbd cell $x$ of $X$, and
\begin{equation*}
	\bord{}{\alpha}(\oppn{J}{x}) := \begin{cases} \oppn{J}{\bord{}{-\alpha}x} & \text{if } \dmn{x} \in J, \\
	\oppn{J}{\bord{}{\alpha}x} & \text{if } \dmn{x} \notin J, \end{cases}
\end{equation*}
\begin{equation*}
	\idd{}(\oppn{J}{x}) := \oppn{J}{\idd{}x}, \quad \quad \oppn{J}{x} \cp{k} \oppn{J}{y} := \begin{cases} \oppn{J}{y \cp{k} x} & \text{if } k+1 \in J, \\
	\oppn{J}{x \cp{k} y} & \text{if } k+1 \notin J, \end{cases}
\end{equation*}
where $\oppn{J}{x} \cp{k} \oppn{J}{y}$ is defined whenever $y \cp{k} x$ is defined and $k+1 \in J$, or $x \cp{k} y$ is defined and $k+1 \notin J$.
\end{cons}

\begin{dfn}
Let $U_1, U_2 \subseteq P$ be closed subsets of an oriented graded poset. If $U_1 \cap U_2 = \bord{n}{+}U_1 = \bord{n}{-}U_2$, let
\begin{equation*}
	U_1 \cp{n} U_2 := U_1 \cup U_2;
\end{equation*}
this defines partial $n$\nbd composition operations on the closed subsets of $P$, for all $n$. 

Let $P$ be an oriented graded poset. For each $n \in \mathbb{N}$, we define a family $\molec{n}{P}$ of closed subsets of $P$, the \emph{$n$\nbd molecules} of $P$, together with a partial order $\submol$ on each $\molec{n}{P}$, to be read ``is a submolecule of''.

Let $U \subseteq P$ be closed. Then $U \in \molec{n}{P}$ if and only if $\dmn{U} \leq n$, and, inductively on proper subsets of $U$, either
\begin{itemize}
	\item $U$ has a greatest element, in which case we call it an \emph{atom}, or
	\item there exist $n$\nbd molecules $U_1, U_2$ properly contained in $U$, and $k < n$ such that $U_1 \cap U_2 = \bord{k}{+}U_1 = \bord{k}{-}U_2$, and $U = U_1 \cp{k} U_2$.
\end{itemize}
We define $\submol$ to be the smallest partial order relation on $\molec{n}{P}$ such that $U_1, U_2 \submol U$ for all triples $U, U_1, U_2$ in the latter situation.
\end{dfn}

We will call a decomposition $U = U_1 \cp{k} U_2$ \emph{proper} if $U_1$ and $U_2$ are properly contained in $U$.

\begin{exm} 
All $0$\nbd molecules are atoms: $\molec{0}{P} = \{\{x\} \,|\, \dmn{x} = 0\}$. If $P$ is an oriented thin poset, 0-molecules and 1-molecules are constructible: in fact, $\molec{1}{P}$ is equal to $\glob{0}{P} + \glob{1}{P}$. However, there are 2-molecules that are not constructible, such as the ``whiskering'' of a 2-globe with a 1-globe:
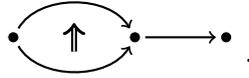
\begin{equation*}
\begin{tikzpicture}[baseline={([yshift=-.5ex]current bounding box.center)}, scale=.8]
	\node[0c] (0) at (-1, 0) {};
	\node[0c] (1) at (1,0) {};
	\node[0c] (2) at (2.5,0) {};
	\draw[1c,out=60,in=120] (0) to (1);
	\draw[1c,out=-60,in=-120] (0) to (1);
	\draw[1c] (1) to (2);
	\draw[2c] (0,-.5) to (0,.5);
	\node[scale=1.25] at (2.875,-.375) {.};
\end{tikzpicture}

\end{equation*}
We will prove that, on the other hand, all constructible molecules are molecules.
\end{exm}

\begin{remark} \label{rmk:molec_ndim}
Clearly, $\molec{n-1}{P} \subseteq \molec{n}{P}$ for all $n > 0$, and the inclusion is order-preserving. Moreover, if $U$ is $n$\nbd dimensional and a molecule, then $U$ is an $n$\nbd molecule. This is obvious if $U$ is an atom. Otherwise, if $U$ decomposes as $U_1 \cp{k} U_2$, where $U_1$ and $U_2$ are $n$\nbd molecules, if $k < n$ then $U$ is an $n$\nbd molecule, and if $k \geq n$, then $U_1 = \bord{k}{+}U_1 = \bord{k}{-}U_2 = U_2$, so $U = U_1 = U_2$, and again $U$ is an $n$\nbd molecule. Therefore $\molec{n}{P} \setminus \molec{n-1}{P}$ only contains $n$\nbd dimensional subsets.
\end{remark}

\begin{dfn} \label{dfn:directedcomplex}
A \emph{directed complex} is an oriented graded poset $P$ such that, for all $x \in P$, with $\dmn{x} = n > 0$, and all $\alpha, \beta$,
\begin{enumerate}
	\item $\bord{}{\alpha}x$ is a molecule, and
	\item $\bord{}{\alpha}(\bord{}{\beta}x) = \bord{n-2}{\alpha}x$.
\end{enumerate}
\end{dfn}

\begin{remark}
Compared to Steiner's definition based on directed precomplexes, ours has in addition the built-in constraint that $\sbord{}{+}x$ and $\sbord{}{-}x$ are disjoint for all $x$.
\end{remark}

\begin{dfn}
A closed $n$\nbd dimensional subset $U$ of an oriented graded poset is \emph{globelike} if $\bord{j}{\alpha}(\bord{k}{\beta}U) = \bord{j}{\alpha}U$ holds for all $\alpha$ and $j < k < n$.
\end{dfn}

\begin{prop} \label{prop:moleculeglobelike}
In a constructible directed complex, all molecules are globelike.
\end{prop}
\begin{proof}
Let $U$ be a molecule in a constructible directed complex. By induction on the definition of a molecule: if $U$ is an atom, then it is constructible, hence globelike by Proposition \ref{prop:globelike}. If $U$ decomposes as $U_1 \cp{k} U_2$, assuming that $U_1$ and $U_2$ are globelike, we obtain that $U$ is globelike as in \cite[Lemma 3.4]{steiner1993algebra}.
\end{proof}

\begin{remark}
The statement of \cite[Lemma 3.4]{steiner1993algebra} assumes that we are in a directed complex, but only the properties of directed precomplexes (or, \emph{a fortiori}, oriented graded posets) are actually used.
\end{remark}

\begin{cor} \label{cor:globelikecor}
Let $U_1, U_2$ be two $n$\nbd molecules in a constructible directed complex, and suppose $U_1 \cap U_2 = \bord{k}{+}U_1 = \bord{k}{-}U_2$. 
\begin{enumerate}[label=(\alph*)]
	\item If $k = n-1$, then $\bord{n-1}{-}(U_1 \cp{k} U_2) = \bord{n-1}{-}U_1$ and $\bord{n-1}{+}(U_1 \cp{k} U_2) = \bord{n-1}{+}U_2$.
	\item If $k < n-1$, then $\bord{n-1}{\alpha}U_1 \cp{k} \bord{n-1}{\alpha}U_2$ is defined and equal to $\bord{n-1}{\alpha}(U_1 \cp{k} U_2)$. 
\end{enumerate}
\end{cor}
\begin{proof}
The first point holds more generally for closed subsets of an oriented graded poset, see \cite[Proposition 3.1.(vi)]{steiner1993algebra}. 

The second point is also proved in [Lemma 3.4, \emph{ibid.}] under the assumption that $U_1$ and $U_2$ be globelike, which follows for us from Proposition \ref{prop:moleculeglobelike}.
\end{proof}

We have the following fact about $n$\nbd dimensional molecules in a constructible directed complex: their $k$\nbd boundaries, for $k \leq n$, are always of the ``right dimension''.
\begin{prop} \label{prop:molecdim}
Let $U$ be an $n$\nbd dimensional molecule in a constructible directed complex. For all $\alpha \in \{+,-\}$ and $k < n$, $\bord{k}{\alpha}U$ is $k$\nbd dimensional, and there exists $x \in \sbord{k}{\alpha}U$ such that $x \notin \sbord{k}{-\alpha}U$.
\end{prop}
\begin{proof}
We proceed by induction on the dimension and proper subsets of $U$. If $U$ is 0-dimensional, there is nothing to prove. If $U$ is an atom of any dimension, then $\bord{k}{\alpha}U$ is a constructible $k$\nbd molecule, and by Lemma \ref{lem:basic} applied to the constructible $(k+1)$\nbd molecule $\bord{k+1}{\beta}U$, we have that $\sbord{k}{-}U$ and $\sbord{k}{+}U$ are disjoint and both inhabited. So any $x \in \sbord{k}{\alpha}U$ satisfies the statement.

Suppose $U$ is $n$\nbd dimensional and has a proper decomposition $U_1 \cp{m} U_2$, where at least one of $U_1$ and $U_2$ is $n$\nbd dimensional. If $m = n-1$, then $\bord{n-1}{-}U = \bord{n-1}{-}U_1$ and $\bord{n-1}{+}U = \bord{n-1}{+}U_2$. In this case, both $U_1$ and $U_2$ are $n$\nbd dimensional (otherwise the decomposition is not proper), so by the inductive hypothesis $\bord{n-1}{\alpha}U$ is $(n-1)$\nbd dimensional. Picking any $x \in \sbord{n-1}{-}U = \sbord{n-1}{-}U_1$ such that $x \notin \sbord{n-1}{+}U_1 = U_1 \cap U_2$, necessarily $x \notin \sbord{n-1}{+}U_2 = \sbord{n-1}{+}U$; similarly we find $y \in \sbord{n-1}{+}U$ such that $y \notin \sbord{n-1}{-}U$. For $k < n-1$, we have $\sbord{k}{\alpha}U = \sbord{k}{\alpha}(\bord{n-1}{-}U) = \sbord{k}{\alpha}U_1$, so we can use the inductive hypothesis on $U_1$ to find suitable $k$\nbd dimensional elements.

If $m < n -1$, then $\sbord{k}{\alpha}U = \sbord{k}{\alpha}(\bord{n-1}{\alpha}U_1\cp{m}\bord{n-1}{\alpha}U_2)$. By the inductive hypothesis, either $\bord{n-1}{\alpha}U_1$ or $\bord{n-1}{\alpha}U_2$ is $(n-1)$\nbd dimensional; therefore, $\bord{n-1}{\alpha}U_1\cp{m}\bord{n-1}{\alpha}U_2$ is an $(n-1)$\nbd dimensional molecule, and again we can apply the inductive hypothesis.
\end{proof}

The key step in our proof is provided by the following technical lemma.

\begin{lem} \label{lem:moleculesubs}
Let $U$ be an $n$\nbd molecule and $V$ a constructible $(n+1)$\nbd molecule in a constructible directed complex. Suppose that $V$ and $\bord{}{\alpha}V$ are molecules, and that $U \cap V = \bord{}{\alpha}V \submol U$. Then $U \cup V$ is an $(n+1)$\nbd molecule and $V \submol U \cup V$.
\end{lem}
\begin{proof}
By definition of the submolecule relation, either $U = \bord{}{\alpha}V$, or $U = U_1 \cp{m} U_2$ for some $n$\nbd molecules $U_1, U_2 \subset V$ and $m < n$, and $\bord{}{\alpha}V \submol U_i$ for some $i \in \{1,2\}$. In the first case, $U \cup V = V$, which is an $(n+1)$\nbd molecule by assumption. Moreover, we trivially have $V \submol U \cup V$ and $\bord{k}{\alpha}V = \bord{k}{\alpha}(U \cup V)$ for all $k < n$.

In the second case, suppose without loss of generality that $\bord{}{\alpha}V \submol U_2$. We can assume, inductively, that $U_2 \cup V$ is an $(n+1)$\nbd molecule, that $V \submol U_2 \cup V$, and that $\bord{k}{\alpha}V = \bord{k}{\alpha}(U_2 \cup V)$ for all $k < n$. Then 
\begin{equation*}
	U_1 \cap (U_2 \cup V) = U_1 \cap U_2 = \bord{m}{+}U_1 = \bord{m}{-}U_2 = \bord{m}{-}(U_2 \cup V).
\end{equation*}
Hence, $U \cup V = U_1 \cp{m} (U_2 \cup V)$ is a molecule, $V \submol U_2 \cup V \submol U \cup V$, and by Corollary \ref{cor:globelikecor}, $\bord{n-1}{\alpha}(U \cup V)$ depends only on $\bord{n-1}{\alpha}U_1$ and $\bord{n-1}{\alpha}(U_2 \cup V) = \bord{n-1}{\alpha}U_2$. Because molecules are globelike, it follows that $\bord{k}{\alpha}V = \bord{k}{\alpha}(U \cup V)$ for all $k < n$, which completes the inductive step.
\end{proof}

\begin{thm} \label{thm:globmolec}
Let $U$ be a constructible $n$\nbd molecule in an oriented thin poset. Then $U$ is an $n$\nbd molecule. If $V$ is a constructible submolecule of $U$, it is also a submolecule of $U$.
\end{thm}
\begin{remark}
In other words, if $P$ is an oriented thin poset, there are inclusions of posets $(\glob{n}{P}, \sqsubseteq) \hookrightarrow (\molec{n}{P}, \submol)$ for all $n$.
\end{remark}
\begin{proof}
We proceed by induction on the dimension and number of maximal elements of $U$. If $U$ is an atom of any dimension, it is obviously a molecule. Moreover, $U$ has no proper constructible submolecules, and it is trivially a submolecule of itself.

Suppose $U$ is a non-atomic constructible $(n+1)$\nbd molecule, splitting as $U_1 \cup U_2$; by the inductive hypothesis, $U_1$ and $U_2$ are $(n+1)$\nbd molecules. Let
\begin{equation*}
	\tilde{U}_1 := \bord{}{-}U \cup U_1, \quad \quad \quad \tilde{U}_2 := \bord{}{+}U \cup U_2;
\end{equation*}
then $\tilde{U}_1 \cap \tilde{U}_2 = \bord{}{+}U_1 \cup \bord{}{-}U_2 = \bord{n}{+}\tilde{U}_1 = \bord{n}{-}\tilde{U}_2$, so $U = \tilde{U}_1 \cp{n} \tilde{U}_2$. 

By the inductive hypothesis, both $\bord{}{-}U$ and $\bord{}{-}U_1$ are $n$\nbd molecules, and because $\bord{}{-}U_1 \sqsubseteq \bord{}{-}U$, also $\bord{}{-}U_1 \submol \bord{}{-}U$. Since $U_1 \cap \bord{}{-}U = \bord{}{-}U_1$, we fall under the hypotheses of Lemma \ref{lem:moleculesubs}: we deduce that $\bord{}{-}U \cup U_1 = \tilde{U}_1$ is an $(n+1)$\nbd molecule, and that $U_1 \submol \tilde{U}_1$. 

Similarly, we obtain that $\tilde{U}_2$ is an $(n+1)$\nbd molecule, and $U_2 \submol \tilde{U}_2$. It follows that $U$ is an $(n+1)$\nbd molecule and that $U_1, U_2 \submol U$. This completes the proof.
\end{proof}

\begin{cor}
Constructible directed complexes are directed complexes.
\end{cor}
\begin{proof}
Let $P$ be a constructible directed complex, and $x \in P^{(n)}$, with $n > 0$. Then $\bord{}{\alpha}x$ is a constructible molecule, hence a molecule by Theorem \ref{thm:globmolec}, and $\bord{}{\alpha}(\bord{}{\beta}x) = \bord{n-2}{\alpha}x$ is an instance of Proposition \ref{prop:globelike}.
\end{proof}

Before moving on, we briefly discuss the loop-freeness conditions that Steiner considered for directed complexes, in order to show that constructible directed complexes do not fall into one of the loop-free classes.

\begin{dfn} \label{dfn:loopfree}
Given a constructible directed complex $P$, for each $n \in \mathbb{N}$, let $\loopd{n}{P}$ be the bipartite directed graph with $P$ as set of vertices, and an edge $y \to x$ if and only if
\begin{itemize}
	\item $\dmn{y} \leq n$, $\dmn{x} > n$, and $y \in \bord{n}{-}x \setminus \bord{n-1}{}x$, or
	\item $\dmn{y} > n$, $\dmn{x} \leq n$, and $x \in \bord{n}{+}y \setminus \bord{n-1}{}y$.
\end{itemize}
We say that $P$ is \emph{loop-free} if $\loopd{n}{P}$ is acyclic for all $n$. We say that $P$ is \emph{totally loop-free} if $\hasso{P}$ is acyclic.
\end{dfn}

Being totally loop-free is a quite natural condition: it applies to globes, cubes and oriented simplices, and it is preserved by suspensions, lax Gray products, and joins \cite[Theorem 2.19]{steiner1993algebra}. Totally loop-free constructible directed complex are loop-free [Proposition 2.15, \emph{ibid.}], which is a strictly weaker condition (and seemingly less natural, as it is not preserved by lax Gray products or joins), that nevertheless is sufficient to prove a number of results. 

\begin{remark}
This characterisation of total loop-freeness is inspired by acyclic matchings in discrete Morse theory, see \cite[Chapter 11]{kozlov2008combinatorial}. In a totally loop-free constructible directed complex, the ``flow'' on cells described in Remark \ref{exm:standard} never returns to a cell once it has left it. 
\end{remark}

\begin{exm} \label{exm:nonloopfree}
The following is an example of an constructible 3\nbd atom $U$ that is neither loop-free nor totally loop-free, based on \cite[Example 3.12]{power1991pasting}:
\begin{equation*}
\begin{tikzpicture}[baseline={([yshift=-.5ex]current bounding box.center)},scale=.7]
\begin{scope}
	\node[0c] (a) at (-1.75,0) {};
	\node[0c] (b) at (.5,-1.25) [label=below:$x_2$] {};
	\node[0c] (c) at (-.5,1.25) [label=above:$x_1$] {};
	\node[0c] (d) at (1.75,0) {};
	\draw[1c, out=-60, in=180] (a) to (b);
	\draw[1c, out=0, in=120] (c) to (d);
	\draw[1c, out=15, in=-105] (b) to (d);
	\draw[1c, out=75, in=-165] (a) to (c);
	\draw[1c] (c) to node[auto] {$\!\!y_1$} (b);
	\draw[2c] (-.75,-.7) to (-.75,.9);
	\draw[2c] (.75,-.9) to (.75,.7);
	\node[scale=1.25] at (2.25,-.9) {$,$};
	\node[scale=1.25] at (-3,0) {$\bord{}{-}U : $};
\end{scope}
\begin{scope}[shift={(7,0)}]
	\node[0c] (a) at (-1.75,0) {};
	\node[0c] (b) at (-.5,-1.25) [label=below:$x_2$] {};
	\node[0c] (c) at (.5,1.25) [label=above:$x_1$] {};
	\node[0c] (d) at (1.75,0) {};
	\draw[1c, out=-75, in=165] (a) to (b);
	\draw[1c, out=-15,in=105] (c) to (d);
	\draw[1c, out=0, in=-120] (b) to (d);
	\draw[1c, out=60,in=180] (a) to (c);
	\draw[1c] (b) to node[auto] {$y_2\!\!$} (c);
	\draw[2c] (.75,-.7) to (.75,.9);
	\draw[2c] (-.75,-.9) to (-.75,.7);
	\node[scale=1.25] at (-3,0) {$\bord{}{+}U : $};
	\node[scale=1.25] at (2.25,-.9) {$.$};
\end{scope}
\end{tikzpicture}
\end{equation*}
Indeed, $x_1 \to y_1 \to x_2 \to y_2 \to x_1$ is a loop in $\loopd{0}{U}$ and in $\hasso{U}$. Thus, not all constructible directed complex are loop-free.
\end{exm}

\begin{remark} 
Example \ref{exm:nonloopfree} is also a counterexample to the restriction of Proposition \ref{prop:atomicnglobe} to loop-free or totally loop-free constructible molecules: $\bord{}{-}U$ and $\bord{}{+}U$ are individually totally loop-free. This is one reason why we do not consider loop-free atoms to be an adequate shape category for rewriting theory: given two composable diagrams with the same boundary, we want to be able to rewrite one into the other, which is modelled by a higher-dimensional cell.
\end{remark}

On the other hand, constructible directed complexes do satisfy a restricted loop-freeness condition, in that there can be no looping via top-dimensional elements in a molecule; this makes them less general than directed complexes, where looping is possible at every level.

\begin{prop} \label{prop:loopd-acyclic}
Let $U$ be an $(n+1)$\nbd dimensional molecule in a constructible directed complex. Then $\loopd{n}{U}$ is acyclic.
\end{prop}
\begin{proof}
We proceed by induction on proper subsets of $U$. If $U$ is an atom, the statement is true because $\sbord{}{-}U$ and $\sbord{}{+}U$ are disjoint. Otherwise, suppose $U$ has a proper decomposition $U_1 \cp{k} U_2$, where necessarily $k \leq n$; it is easy to check that $\loopd{n}{U}$ contains both $\loopd{n}{U_1}$ and $\loopd{n}{U_2}$.  

If $k < n$, then $U_1 \cap U_2 = \bord{k}{+}U_1 = \bord{k}{-}U_2$ is contained in the union of the $\bord{n-1}{}x$ for $x \in U$, so it cannot contain any element of $\bord{n}{\alpha}x\setminus\bord{n-1}{}x$. It follows that there can be no edges connecting vertices of $\loopd{n}{U_1}$ to vertices of $\loopd{n}{U_2}$, and any path stays within one or the other. 

If $k = n$, finite paths in $\loopd{n}{U}$ are either contained in $\loopd{n}{U_1}$ or in $\loopd{n}{U_2}$, or they reach an element $x \in \bord{}{+}U_1 = \bord{}{-}U_2 = U_1 \cap U_2$ from an $(n+1)$\nbd dimensional element $y \in U_1$, before entering an $(n+1)$\nbd dimensional element $y' \in U_2$. But $y$ can only be reached from an element of $\bord{}{-}y\setminus\bord{n-1}{}y$, which does not belong to $\bord{}{+}U_1$, and from $y'$ the path can only reach an element of $\bord{}{+}y'\setminus\bord{n-1}{}y'$, which does not belong to $\bord{}{-}U_2$. Hence, any path in $\loopd{n}{U}$ consists of a path in $\loopd{n}{U_1}$, followed by a path in $\loopd{n}{U_2}$; by the inductive hypothesis, $\loopd{n}{U}$ is acyclic.
\end{proof}

\begin{cons}
Let $P$ be a constructible directed complex. By Proposition \ref{prop:moleculeglobelike}, the diagram
\begin{equation*}
\begin{tikzpicture}
	\node[scale=1.25] (0) at (0,0) {$\molec{0}{P}$};
	\node[scale=1.25] (1) at (2.5,0) {$\molec{1}{P}$};
	\node[scale=1.25] (2) at (5,0) {$\ldots$};
	\node[scale=1.25] (3) at (7.5,0) {$\molec{n}{P}$};
	\node[scale=1.25] (4) at (10,0) {$\ldots$};
	\draw[1c] (1.west |- 0,.15) to node[auto,swap] {$\bord{}{+}$} (0.east |- 0,.15);
	\draw[1c] (1.west |- 0,-.15) to node[auto] {$\bord{}{-}$} (0.east |- 0,-.15);
	\draw[1c] (2.west |- 0,.15) to node[auto,swap] {$\bord{}{+}$} (1.east |- 0,.15);
	\draw[1c] (2.west |- 0,-.15) to node[auto] {$\bord{}{-}$} (1.east |- 0,-.15);
	\draw[1c] (3.west |- 0,.15) to node[auto,swap] {$\bord{}{+}$} (2.east |- 0,.15);
	\draw[1c] (3.west |- 0,-.15) to node[auto] {$\bord{}{-}$} (2.east |- 0,-.15);
	\draw[1c] (4.west |- 0,.15) to node[auto,swap] {$\bord{}{+}$} (3.east |- 0,.15);
	\draw[1c] (4.west |- 0,-.15) to node[auto] {$\bord{}{-}$} (3.east |- 0,-.15);
\end{tikzpicture}
\end{equation*}
is an $\omega$\nbd graph $\molec{}{P}$. 

For any $n$\nbd molecule $U$, let $\idd{}(U) := U$ as an $(n+1)$\nbd molecule, and for any pair $U_1, U_2$ of $n$\nbd molecules, let $U_1 \cp{k} U_2$ be defined if and only if $U_1 \cap U_2 = \bord{k}{+}U_1 = \bord{k}{-}U_2$, and in that case be equal to $U_1 \cup U_2$. By \cite[Proposition 2.9]{steiner1993algebra}, this makes $\molec{}{P}$ a partial $\omega$\nbd category. 
\end{cons}

For any partial $\omega$\nbd category $X$, if $X^*$ is the $\omega$\nbd category generated by $X$, we obtain a functor $X \to X^*$ from the unit of the adjunction between $\pomegacat$ and $\omegacat$. 

\begin{prop}\emph{\cite[Theorem 2.13]{steiner1993algebra}} \label{thm:steiner213} Let $P$ be a directed complex. Then the unit $\molec{}P \to (\molec{}P)^*$ is an inclusion of partial $\omega$\nbd categories.
\end{prop} 

The fact that $\molec{}P \to (\molec{}P)^*$ is an inclusion allows us to identify the cells in its image with unique molecules of $P$.

\begin{prop}
The assignment $P \mapsto (\molec{}{P})^*$ for constructible directed complexes extends to a faithful functor $(\molec{}{-})^*: \globpos \to \omegacat$.
\end{prop}
\begin{proof}
The $\omega$\nbd category $(\molec{}{P})^*$ is composition-generated by the atoms of $P$, and an inclusion $\imath: P \hookrightarrow Q$  determines a sequence of functions $\clos{\{x\}} \mapsto \clos{\{\imath(x)\}}$ sending generators of $(\molec{}{P})^*$ to generators of $(\molec{}{Q})^*$. These are compatible with boundaries and compositions (because they are between the partial $\omega$\nbd categories $\molec{}{P}$ and $\molec{}{Q}$), so they determine a unique map of $\omega$\nbd categories. Faithfulness is a consequence of Proposition \ref{thm:steiner213}.
\end{proof}

Restricting to the subcategory $\globe$, we obtain a functor $\hatto{-}: \globe \to \omegacat$. Now, $\globe$ is a small category (there are countably many constructible atoms, and only finitely many morphisms between each two), $\cpol$ is locally small, and $\omegacat$ is cocomplete \cite{batanin1998monoidal}. By \cite[Theorem X.4.1-2]{maclane1971cats}, the left Kan extension of $(\molec{}{-})^*$ along the Yoneda embedding $\globe \hookrightarrow \cpol$ exists, and is given, for each constructible polygraph $X$, by the coend
\begin{equation*} 
	X \mapsto \hatto{X} := \int^{U \in \globe} \hatto{U} \times X(U).
\end{equation*}
This functor has a right adjoint $P: \omegacat \to \cpol$, defined by
\begin{equation*}
	PX(-) := \mathrm{Hom}_{\omegacat}(\hatto{-},X)
\end{equation*}
for each $\omega$\nbd category $X$. 

The lax Gray product of $\omega$\nbd categories was defined by Al-Agl, Brown, and Steiner \cite{al2002multiple}, and later studied by Crans \cite{crans1995pasting}. A definition based on augmented directed complexes was given by Steiner in \cite{steiner2004omega}; the proof of its correctness was completed by Ara and Maltsiniotis \cite{ara2016joint}, who also defined the join of $\omega$\nbd categories. 

Let $(\omegacat, \tensor, 1)$ and $(\omegacat, \join, \emptyset)$ denote the monoidal structures on $\omegacat$ corresponding, respectively, to the lax Gray product and the join. It seems almost inevitable that the following should hold.

\begin{conj} \label{conj:monoidality}
The functor $\hatto{-}: \cpol \to \omegacat$ is monoidal both from $(\cpol, \tensor, 1)$ to $(\omegacat, \tensor, 1)$ and from $(\cpol, \join, \emptyset)$ to $(\omegacat, \join, \emptyset)$.
\end{conj} 
Proving this, however, seems to require developing the relation between our theory and that of Steiner, Ara, and Maltsiniotis further than we intended. We give a few partial results, and leave a full proof to future work.

\begin{lem} \label{lem:restrictmonoidal}
The following hold:
\begin{enumerate}[label=(\alph*)]
	\item $\hatto{-}: (\cpol, \tensor, 1) \to (\omegacat, \tensor, 1)$ is a monoidal functor if and only if its restriction to $\globe$ is monoidal;
	\item $\hatto{-}: (\cpol, \join, \emptyset) \to (\omegacat, \join, \emptyset)$ is a monoidal functor if and only if its restriction to $\globe_+$ is monoidal.
\end{enumerate} 
\end{lem}
\begin{proof}
For the first point: $\globe$ is a dense subcategory of $\cpol$, and the lax Gray product, being part of a biclosed structure, preserves colimits separately in each argument. Thus, for any pair of constructible polygraphs $X, Y$,
\begin{equation*}
	X \tensor Y \simeq \colim{U \to X}U \tensor \colim{V \to Y}V \simeq \colim{U \to X, V \to Y}(U \tensor V),
\end{equation*}
where $U, V$ range over atoms. Moreover, $\hatto{-}$, being a left adjoint, preserves colimits, and by \cite[Th\'eor\`eme A.15]{ara2016joint} the lax Gray product of $\omega$\nbd categories also preserves colimits separately in each variable: therefore, if $\hatto{-}: (\globe, \tensor, 1) \to (\omegacat, \tensor, 1)$ is monoidal,
\begin{equation*}
	\hatto{(X \tensor Y)} \simeq \colim{\hatto{(U \to X)}, \hatto{(V \to Y)}}\hatto{(U \tensor V)} \simeq \colim{\hatto{(U \to X)}, \hatto{(V \to Y)}}(\hatto{U} \tensor \hatto{V}) \simeq \hatto{X} \tensor \hatto{Y}.
\end{equation*}
and all isomorphisms are natural in $X$ and $Y$. The other direction is obvious.

The proof of the second point is analogous, using the fact that any constructible polygraph is the colimit of a \emph{connected} diagram in $\globe_+$ (because the latter has an initial object), and that joins of $\omega$\nbd categories preserve connected colimits separately in each argument \cite[Th\'eor\`eme 7.28]{ara2016joint}.
\end{proof}

Let $\globe^\mathrm{tlf}$ and $\globe^\mathrm{tlf}_+$ be the full subcategories of $\globe$ and $\globe_+$ on totally loop-free constructible atoms (Definition \ref{dfn:loopfree}). By \cite[Theorem 2.19]{steiner1993algebra}, they form monoidal subcategories of $(\globe, \tensor, 1)$ and of $(\globe_+, \join, \emptyset)$, respectively.

\begin{lem}
The restriction of $\hatto{-}$ to $\globe^\mathrm{tlf}$ is monoidal from $(\globe^\mathrm{tlf}_+, \tensor, 1)$ to $(\omegacat, \tensor, 1)$ and from $(\globe^\mathrm{tlf}_+, \join, \emptyset)$ to $(\omegacat, \join, \emptyset)$.
\end{lem}
\begin{proof}[Sketch of the proof]
If $U$ is a totally loop-free constructible atom, then the augmented directed complex $KU$ (Definition \ref{cons:adc}) is a strong Steiner complex (\emph{complexe de Steiner fort}) with basis $U$, in the sense of \cite[2.10]{ara2016joint}. Strong Steiner complexes form a full monoidal subcategory $\cat{St}_\mathrm{f}$ of $\adc$ both with the tensor product and the join.

There is a functor $\nu: \adc \to \omegacat$, described in \cite[Definition 2.8]{steiner2004omega}, which restricted to $\cat{St}_\mathrm{f}$ is full and faithful, and exhibits $(\cat{St}_\mathrm{f}, \otimes, I)$ and $(\cat{St}_\mathrm{f}, \join, 0)$ as full monoidal subcategories of $(\omegacat, \tensor, 1)$ and of $(\omegacat, \join, \emptyset)$, respectively \cite[Th\'eor\`eme A.15 and Th\'eor\`eme 7.28]{ara2016joint}. 

Moreover, by \cite[Theorem 6.1]{steiner2004omega}, if $K$ is a strong Steiner complex with basis $B$, the $\omega$\nbd category $\nu K$ is a polygraph, whose generators correspond to elements of $B$. We can use this and Proposition \ref{prop:molec_hatto_iso} to establish an isomorphism between $(\molec{}{U})^*$ and $\nu K(U)$, natural in $U$, for all totally loop-free molecules $U$. 

It follows that $\hatto{-}$ restricted to $\globe^\mathrm{tlf}$ or $\globe^\mathrm{tlf}_+$ factors as $\nu K(-)$ through $\adc$, and we conclude by Proposition \ref{prop:adcmonoidal}.
\end{proof}

By the two lemmas combined, we have reduced the task of proving Conjecture \ref{conj:monoidality} to proving that the definitions of lax Gray products and joins coincide on non-totally-loop-free molecules. Unfortunately, the monoidality of $K: \globpos \to \adc$ does not seem directly useful for this extension, as $\nu$ is in general only lax monoidal: we would need to prove both that $\nu K(P)$ is naturally isomorphic to $(\molec{}{P})^*$ for all constructible directed complexes $P$, and that $\nu$ is monoidal on the image of $K$.

\begin{remark}
Something we can do is to restrict $K$ to $\globe$, and extend it along colimits to obtain a monoidal functor $K: (\cpol,\tensor,1) \to (\adc, \otimes, I)$. Post-composed with the forgetful functor from augmented directed complexes to chain complexes of abelian groups, this coincides, on constructible polygraphs, with the linearisation functor of \cite[Subsection 3.3]{metayer2003resolutions}.
\end{remark}

We conclude this section with the proof that $\hatto{-}$ is compatible with $J$\nbd duals.

\begin{prop}
Let $J \subseteq \mathbb{N}^+$. There exist isomorphisms $\hatto{(\oppn{J}{X})} \simeq \oppn{J}{\hatto{X}}$, natural in the constructible polygraph $X$.
\end{prop}
\begin{proof}
It suffices to construct isomorphisms $\hatto{(\oppn{J}{U})} \simeq \oppn{J}{\hatto{U}}$ natural in atoms $U$ and inclusions, because $\oppn{J}{-}$ is an invertible endofunctor both in $\cpol$ and in $\omegacat$, so in particular it preserves colimits; we will then have
\begin{align*}
	\hatto{(\oppn{J}{X})} & \simeq \hatto{(\colim{U \to X}\oppn{J}{U})} \simeq \colim{\hatto{(U \to X)}}\hatto{(\oppn{J}{U})} \simeq \colim{\hatto{(U \to X)}}\oppn{J}{\hatto{U}} \simeq \\
	& \simeq \oppn{J}{\colim{\hatto{(U \to X)}} \hatto{U}} \simeq \oppn{J}{\hatto{X}}.
\end{align*}
Let $P$ be a constructible directed complex. We will prove that $U$ is an $n$\nbd molecule in $P$ if and only if $\oppn{J}{U}$ is an $n$\nbd molecule in $\oppn{J}{P}$, and $\bord{k}{\alpha}(\oppn{J}{U}) = \oppn{J}{\bord{k}{\pm \alpha}U}$, the sign depending on whether $k+1 \in J$ or not. 

If $U$ is an atom, this is immediate from Proposition \ref{prop:oppglobpos} and the characterisation of $k$\nbd boundaries of constructible molecules in Proposition \ref{prop:globelike}. If $U$ has a proper decomposition $U_1 \cp{k} U_2$, then $U_1 \cap U_2 = \bord{k}{+}U_1 = \bord{k}{-}U_2$, and, equivalently by the inductive hypothesis, $\oppn{J}{U_1} \cap \oppn{J}{U_2} = \bord{k}{\pm}\oppn{J}{U_1} = \bord{k}{\mp}\oppn{J}{U_2}$, which means that either $\oppn{J}{U_1} \cp{k} \oppn{J}{U_2}$ or $\oppn{J}{U_2} \cp{k} \oppn{J}{U_1}$ is defined and equal to $\oppn{J}{U}$. The relation between the $k$\nbd boundaries of $U$ and of $\oppn{J}{U}$ is then a consequence of Corollary \ref{cor:globelikecor}.

This determines an isomorphism of the underlying $\omega$\nbd graphs of $\molec{}{(\oppn{J}{P})}$ and $\oppn{J}{\molec{}{P}}$; naturality in $P$ and compatibility with units and compositions are obvious. Restricting to atoms and whiskering with the functor $(-)^*$, we obtain a natural isomorphism between $\hatto{(\oppn{J}{-})}$ and $(\oppn{J}{\molec{}{-}})^*$;  to conclude, it suffices to show that $\oppn{J}{-}^*$ and $\oppn{J}{-^*}$ are naturally isomorphic, which is a simple exercise.
\end{proof}

\section{Constructible polygraphs as polygraphs} \label{sec:polygraphs}

We expect that the realisation of a constructible polygraph should admit the structure of a polygraph, with cells of the former corresponding to the generators of the latter. Steiner's results allow us to state this as a theorem for some specific subclasses of constructible polygraphs, but we were not able, so far, to prove this in the general case, where it remains conjectural. In this section, we give some partial results.

Let us recall, first, the definition of polygraphs.

\begin{dfn}
Let $X$ be a partial $\omega$\nbd category and $n \in \mathbb{N}$. The \emph{$n$\nbd skeleton} $\skel{n}{X}$ of $X$ is the partial $\omega$\nbd category whose underlying $\omega$\nbd graph is
\begin{equation*}
\begin{tikzpicture}
	\node[scale=1.25] (0) at (0,0) {$X_0$};
	\node[scale=1.25] (1) at (2,0) {$\ldots$};
	\node[scale=1.25] (2) at (4,0) {$X_n$};
	\node[scale=1.25] (3) at (6,0) {$\idd{}(X_n)$};
	\node[scale=1.25] (4) at (8,0) {$\ldots$};
	\node[scale=1.25] (5) at (10,0) {$\idd{m}(X_n)$};
	\node[scale=1.25] (6) at (12,0) {$\ldots$};
	\draw[1c] (1.west |- 0,.15) to node[auto,swap] {$\bord{}{+}$} (0.east |- 0,.15);
	\draw[1c] (1.west |- 0,-.15) to node[auto] {$\bord{}{-}$} (0.east |- 0,-.15);
	\draw[1c] (2.west |- 0,.15) to node[auto,swap] {$\bord{}{+}$} (1.east |- 0,.15);
	\draw[1c] (2.west |- 0,-.15) to node[auto] {$\bord{}{-}$} (1.east |- 0,-.15);
	\draw[1c] (3.west |- 0,.15) to node[auto,swap] {$\bord{}{+}$} (2.east |- 0,.15);
	\draw[1c] (3.west |- 0,-.15) to node[auto] {$\bord{}{-}$} (2.east |- 0,-.15);
	\draw[1c] (4.west |- 0,.15) to node[auto,swap] {$\bord{}{+}$} (3.east |- 0,.15);
	\draw[1c] (4.west |- 0,-.15) to node[auto] {$\bord{}{-}$} (3.east |- 0,-.15);
	\draw[1c] (5.west |- 0,.15) to node[auto,swap] {$\bord{}{+}$} (4.east |- 0,.15);
	\draw[1c] (5.west |- 0,-.15) to node[auto] {$\bord{}{-}$} (4.east |- 0,-.15);
	\draw[1c] (6.west |- 0,.15) to node[auto,swap] {$\bord{}{+}$} (5.east |- 0,.15);
	\draw[1c] (6.west |- 0,-.15) to node[auto] {$\bord{}{-}$} (5.east |- 0,-.15);
	\node[scale=1.25] at (12.5,-.25) {$,$};
\end{tikzpicture}
\end{equation*}
that is, the restriction of $X$ to cells $x$ with $\dmn{x} \leq n$; unit and composition operations are restricted as appropriate. 

There is an obvious inclusion $\skel{n}X \hookrightarrow X$, which factors through $\skel{m}{X} \hookrightarrow X$ for all $m > n$. Any partial $\omega$\nbd category is the sequential colimit of its skeleta.
\end{dfn}

For each $n \in \mathbb{N}$, let $\hatto{O^n}$ be the $n$\nbd globe as an $\omega$\nbd category: $\hatto{O^n}$ has two $k$\nbd dimensional cells $\underline{k}^+, \underline{k}^-$ for each $k < n$, and a single $n$\nbd dimensional cell $\underline{n}$, such that $\bord{k}{\alpha}\underline{n} = \underline{k}^\alpha$ for all $k < n$. Let $\bord{}{}\hatto{O^n}$ be the $(n-1)$\nbd skeleton of $\hatto{O^n}$. There is a bijection between $n$\nbd cells $x$ of a partial $\omega$\nbd category $X$ and functors $x: \hatto{O^n} \to X$, and we will identify the two. We write $\bord{}{}x$ for the precomposition of $x$ with the inclusion $\bord{}{}\hatto{O^n} \hookrightarrow \hatto{O^n}$.

For a family $X_i$ of $\omega$\nbd categories, let $\coprod_{i \in I} X_i$ be its coproduct in $\omegacat$; the sets of $n$\nbd cells and the boundary, unit, and composition operations of the coproduct are all induced pointwise by coproducts of sets and functions. Given a family of functors $\{f_i: X_i \to Y\}_{i \in I}$, let $(f_i)_{i \in I}: \coprod_{i \in I}X_i \to Y$ be the functor produced by the universal property of coproducts.

\begin{dfn}
A \emph{polygraph} is an $\omega$\nbd category $X$ together with families $\{\atom{n}{X}\}_{n \in \mathbb{N}}$ of $n$\nbd dimensional cells of $X$, such that, for all $n$,
\begin{equation*}
\begin{tikzpicture}[baseline={([yshift=-.5ex]current bounding box.center)}]
	\node[scale=1.25] (0) at (0,2) {$\displaystyle \coprod_{x\in\atom{n}{X}} \bord{}{} \hatto{O^n}$};
	\node[scale=1.25] (1) at (3.5,2) {$\displaystyle\coprod_{x\in\atom{n}{X}} \hatto{O^n}$};
	\node[scale=1.25] (2) at (0,0) {$\skel{n-1}{X}$};
	\node[scale=1.25] (3) at (3.5,0) {$\skel{n}{X}$};
	\draw[1c] (0) to node[auto,swap] {$(\bord{}{}x)_{x\in\atom{n}{X}}$} (2);
	\draw[1c] (1) to node[auto] {$(x)_{x\in\atom{n}{X}}$} (3);
	\draw[1cinc] (0) to (1);
	\draw[1cinc] (2) to (3);
	\draw[edge] (2.5,0.2) to (2.5,0.8) to (3.3,0.8);
\end{tikzpicture}
\end{equation*}
is a pushout diagram in $\omegacat$. The cells in $\atom{n}{X}$ are called $n$\nbd dimensional \emph{generators} of $X$.

A \emph{map} of polygraphs is a functor $f: X \to Y$ of $\omega$\nbd categories that sends $n$\nbd dimensional generators of $X$ to $n$\nbd dimensional generators of $Y$; that is, $f$ restricts to, and is essentially determined by a sequence of functions $\{f_n: \atom{n}{X} \to \atom{n}{Y}\}_{n \in \mathbb{N}}$. Polygraphs and their maps form a category $\pol$.
\end{dfn}

\begin{remark}
This is a concise definition of polygraph; there are more explicit ones, based on iterated free algebra constructions, including the original \cite{burroni1993higher}. 

In our definition, we include the generators as structure on an $\omega$\nbd category. Occasionally, ``being a polygraph'' may be used as the property of an $\omega$\nbd category which admits the structure of a polygraph: this is the sense in which ``polygraphs are the cofibrant objects'' in the model structure on $\omegacat$ defined by Lafont, M\'etayer, and Worytkiewicz \cite{lafont2010folk}.
\end{remark}

\begin{dfn} \label{dfn:freegen}
Let $P$ be a directed complex. We say that $P$ is \emph{freely generating} if $(\molec{}{P})^*$ admits the structure of a polygraph, whose $n$\nbd dimensional generators are the $n$\nbd dimensional atoms of $P$.
\end{dfn}

\begin{remark}
Clearly, the atoms of $P$ are composition-generators of $(\molec{}{P})^*$; what is not guaranteed is that they are freely generating in the sense of polygraphs. More specifically, it is possible in general that there are two expressions of a molecule of $P$ as an $\omega$\nbd categorical composite of atoms which are not equal modulo the axioms of $\omega$\nbd categories.
\end{remark}

So far, we have not been able to prove the following conjecture.
\begin{conj} \label{conj:freegen}
Every constructible directed complex is freely generating.
\end{conj}
This is certainly not true of general directed complexes: see \cite{forest2019unifying} and \cite[Section 2.5]{henry2019non} for some counterexamples which can be rephrased in Steiner's theory. However, all such counterexamples seem to involve some looping which is not permitted in a constructible directed complex, such as the input and output boundary of an atom sharing some non-boundary elements.

\begin{remark} \label{rmk:regularity}
More generally, we conjecture that a directed complex which is \emph{regular} in the sense of \cite{henry2018regular} should be freely generating: that is, a directed complex such that, for all $n$\nbd dimensional atoms and $k < n$,  $\bord{k}{-}x \cap \bord{k}{+}x = \bord{k-1}{}x$. 
\end{remark}

At present, we have two potentially different ways of realising a constructible directed complex as an $\omega$\nbd category, extending $\hatto{-}: \globe \to \omegacat$: one is $(\molec{}{-})^*$, and the other is the left Kan extension of $\hatto{-}$ along $\globe \incl \globpos$, which we denote by $\hatto{-}: \globpos \to \omegacat$. 

\begin{prop} \label{prop:molec_hatto_iso}
Let $P$ be a freely generating constructible directed complex. Then $(\molec{}{P})^*$ and $\hatto{P}$ are isomorphic.
\end{prop}
\begin{proof}
It suffices to show that the $n$\nbd skeleta are isomorphic for each $n$. If $n = 0$, this is obvious, so suppose $n > 0$, and assume $\skel{n-1}{(\molec{}{P})^*}$ is isomorphic to $\skel{n-1}{\hatto{P}}$.

Any $n$\nbd dimensional atom $\clos\{x\} \subseteq P$ is freely generating. Let $x: \hatto{O^n} \to \hatto{\clos\{x\}}$ be the unique functor sending $\underline{n}$ to $\clos\{x\}$, and $\bord{}{}x: \bord{}{}\hatto{O^n} \to \hatto{\bord{}{}x}$ its restriction to $(n-1)$\nbd skeleta, where we use the fact that $\hatto{(\bord{}{}x)} = (\molec{}{\bord{}{}x})^*$. Then
\begin{equation*}
\begin{tikzpicture}[baseline={([yshift=-.5ex]current bounding box.center)}]
	\node[scale=1.25] (0) at (0.5,2) {$\hatto{\bord{}{}O^n}$};
	\node[scale=1.25] (1) at (3.5,2) {$\hatto{O^n}$};
	\node[scale=1.25] (2) at (0.5,0) {$\hatto{\bord{}{}x}$};
	\node[scale=1.25] (3) at (3.5,0) {$\hatto{\clos\{x\}}$};
	\draw[1c] (0) to node[auto,swap] {$\bord{}{}x$} (2);
	\draw[1c] (1) to node[auto] {$x$} (3);
	\draw[1cinc] (0) to (1);
	\draw[1cinc] (2) to (3);
	\draw[edge] (2.5,0.2) to (2.5,0.8) to (3.3,0.8);
\end{tikzpicture}
\end{equation*}
is a pushout diagram in $\omegacat$. This allows us to replace $n$\nbd globes with constructible $n$\nbd atoms in the polygraphic extension of $\skel{n-1}{\hatto{P}}$:
\begin{equation} \label{eq:modcolimit_cpx}
\begin{tikzpicture}[baseline={([yshift=-.5ex]current bounding box.center)}]
	\node[scale=1.25] (0) at (-.5,2) {$\displaystyle \coprod_{x\in P^{(n)}} \hatto{\bord{}{}x}$};
	\node[scale=1.25] (1) at (3.5,2) {$\displaystyle\coprod_{x\in P^{(n)}} \hatto{\clos\{x\}}$};
	\node[scale=1.25] (2) at (-.5,0) {$\skel{n-1}{\hatto{P}}$};
	\node[scale=1.25] (3) at (3.5,0) {$\skel{n}{(\molec{}{P})^*}$.};
	\draw[1c] (0) to node[auto,swap] {$(\bord{}{}x)_{x\in P^{(n)}}$} (2);
	\draw[1c] (1) to node[auto] {$(\clos\{x\})_{x\in P^{(n)}}$} (3);
	\draw[1cinc] (0) to (1);
	\draw[1cinc] (2) to (3);
	\draw[edge] (2.5,0.2) to (2.5,0.8) to (3.3,0.8);
\end{tikzpicture}
\end{equation}
By Corollary \ref{cor:globpos_is_colimit}, $\skel{n-1}{P}$ is the colimit of the diagram of inclusions of its atoms, and $\hatto{-}$ preserves colimits. Thus (\ref{eq:modcolimit_cpx}) exhibits $\skel{n}{(\molec{}{P})^*}$ as the colimit of the image through $\hatto{-}$ of the diagram of inclusions of atoms of $\skel{n}{P}$, but this is the same universal property satisfied by $\skel{n}{\hatto{P}}$.
\end{proof}

Let $\globefree$ be the full subcategory of $\globe$ on the freely generating atoms, and $\cpolfree$ the category of presheaves on $\globefree$, which can be identified with a full subcategory of $\cpol$; conjecturally, these coincide with $\globe$ and $\cpol$. 

\begin{thm} \label{thm:rpolintopol}
Let $X$ be a constructible polygraph in $\cpolfree$. Then the $\omega$\nbd category $\hatto{X}$ admits the structure of a polygraph, whose $n$\nbd dimensional generators are indexed by the $n$\nbd cells of $X$. This extends to a full and faithful functor $\hatto{-}: \cpolfree \to \pol$.
\end{thm}
\begin{proof}
For each freely generating atom $U$ with greatest element $x$, let $x: \hatto{O^n} \to \hatto{U}$ be the unique functor sending $\underline{n}$ to $\clos\{x\}$, and $\bord{}{}x: \bord{}{}\hatto{O^n} \to \skel{n-1}{\hatto{U}}$ its restriction to $(n-1)$\nbd skeleta. By Proposition \ref{prop:molec_hatto_iso}, 
\begin{equation*}
\begin{tikzpicture}[baseline={([yshift=-.5ex]current bounding box.center)}]
	\node[scale=1.25] (0) at (0.5,2) {$\bord{}{}\hatto{O^n}$};
	\node[scale=1.25] (1) at (3.5,2) {$\hatto{O^n}$};
	\node[scale=1.25] (2) at (0.5,0) {$\bord{}{}\hatto{U}$};
	\node[scale=1.25] (3) at (3.5,0) {$\hatto{U}$};
	\draw[1c] (0) to node[auto,swap] {$\bord{}{}x$} (2);
	\draw[1c] (1) to node[auto] {$x$} (3);
	\draw[1cinc] (0) to (1);
	\draw[1cinc] (2) to (3);
	\draw[edge] (2.5,0.2) to (2.5,0.8) to (3.3,0.8);
\end{tikzpicture}
\end{equation*}
is a pushout diagram in $\omegacat$. Since $\hatto{-}$ preserves colimits, the pushout diagrams of Proposition \ref{prop:polygraph_ext} are sent to pushout diagrams
\begin{equation*}
\begin{tikzpicture}[baseline={([yshift=-.5ex]current bounding box.center)}]
	\node[scale=1.25] (0) at (0,2) {$\displaystyle \coprod_{x\in\atom{n}{X}} \hatto{\bord{}{} U(x)}$};
	\node[scale=1.25] (1) at (3.5,2) {$\displaystyle\coprod_{x\in\atom{n}{X}} \hatto{U(x)}$};
	\node[scale=1.25] (2) at (0,0) {$\skel{n-1}{\hatto{X}}$};
	\node[scale=1.25] (3) at (3.5,0) {$\skel{n}{\hatto{X}}$,};
	\draw[1c] (0) to node[auto,swap] {$(\bord{}{}x)_{x\in\atom{n}{X}}$} (2);
	\draw[1c] (1) to node[auto] {$(x)_{x\in\atom{n}{X}}$} (3);
	\draw[1cinc] (0) to (1);
	\draw[1cinc] (2) to (3);
	\draw[edge] (2.5,0.2) to (2.5,0.8) to (3.3,0.8);
\end{tikzpicture}
\end{equation*}
and proceeding as in the proof of Proposition \ref{prop:molec_hatto_iso}, we see that these exhibit $\hatto{X}$ as a polygraph with the specified structure. 

The fact that a map of regular polygraphs $f: X \to Y$ induces a map of polygraphs is immediate from the description of the generators of $\hatto{X}$ and $\hatto{Y}$, since $f$ sends $x \in X(U)$ to $f(x) \in Y(U)$. Faithfulness is also immediate: two maps of polygraphs $\hatto{f}, \hatto{g}: \hatto{X} \to \hatto{Y}$ are equal if and only if they are equal on the generators of $\hatto{X}$, if and only if $f_U, g_U: X(U) \to Y(U)$ are equal for all globes $U$, equivalently, if $f, g$ are equal as maps of regular polygraphs.

Finally, if $f: \hatto{X} \to \hatto{Y}$ is a map of polygraphs, it must send an $n$\nbd dimensional generator indexed by $x \in X(U)$ to another $n$\nbd dimensional generator indexed by $f(x) \in X(U')$; it suffices to show that $U = U'$. If $U$ is the only 0-globe, this is obvious, and it is enough if $X$ is 0-dimensional, that is, $X(U)$ is empty when $U$ has dimension $n > 0$.

For $n > 0$, we can assume the inductive hypothesis that, for all $k < n$ and $k$\nbd dimensional regular polygraphs $Z$, any map of polygraphs $g: \hatto{Z} \to \hatto{Y}$ is in the image of $\hatto{-}$. Now, $\bord{}{\alpha}f(x): \bord{}{\alpha}\hatto{U'} \to \hatto{Y}$, which is a colimit in $\slice{\pol}{\hatto{Y}}$ of the generators $f(y): \hatto{V} \to \hatto{Y}$ in $\bord{}{\alpha}f(x)$ and their inclusions, can be computed as a colimit in $\slice{\cpolfree}{Y}$, preserved by $\hatto{-}$. 

By the inductive hypothesis, the post-composition with $f$ of the diagram consisting of the $\bord{}{\alpha}x: \bord{}{\alpha}\hatto{U} \to \hatto{X}$, the generators $y: \hatto{V} \to \hatto{X}$ in $\bord{}{\alpha}x$, and their inclusions is also the image through $\hatto{-}$ of a diagram in $\slice{\cpolfree}{Y}$; from the universal property of colimits, we obtain a map $\bord{}{\alpha}U' \to \bord{}{\alpha}U$ in $\cpolfree$, which is necessarily an isomorphism. By Proposition \ref{prop:atomicnglobe} atomic $n$\nbd globes are classified up to isomorphism by their $(n-1)$\nbd boundaries; since $\globefree$ is skeletal, $U = U'$.
\end{proof}

\begin{remark}
Let $P'$ be the right adjoint to $\hatto{-}: \cpolfree \to \omegacat$. By Theorem \ref{thm:rpolintopol} and the classification of constructible atoms in Proposition \ref{prop:atomicnglobe}, given an $\omega$\nbd category $X$, the counit $\hatto{P'X} \to X$ coincides with the standard resolution by polygraphs of $X$, as defined by M\'etayer \cite[Subsection 4.2]{metayer2003resolutions}, when at each level the boundary of $n$\nbd generators is restricted to constructible diagrams of $(n-1)$\nbd generators.
\end{remark}

We conclude this section by proving that certain classes of constructible directed complexes are freely generating. The most general criterion currently available, to our knowledge, is provided by Steiner's theory of \emph{split molecules}.

\begin{dfn}
Let $U$ be a closed subset of an oriented graded poset. The \emph{frame dimension} of $U$ is the integer $\frdmn{U} := \dmn{\bigcup\{\clos\{x\} \cap \clos\{y\} \,|\, x, y$ maximal in $U$, $x \neq y\}}$.
\end{dfn}

\begin{exm}
If $U$ has a greatest element, then $\frdmn{U} = \dmn{\emptyset} = -1$. If $U$ is a non-atomic constructible $n$\nbd molecule, there are two maximal elements $x, y$ of $U$ such that $\clos\{x\} \cap \clos\{y\}$ is a constructible $(n-1)$\nbd molecule, hence $\frdmn{U} = n-1$.
\end{exm}

\begin{dfn}
A molecule $U$ in an oriented graded poset $P$ is \emph{split} if any factor $V$ in a decomposition of $U$ as an iterated composite admits an expression as an iterated composite of atoms using only the compositions $\cp{k}$ for $k \leq \frdmn{V}$.
\end{dfn}

\begin{prop}\emph{\cite[Theorem 2.13]{steiner1993algebra}} \label{prop:split_generates} Let $P$ be a directed complex whose molecules are all split. Then $P$ is freely generating.
\end{prop}

\begin{remark}
The condition on split molecules is technical and somewhat mysterious. In practice, it is used as a combinatorial counterpart to Power's topological ``domain replacement'' condition \cite[Definition 3.9]{power1991pasting}: it ensures, for example, that the $(k+1)$\nbd boundaries of two $n$\nbd dimensional atoms $\clos\{x_1\}$, $\clos\{x_2\}$ in an $n$\nbd molecule $U$ of frame dimension $k$ are ``simultaneous'' submolecules of $\bord{k}{\alpha}U$, in the sense that there exist expressions of $\bord{k}{\alpha}U$ as an iterated composite which contain $\bord{k}{\alpha}x_1$ and $\bord{k}{\alpha}x_2$ at the same time.
\end{remark}

\begin{exm}
By \cite[Proposition 6.7]{steiner1993algebra}, all molecules are split in loop-free directed complexes, and in particular in totally loop-free directed complexes. Since $\vec{I}$ is totally loop-free and the property is preserved by lax Gray products and joins, both cubes and oriented simplices are freely generating. Thus via Theorem \ref{thm:rpolintopol}, restricted to the full subcategories on cubes and oriented simplices, we recover the embedding of pre-cubical sets and of semi-simplicial sets into $\pol$.
\end{exm}

\begin{exm} \label{exm:nonsplit}
Not all molecules in a constructible directed complex are split; in particular, the non-split molecule of \cite[Section 8]{steiner1993algebra} is a constructible directed complex (although it is not a constructible molecule). 
\end{exm} 

By the last example, a general proof of Conjecture \ref{conj:freegen} cannot be based purely on Proposition \ref{prop:split_generates}. It may, however, use it as an intermediate step. Thus, in the remainder of the section, we develop this theory further, by exhibiting some criteria for recognising split molecules that are weaker than loop-freeness, but simpler to check than the direct definition.

The following result is implied by Steiner's proofs, but not made explicit.
\begin{lem} \label{lem:splitcondition}
Let $P$ be a directed complex. The following are equivalent:
\begin{enumerate}[label=(\alph*)]
	\item all molecules in $P$ are split;
	\item each molecule $U$ in $P$ is either an atom, or it has a proper decomposition $U_1 \cp{n} U_2$ such that $n = \frdmn{U}$ and $U_1 \cap U_2 \not\subseteq \bord{n}{}U$.
\end{enumerate}
\end{lem}
\begin{proof}
The implication from $(a)$ to $(b)$ is \cite[Proposition 4.3]{steiner1993algebra}. Conversely, by an analysis of the proof of [Proposition 6.7, \emph{ibid.}] that ``all loop-free molecules are split'', we find that the only properties of loop-free molecules $U$ that are used are:
\begin{enumerate}
	\item for all $x \in U$, the sets $\sbord{}{+}x$ and $\sbord{}{-}x$ are disjoint;
	\item $U$ is either an atom, or it has a proper decomposition $U_1 \cp{n} U_2$ with $n = \frdmn{U}$, and $U_1 \cap U_2 \not\subseteq \bord{n}{}U$;
	\item if $V \subseteq U$ is another molecule, then it also satisfies the first two properties.
\end{enumerate}
The first property is trivial for oriented graded posets, and the other two are also satisfied by the family of ``all molecules in $P$'', conditionally to $(b)$.
\end{proof}

\begin{lem}  \label{lem:framedim_intersec}
Let $U$ be a molecule in a constructible directed complex, $\frdmn{U} = n$, and let $x, y$ be distinct maximal elements of $U$. Then 
\begin{equation*}
	\clos\{x\} \cap \clos\{y\} \subseteq (\bord{n}{+}x \cap \bord{n}{-}y) \cup (\bord{n}{+}y \cap \bord{n}{-}x).
\end{equation*}
\end{lem}
\begin{proof}
The proof of \cite[Proposition 6.4]{steiner1993algebra}, showing that the statement holds for loop-free molecules in a directed complex, only uses two properties of loop-free atoms $x$ of dimension greater than $n$: that $\bord{n}{+}x \cap \bord{n}{-}x = \bord{n-1}{}x$, and that $\bord{n}{\alpha}x$ is pure and $n$\nbd dimensional. Both are also true of atoms in a constructible directed complex.
\end{proof}

\begin{dfn}
Given a closed subset $U$ of a constructible directed complex, for each $n \in \mathbb{N}$, let $\maxd{n}{U}$ be the bipartite directed graph with 
\begin{equation*}
	\{x \in U \,|\, \dmn{x} \leq n\} + \{x \in U \,|\, x \text{ is maximal and } \dmn{x} > n\}
\end{equation*}
as set of vertices, and an edge $y \to x$ if and only if 
\begin{itemize}
	\item $\dmn{y} \leq n$, $\dmn{x} > n$, and $y \in \bord{n}{-}x \setminus \bord{n-1}{}x$, or
	\item $\dmn{y} > n$, $\dmn{x} \leq n$, and $x \in \bord{n}{+}y \setminus \bord{n-1}{}y$.
\end{itemize}
This is the subgraph of $\loopd{n}{U}$ whose vertices of dimension greater than $n$ are restricted to the maximal elements of $U$.
\end{dfn}

We have the following strengthening of \cite[Proposition 6.7]{steiner1993algebra}.
\begin{prop} \label{prop:max_free}
Let $P$ be a constructible directed complex. Suppose that, for all molecules $U$ of $P$, if $\frdmn{U} = n$, then $\maxd{n}{U}$ is acyclic. Then all molecules of $P$ are split, and $P$ is freely generating.
\end{prop}
\begin{proof}
If $U$ is an atom, any factor in a decomposition of $U$ is either $U$ or a factor of $\bord{k}{\alpha}U$ for some $k$, so suppose $U$ is not an atom.

By Lemma \ref{lem:framedim_intersec}, if $\maxd{n}{U}$ is acyclic, we can list the maximal elements of $U$ of dimension greater than $n$ as $x_1,\ldots,x_m$, in such a way that there is no path from $x_j$ to $x_i$ in $\maxd{n}{U}$ if $j > i$. Then, as in the proof of \cite[Proposition 6.5]{steiner1993algebra}, we show that $\clos\{x_i\} \cap \clos\{x_j\} \subseteq \bord{n}{+}x_i \cap \bord{n}{-}x_j$ if $i < j$.

Because $U$ has frame dimension $n$, there exist $p < q \leq m$ such that $\sbord{n}{+}x_p \cap \sbord{n}{-}x_q$ is inhabited. Let
\begin{equation*}
	U_1 := \bord{n}{-}U \cup \bigcup_{i=1}^p \clos\{x_i\}, \quad \quad \quad U_2 := \bord{n}{+}U \cup \bigcup_{j=p+1}^m \clos\{x_j\}.
\end{equation*}
Then, $U = U_1 \cp{n} U_2$ and $U_1 \cap U_2 \not\subseteq \bord{n}{}U$ is proved as in \cite[Proposition 6.6]{steiner1993algebra}, and we conclude by Lemma \ref{lem:splitcondition}.
\end{proof}

\begin{exm}
In Example \ref{exm:nonloopfree}, $U$ is not loop-free, but there are no molecules of $U$ which contain both $y_1$ and $y_2$ as maximal elements, thus Proposition \ref{prop:max_free} applies.
\end{exm}

We give a sufficient (but not necessary) condition for the applicability of Proposition \ref{prop:max_free}. Recall the definition of flow-connected molecules (Definition \ref{dfn:flowconnect}). 
\begin{dfn}
Let $P$ be a constructible directed complex, $x \in P$. We say that $x$ has \emph{flow-connected boundaries} if, for all $k < \dmn{x}$, the constructible $k$\nbd molecules $\bord{k}{+}x, \bord{k}{-}x$ are flow-connected.
\end{dfn}

\begin{lem} \label{lem:flow-connect-maxd}
Let $U$ be a flow-connected constructible $(n+1)$\nbd molecule. Then for all $x \in \bord{n}{-}U \setminus \bord{n-1}{}U$ and $x' \in \bord{n}{+}U \setminus \bord{n-1}{}U$, there is a path from $x$ to $x'$ in $\maxd{n}{U}$.
\end{lem}
\begin{proof}
Any element of $\bord{n}{\alpha}U \setminus \bord{n-1}{}U$ is in the closure of an $n$\nbd dimensional element, so the statement immediately follows from the definition of flow-connectedness.
\end{proof}

\begin{prop}
Let $U$ be a molecule in a constructible directed complex, and suppose the atoms of $U$ have flow-connected boundaries. If $\frdmn{U} = n$, then $\maxd{n}{U}$ is acyclic.
\end{prop}
\begin{proof}
Suppose $U$ is an $(n+1)$\nbd dimensional molecule. Then $\maxd{n}{U} = \loopd{n}{U}$, which is acyclic by Proposition \ref{prop:loopd-acyclic}. 

Otherwise, since $\frdmn{U} = n$, for all pairs $x, y$ of maximal elements $\clos{\{x\}} \cap \clos{\{y\}}$ is at most $n$\nbd dimensional, so any $z \in \sbord{n+1}{}x$ is only covered by elements in the closure of $x$, and $z \in \sbord{n+1}{\alpha}x$ implies $z \in \sbord{n+1}{\alpha}U$. It follows that $\bord{n+1}{\alpha}U$ contains $\bord{n+1}{\alpha}x$ for all maximal elements $x \in U$. 

Suppose $\maxd{n}{U}$ has a cycle. This cycle is a concatenation of two-step paths $x \to y \to x'$ where $y \in U$ is a maximal element, and $x \in \bord{n}{-}y\setminus\bord{n-1}{}y$, $x' \in \bord{n}{+}y\setminus\bord{n-1}{}y$. By Lemma \ref{lem:flow-connect-maxd} applied to $\bord{n+1}{\alpha}y$, there is a path $x \to \ldots \to x'$ in $\maxd{n}{(\bord{n+1}{\alpha}y)}$, hence in $\maxd{n}{(\bord{n+1}{\alpha}U)}$. 

Replacing each two-step path in the cycle with such a path, we obtain a cycle in $\maxd{n}{(\bord{n+1}{\alpha}U)}$, and $\bord{n+1}{\alpha}U$ is an $(n+1)$\nbd dimensional molecule, which contradicts the first part of the proof. Therefore $\maxd{n}{U}$ is acyclic.
\end{proof}

\begin{cor} \label{cor:flow-conn-split}
Let $P$ be a constructible directed complex whose atoms have flow-connected boundaries. Then all molecules of $P$ are split, and $P$ is freely generating.
\end{cor}

\begin{exm}
Let $\pope$ to be the full subcategory of $\globe$ on positive opetopes, and $\popeset$ its category of presheaves. By Proposition \ref{prop:pope-flowconn}, all positive opetopes are flow-connected, and all boundaries of positive opetopes are positive opetopes. It follows from Corollary \ref{cor:flow-conn-split} that positive opetopes are freely generating.

From Theorem \ref{thm:rpolintopol}, we obtain a full and faithful functor $\popeset \to \pol$ whose image consists of ``positive-to-one'' polygraphs, in the sense of Zawadowski \cite{zawadowski2007positive}, hence positive opetopic sets in the sense of \cite{zawadowski2017positive}. 
\end{exm}

\section{Geometric realisation} \label{sec:geometric}

We informally stated that constructible $n$\nbd molecules have $k$\nbd boundaries shaped as $k$\nbd balls for each $k < n$. In this section, we make this precise, by defining the geometric realisation of a constructible directed complex and of a constructible polygraph.

First, we need to recall some basic notions from algebraic and combinatorial topology; we refer to any textbook, for example \cite{may1999concise} and \cite{goerss2009simplicial}, for more details. We will work with the ``convenient'' category $\cghaus$ of compactly generated Hausdorff spaces. We assume that the definitions of simplicial set, their category $\sset$, and the geometric realisation $\realis{-}_\Delta: \sset \to \cghaus$ are known by the reader. Throughout this section, we will not always distinguish between an oriented graded poset and its underlying poset; the context should decide which one we mean.

\begin{dfn}
Let $P$ be a poset. The \emph{nerve} of $P$ is the simplicial set $NP$ whose
\begin{itemize}
	\item $n$\nbd simplices are chains $(x_0 \leq \ldots \leq x_n)$ of length $(n+1)$ in $P$,
	\item the $k$\nbd th face map $d_k: NP_n \to NP_{n-1}$ is defined by 
	\begin{equation*}
		(x_0 \leq \ldots \leq x_n) \mapsto (x_0 \leq \ldots \leq x_{k-1} \leq x_{k+1} \leq \ldots \leq x_n),
	\end{equation*}
	\item the $k$\nbd th degeneracy map $s_k: NP_n \to NP_{n+1}$ is defined by 
	\begin{equation*}
		(x_0 \leq \ldots \leq x_n) \mapsto (x_0 \leq \ldots \leq x_{k} \leq x_{k} \leq \ldots \leq x_n),
	\end{equation*} 
\end{itemize}
for $0 \leq k \leq n$. 
\end{dfn}

The nerve extends to a functor $N: \pos \to \sset$. Precomposing with the forgetful functor $\globpos \to \pos$, we obtain a functor $k: \globpos \to \sset$. Finally, we can post-compose with $\realis{-}_\Delta: \sset \to \cghaus$ to obtain a functor $\realis{-}: \globpos \to \cghaus$, which we call the \emph{geometric realisation} of a constructible directed complex.

The functor $\realis{-}_\Delta$ has the property that $\realis{K \times L}_\Delta$ is homeomorphic to the product of spaces $\realis{K}_\Delta \times \realis{L}_\Delta$, and $\realis{K \join L}_\Delta$ to the join of spaces $\realis{K}_\Delta \join \realis{L}_\Delta$, for any pair of simplicial sets $K$, $L$. The nerve functor also preserves products and joins, and the underlying poset of the lax Gray product of two oriented graded posets is the product of their underlying posets. It follows that $\realis{-}$ becomes a monoidal functor from $(\globpos, \tensor, 1)$ to $(\cghaus, \times, 1)$ and from $(\globpos, \join, \emptyset)$ to $(\cghaus, \join, \emptyset)$.

\begin{remark}
In combinatorics, it is more common to consider the \emph{order complex} of a poset, an ordered simplicial complex, rather than the nerve, a simplicial set. The two notions of geometric realisation coincide up to homeomorphism.
\end{remark}

In what follows, let $D^n$ be a model of the closed $n$\nbd ball, and $\bord{}{}D^n$ its boundary, homeomorphic to the $(n-1)$\nbd sphere; for a map $x: D^n \to X$, let $\bord{}{}x$ be its restriction to $\bord{}{}D^n$. 
\begin{dfn}
A \emph{CW complex} is a topological space $X$ together with a non-decreasing sequence $\{\skel{n}{X} \hookrightarrow X\}_{n \in \mathbb{N}}$ of subspaces, the \emph{$n$\nbd skeleta}, and families $\{\atom{n}{X}\}_{n \in \mathbb{N}}$ of maps $x: D^n \to \skel{n}{X}$, such that $X = \bigcup_{n\in \mathbb{N}} \skel{n}{X}$ and, for all $n$,
\begin{equation*}
\begin{tikzpicture}[baseline={([yshift=-.5ex]current bounding box.center)}]
	\node[scale=1.25] (0) at (0,2) {$\displaystyle \coprod_{x\in\atom{n}{X}} \bord{}{} D^n$};
	\node[scale=1.25] (1) at (3.5,2) {$\displaystyle\coprod_{x\in\atom{n}{X}} D^n$};
	\node[scale=1.25] (2) at (0,0) {$\skel{n-1}{X}$};
	\node[scale=1.25] (3) at (3.5,0) {$\skel{n}{X}$};
	\draw[1c] (0) to node[auto,swap] {$(\bord{}{}x)_{x\in\atom{n}{X}}$} (2);
	\draw[1c] (1) to node[auto] {$(x)_{x\in\atom{n}{X}}$} (3);
	\draw[1cinc] (0) to (1);
	\draw[1cinc] (2) to (3);
	\draw[edge] (2.5,0.2) to (2.5,0.8) to (3.3,0.8);
\end{tikzpicture}
\end{equation*}
is a pushout diagram in $\cghaus$. 

A CW complex is \emph{regular} if, for all $n$ and $x \in \atom{n}{X}$, the map $x: D^n \to X$ is a homeomorphism onto its image.
\end{dfn}

By analogy with polygraphs, we will call $x \in \atom{n}{X}$ a \emph{generating} $n$\nbd cell of the CW complex $X$, even though the term is more commonly associated to an algebraic setting.

\begin{dfn}
Let $X$ be a CW complex. The \emph{face poset} $\face{X}$ of $X$ is the poset whose elements are the generating cells of $X$, and for any pair of generating cells $x: D^k \to X$ and $y: D^n \to X$ we have $x \leq y$ if and only if $x(D^k) \subseteq y(D^n)$.
\end{dfn}

The face poset is, arguably, the simplest non-trivial combinatorial structure that one can associate to a CW complex, yet for regular CW complexes, it specifies the type of the underlying topological space up to homeomorphism.
\begin{thm}\emph{\cite[Theorem 1.7]{lundell1969topology}}
Let $X$ be a regular CW complex. Then $|N(\face{X})|$ is homeomorphic to $X$.
\end{thm}

In \cite{bjorner1984posets}, Bj\"orner studied criteria for a poset to be the face poset of a regular CW complex. The following is [Definition 2.1, \emph{ibid.}].
\begin{dfn}
A poset with a least element $P_\bot$ is a \emph{CW poset} if $P$ has at least one element, and, for all $x \in P$, the geometric realisation $|N(\bord{}{}x)|$ is homeomorphic to a sphere.
\end{dfn}

\begin{prop}\emph{\cite[Proposition 3.1]{bjorner1984posets}} A poset $P$ is the face poset of a regular CW complex if and only if $P_\bot$ is a CW poset.
\end{prop}

We will prove that if $P$ is a constructible directed complex, then $P_\bot$ is a CW poset. For this purpose, we go through the intermediate notion of a recursively dividable poset, defined in \cite{hachimori2000constructible}. The following is a rephrasing of [Definition 3.1, \emph{ibid.}]; note that we use a pure $n$\nbd dimensional subset $U$ where Hachimori would add an $(n+1)$\nbd dimensional greatest element.

\begin{dfn}
Let $P$ be a graded poset. We define a class of pure subsets $U$ of $P$, that we call \emph{recursively dividable}.

The empty subset is recursively dividable. Suppose $U$ is inhabited, pure and $n$\nbd dimensional. Then $U$ is recursively dividable if and only if $\bord{}{}x$ is recursively dividable for all $x \in U^{(n)}$, and, inductively on the number of maximal elements of $U$, either
\begin{itemize} 
	\item $U$ has a greatest element, or
	\item $U = U_1 \cup U_2$, where 
	\begin{enumerate}
		\item $U_1$ and $U_2$ are pure, $n$\nbd dimensional, and recursively dividable, and
		\item $U_1 \cap U_2$ is pure, $(n-1)$\nbd dimensional, and recursively dividable.
	\end{enumerate}
\end{itemize}
\end{dfn}

\begin{prop} \label{prop:recdivid}
Let $U$ be a constructible molecule. Then $U$ and $\bord{}{}U$ are recursively dividable.
\end{prop}
\begin{proof}
If $U$ is a 0\nbd molecule, this is immediate from the definition, so suppose $U$ is a constructible $n$\nbd molecule with $n > 0$. 

We have $\bord{}{}U = \bord{}{-}U \cup \bord{}{+}U$, where $\bord{}{-}U$ and $\bord{}{+}U$ are constructible $(n-1)$\nbd molecules, and $\bord{}{-}U \cap \bord{}{+}U = \bord{}{}(\bord{}{\alpha}U)$ is the boundary of a constructible $(n-1)$\nbd molecule: all are recursively dividable by the inductive hypothesis. It follows that $\bord{}{}U$ is recursively dividable.

If $U$ is atomic, we are done. Otherwise, $U$ splits as $U_1 \cup U_2$, where $U_1$, $U_2$ are constructible $n$\nbd molecules with fewer maximal elements, and their intersection $U_1 \cap U_2$ is a constructible $(n-1)$\nbd molecule. All are recursively dividable by the inductive hypothesis, so $U$ is recursively dividable.
\end{proof}

We do not know, at the moment, whether constructible molecules are also strongly dividable in the sense of \cite[Definition 4.2]{hachimori2000constructible}.

\begin{thm} \label{thm:regularball}
Let $U$ be a constructible $n$\nbd molecule. Then $\realis{U}$ is homeomorphic to a closed $n$\nbd ball, and $\realis{\bord{}{}U}$ is homeomorphic to an $(n-1)$\nbd sphere.
\end{thm}
\begin{proof}
By Proposition \ref{prop:recdivid} and \cite[Proposition 3.6]{hachimori2000constructible}, the order complexes of $U$ and $\bord{}{}U$ are constructible simplicial complexes [Definition 2.7, \emph{ibid.}]. Moreover, by thinness and Lemma \ref{lem:basic}, each $(n-1)$\nbd simplex is a face of at most two $n$\nbd simplices. The statement then follows from \cite[Proposition 2.16]{hachimori2000combinatorics}.
\end{proof}

\begin{cor} \label{cor:cwposet}
If $P$ is a constructible directed complex, then $P_\bot$ is a CW poset.
\end{cor}

Therefore, the geometric realisation of a constructible directed complex $P$ admits the structure of a CW complex whose generators are the geometric realisations of the atoms of $P$, and this CW complex is regular.

\begin{remark} \label{rmk:nonconstructible}
Not all regular CW decompositions of balls have recursively dividable face posets; similarly, not all molecules whose geometric realisation is a regular CW ball are constructible. The following example of a non-constructible 3-molecule, whose atoms are constructible and whose geometric realisation is a 3-ball, was suggested to us by S.\ Henry; the shaded area in each diagram is the input boundary of the following 3\nbd atom.
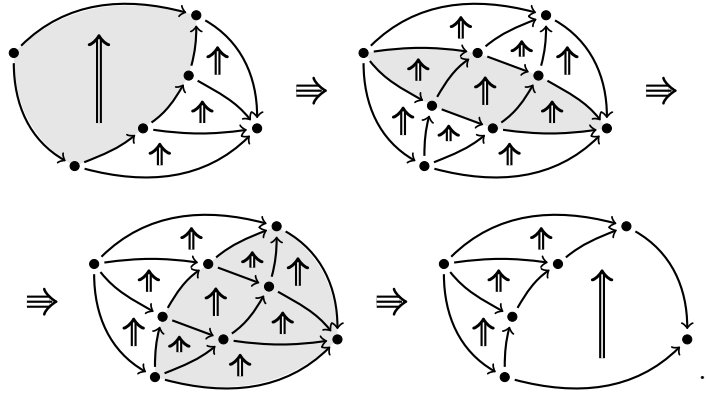
\begin{equation*}
\begin{tikzpicture}[baseline={([yshift=-.5ex]current bounding box.center)},scale=.8]
\begin{scope}
	\path[fill, color=gray!20] (-2,.625) to [out=-90,in=150] (-1,-1.25) to [out=22,in=-140] (.125,-.625) to [out=35,in=-120] (.875,.25) to [out=75,in=-90] (1,1.25) to [out=165,in=45] (-2,.625);
	\node[0c] (0) at (-2,.625) {};
	\node[0c] (s) at (-1,-1.25) {};
	\node[0c] (t) at (1,1.25) {};
	\node[0c] (d) at (.875,.25) {};
	\node[0c] (b) at (.125,-.625) {};
	\node[0c] (1) at (2,-.625) {};
	\draw[1c, out=-90, in=150] (0) to (s);
	\draw[1c, out=-30, in=90] (t) to (1);
	\draw[1c, out=-15, in=-135] (s) to (1);
	\draw[1c, out=45, in=165] (0) to (t);
	\draw[1c, out=-8, in=-172] (b) to (1);
	\draw[1c,out=-30,in=128] (d) to (1);
	\draw[1c, out=22, in=-140] (s) to (b);
	\draw[1c, out=35, in=-120] (b) to (d);
	\draw[1c, out=75, in=-90] (d) to (t);
	\draw[2c] (1.35,0) to (1.35,1);
	\draw[2c] (1.1,-.8) to (1.1,.1);
	\draw[2c] (.4,-1.5) to (.4,-.6);
	\draw[2c] (-.6,-.8) to (-.6,1.2);
	\draw[3c1] (2.25,0) to (3.5,0);
	\draw[3c2] (2.25,0) to (3.5,0);
	\draw[3c3] (2.25,0) to (3.5,0);
\end{scope}
\begin{scope}[shift={(5.75,0)}]
	\path[fill, color=gray!20] (-2,.625) to [out=-52,in=150] (-.875,-.25) to (.125,-.625) to [out=-8,in=-172] (2,-.625) to [out=128,in=-30] (.875,.25) to (-.125,.625) to [out=172,in=8] (-2,.625);
	\node[0c] (0) at (-2,.625) {};
	\node[0c] (s) at (-1,-1.25) {};
	\node[0c] (t) at (1,1.25) {};
	\node[0c] (a) at (-.875,-.25) {};
	\node[0c] (d) at (.875,.25) {};
	\node[0c] (b) at (.125,-.625) {};
	\node[0c] (c) at (-.125,.625) {};
	\node[0c] (1) at (2,-.625) {};
	\draw[1c, out=-90, in=150] (0) to (s);
	\draw[1c, out=-30, in=90] (t) to (1);
	\draw[1c, out=-15, in=-135] (s) to (1);
	\draw[1c, out=45, in=165] (0) to (t);
	\draw[1c, out=-52, in=150] (0) to (a);
	\draw[1c] (a) to (b);
	\draw[1c, out=-8, in=-172] (b) to (1);
	\draw[1c,out=8,in=172] (0) to (c);
	\draw[1c] (c) to (d);
	\draw[1c,out=-30,in=128] (d) to (1);
	\draw[1c, out=90, in=-105] (s) to (a);
	\draw[1c, out=60, in=-145] (a) to (c);
	\draw[1c, out=40, in=-158] (c) to (t);
	\draw[1c, out=22, in=-140] (s) to (b);
	\draw[1c, out=35, in=-120] (b) to (d);
	\draw[1c, out=75, in=-90] (d) to (t);
	\draw[2c] (-1.35,-1) to (-1.35,0);
	\draw[2c] (1.35,0) to (1.35,1);
	\draw[2c] (-.6,-1.1) to (-.6,-.3);
	\draw[2c] (.6,.3) to (.6,1.1);
	\draw[2c] (1.1,-.8) to (1.1,.1);
	\draw[2c] (-1.1,-.1) to (-1.1,.8);
	\draw[2c] (.4,-1.5) to (.4,-.6);
	\draw[2c] (-.4,.6) to (-.4,1.5);
	\draw[2c] (0,-.5) to (0,.5);
	\draw[3c1] (2.25,0) to (3.5,0);
	\draw[3c2] (2.25,0) to (3.5,0);
	\draw[3c3] (2.25,0) to (3.5,0);
\end{scope}
\end{tikzpicture}
\end{equation*}
\begin{equation*}
\begin{tikzpicture}[baseline={([yshift=-.5ex]current bounding box.center)},scale=.8]
\begin{scope}
	\draw[3c1] (2.25,0) to (3.5,0);
	\draw[3c2] (2.25,0) to (3.5,0);
	\draw[3c3] (2.25,0) to (3.5,0);
\end{scope}
\begin{scope}[shift={(5.75,0)}]
	\path[fill, color=gray!20] (-1,-1.25) to [out=90,in=-105] (-.875,-.25) to [out=60,in=-145] (-.125,.625) to [out=40,in=-158] (1,1.25) to [out=-30,in=90] (2,-.625) to [out=-135,in=-15] (-1,-1.25);
	\node[0c] (0) at (-2,.625) {};
	\node[0c] (s) at (-1,-1.25) {};
	\node[0c] (t) at (1,1.25) {};
	\node[0c] (a) at (-.875,-.25) {};
	\node[0c] (d) at (.875,.25) {};
	\node[0c] (b) at (.125,-.625) {};
	\node[0c] (c) at (-.125,.625) {};
	\node[0c] (1) at (2,-.625) {};
	\draw[1c, out=-90, in=150] (0) to (s);
	\draw[1c, out=-30, in=90] (t) to (1);
	\draw[1c, out=-15, in=-135] (s) to (1);
	\draw[1c, out=45, in=165] (0) to (t);
	\draw[1c, out=-52, in=150] (0) to (a);
	\draw[1c] (a) to (b);
	\draw[1c, out=-8, in=-172] (b) to (1);
	\draw[1c,out=8,in=172] (0) to (c);
	\draw[1c] (c) to (d);
	\draw[1c,out=-30,in=128] (d) to (1);
	\draw[1c, out=90, in=-105] (s) to (a);
	\draw[1c, out=60, in=-145] (a) to (c);
	\draw[1c, out=40, in=-158] (c) to (t);
	\draw[1c, out=22, in=-140] (s) to (b);
	\draw[1c, out=35, in=-120] (b) to (d);
	\draw[1c, out=75, in=-90] (d) to (t);
	\draw[2c] (-1.35,-1) to (-1.35,0);
	\draw[2c] (1.35,0) to (1.35,1);
	\draw[2c] (-.6,-1.1) to (-.6,-.3);
	\draw[2c] (.6,.3) to (.6,1.1);
	\draw[2c] (1.1,-.8) to (1.1,.1);
	\draw[2c] (-1.1,-.1) to (-1.1,.8);
	\draw[2c] (.4,-1.5) to (.4,-.6);
	\draw[2c] (-.4,.6) to (-.4,1.5);
	\draw[2c] (0,-.5) to (0,.5);
	\draw[3c1] (2.25,0) to (3.5,0);
	\draw[3c2] (2.25,0) to (3.5,0);
	\draw[3c3] (2.25,0) to (3.5,0);
\end{scope}
\begin{scope}[shift={(11.5,0)}]
	\node[0c] (0) at (-2,.625) {};
	\node[0c] (s) at (-1,-1.25) {};
	\node[0c] (t) at (1,1.25) {};
	\node[0c] (a) at (-.875,-.25) {};
	\node[0c] (c) at (-.125,.625) {};
	\node[0c] (1) at (2,-.625) {};
	\draw[1c, out=-90, in=150] (0) to (s);
	\draw[1c, out=-30, in=90] (t) to (1);
	\draw[1c, out=-15, in=-135] (s) to (1);
	\draw[1c, out=45, in=165] (0) to (t);
	\draw[1c, out=-52, in=150] (0) to (a);
	\draw[1c,out=8,in=172] (0) to (c);
	\draw[1c, out=90, in=-105] (s) to (a);
	\draw[1c, out=60, in=-145] (a) to (c);
	\draw[1c, out=40, in=-158] (c) to (t);
	\draw[2c] (-1.35,-1) to (-1.35,0);
	\draw[2c] (-1.1,-.1) to (-1.1,.8);
	\draw[2c] (-.4,.6) to (-.4,1.5);
	\draw[2c] (.6,-1.2) to (.6,.8);
	\node[scale=1.25] at (2.25,-1.25) {.};
\end{scope}
\end{tikzpicture}
\end{equation*}
\end{remark}

\begin{cons}
Restricting the geometric realisation functor to $\globe$, and taking its left Kan extension along the Yoneda embedding of $\globe$ into $\cpol$, we obtain a functor $|-|: \cpol \to \cghaus$, which sends any regular polygraph $X$ to the coend
\begin{equation*}
	\realis{X} := \int^{U \in \globe} \realis{U} \times X(U);
\end{equation*}
we call this the geometric realisation of constructible polygraphs.

This functor has a right adjoint $S: \cghaus \to \cpol$, defined on a space $X$ by
\begin{equation*}
	SX(-) := \mathrm{Hom}_{\cghaus}(\realis{-}, X);
\end{equation*}
we call $SX$ the \emph{singular constructible polygraph} of the space $X$.
\end{cons}

\begin{prop}
The functor $\realis{-}: \cpol \to \cghaus$ is monoidal from $(\cpol, \tensor, 1)$ to $(\cghaus, \times, 1)$ and from $(\cpol, \join, \emptyset)$ to $(\cghaus, \join, \emptyset)$.
\end{prop}
\begin{proof}
The functor $\realis{-}$ is compatible with lax Gray products and joins on $\globe$, and the argument of Lemma \ref{lem:restrictmonoidal} applies.
\end{proof}

\begin{thm} \label{thm:georegpoly}
Let $X$ be a constructible polygraph. Then the topological space $\realis{X}$ admits the structure of a CW complex, whose generating $n$\nbd cells are indexed by the $n$\nbd cells of $X$.
\end{thm}
\begin{proof}
Since $\realis{-}$ is a left adjoint, it takes the sequence of pushouts of Proposition \ref{prop:polygraph_ext} to a sequence of pushouts 
\begin{equation*}
\begin{tikzpicture}[baseline={([yshift=-.5ex]current bounding box.center)}]
	\node[scale=1.25] (0) at (0,2) {$\displaystyle \coprod_{x\in\atom{n}{X}} \realis{\bord{}{} U(x)}$};
	\node[scale=1.25] (1) at (3.5,2) {$\displaystyle\coprod_{x\in\atom{n}{X}} \realis{U(x)}$};
	\node[scale=1.25] (2) at (0,0) {$\realis{\skel{n-1}{X}}$};
	\node[scale=1.25] (3) at (3.5,0) {$\realis{\skel{n}{X}}$,};
	\draw[1c] (0) to node[auto,swap] {$(\realis{\bord{}{}x})_{x\in\atom{n}{X}}$} (2);
	\draw[1c] (1) to node[auto] {$(\realis{x})_{x\in\atom{n}{X}}$} (3);
	\draw[1cinc] (0) to (1);
	\draw[1cinc] (2) to (3);
	\draw[edge] (2.5,0.2) to (2.5,0.8) to (3.3,0.8);
\end{tikzpicture}
\end{equation*}
and by Theorem \ref{thm:regularball} the $\realis{U(x)}$ are homeomorphic to $D^n$ and the $\realis{\bord{}{}U(x)}$ to $\bord{}{}D^n$. By induction on $n$\nbd skeleta, we see that this gives $\realis{X}$ the appropriate structure of a CW complex. 
\end{proof}

\bibliographystyle{alpha}
\small \bibliography{main}

\normalsize
\appendix
\section{Lax Gray product of constructible molecules} \label{sec:appendix}

In what follows, let $n \lor m$ denote the greatest, and $n \land m$ the least between two natural numbers $n, m$.

\begin{lem}
Let $U$ be a constructible $n$\nbd molecule and $V$ a constructible $m$\nbd molecule. For all $k < n+m$,
\begin{equation} \label{eq:grayboundary}
	\underbrace{\bord{}{\alpha}\ldots\bord{}{\beta}}_{n+m-k}(U \tensor V) = \bord{k}{\alpha}(U \tensor V) = \bigcup_{i=(k-m) \lor 0}^{n \land k} \bord{i}{\alpha}U \tensor \bord{k-i}{(-)^i\alpha}V.
\end{equation}
\end{lem}
\begin{proof}
If $n = 0$ or $m = 0$, $U \tensor V$ is isomorphic to $V$ and $U$, respectively, and the equation becomes $\bord{k}{\alpha}(U \tensor V) = U \tensor \bord{k}{\alpha}V$ and $\bord{k}{\alpha}(U \tensor V) = \bord{k}{\alpha}U \tensor V$, respectively.

Suppose $n, m > 0$. The product $U \tensor V$ is pure and $(n+m)$\nbd dimensional, so its boundary $\bord{}{\alpha}(U \tensor V)$ is equal to $\clos\sbord{}{\alpha}(U \tensor V)$, and by construction
\begin{equation*}
	\sbord{}{\alpha}(U \tensor V) = (\sbord{}{\alpha}U \tensor V) + (U \tensor \sbord{}{(-)^n\alpha}V).
\end{equation*}
Ths proves (\ref{eq:grayboundary}) for $k = n+m-1$. 

Supposing we have proved the equation for $k+1$, $\bord{k+1}{\beta}(U \tensor V)$ is pure and $(k+1)$\nbd dimensional as a union of pure $(k+1)$\nbd dimensional sets. It then suffices to identify $\sbord{}{\alpha}(\bord{k+1}{\beta}(U \tensor V))$. Suppose $x \tensor y$ is $k$\nbd dimensional with $\dmn{x} = i$ and $\dmn{y} = k-i$. If $x \tensor y \in \sbord{}{\alpha}(\bord{k+1}{\beta}(U \tensor V))$, and $x'$ covers $x$, then $x' \tensor y$ covers $x \tensor y$ with orientation $\alpha$, so $x'$ covers $x$ with orientation $\alpha$. Therefore $x \in \sbord{i}{\alpha}U$. Similarly, if $y'$ covers $y$, then $x \tensor y'$ covers $x \tensor y$ with orientation $\alpha$, so $y'$ covers $y$ with orientation $(-)^i\alpha$, and $y \in \sbord{k-i}{(-)^i\alpha}V$. 

This proves one inclusion, and the converse is proved similarly. The proof does not change if we assume $x \tensor y \in \sbord{k}{\alpha}(U \tensor V)$. 
\end{proof}

\begin{lem} \label{lem:globtensorapp}
Let $U$ be a constructible $n$\nbd molecule and $V$ a constructible $m$\nbd molecule. Then:
\begin{enumerate}[label=(\alph*)]
	\item $U \tensor V$ is a constructible $(n+m)$\nbd molecule;
	\item if $n,m >0 $, $\bord{}{\alpha}(U \tensor V)$ splits as $(\bord{}{\alpha}U \tensor V) \cup (U \tensor \bord{}{(-)^n\alpha}V)$;
	\item if $U' \sqsubseteq U$ and $V' \sqsubseteq V$, then $U' \tensor V' \sqsubseteq U \tensor V$.
\end{enumerate}
\end{lem}
\begin{proof}
We proceed by double induction on the dimension and number of maximal elements of $U$ and $V$. If $U$ or $V$ is 0-dimensional, then $U \tensor V$ is isomorphic to $V$ or $U$, respectively, and there is nothing to prove. 

Suppose $n, m > 0$; then $U \tensor V$ is pure and $(n+m)$\nbd dimensional. We first need to show that $\bord{}{\alpha}(U \tensor V)$ is a constructible molecule, splitting as in $(b)$. This will be a consequence of the following, more general statement.
\begin{sublem} \label{globtensor_sub1}
For all $k < n+m$, $\underbrace{\bord{}{\alpha}\ldots\bord{}{\alpha}}_{n+m-k}(U \tensor V)$ is a constructible $k$\nbd molecule, and it has $\bord{k}{\alpha}(\bord{}{\alpha}U \tensor V)$ and $\bord{k}{\alpha}(U \tensor \bord{}{(-)^n\alpha}V)$ as constructible submolecules.
\end{sublem}
\begin{proof}[Proof of the Sub-Lemma]
If $k < n$, only non-trivial boundaries of $U$ appear in (\ref{eq:grayboundary}), which is therefore equal to 
\begin{equation*}
	\underbrace{\bord{}{\alpha}\ldots\bord{}{\alpha}}_{(n+m-1)-k}(\bord{}{\alpha}U \tensor V) = \bord{k}{\alpha}(\bord{}{\alpha}U \tensor V),
\end{equation*} 
a constructible $k$\nbd molecule by the inductive hypothesis that $\bord{}{\alpha}U \tensor V$ is a constructible molecule, and trivially a constructible submolecule of itself. Similarly if $k < m$ only non-trivial boundaries of $V$ appear, and by the inductive hypothesis $\bord{k}{\alpha}(U \tensor \bord{}{(-)^n\alpha}V)$ is a constructible $k$\nbd molecule and a constructible submolecule of itself.

Suppose $k \geq n$. Then (\ref{eq:grayboundary}) splits as
\begin{equation} \label{eq:graykdecomp}
(U \tensor \bord{k-n}{(-)^n\alpha}V) \cup \bigcup_{i=(k-m)\lor 0}^{n-1} \bord{i}{\alpha}U \tensor \bord{k-i}{(-)^i\alpha}V = (U \tensor \bord{k-n}{(-)^n\alpha}V) \cup \bord{k}{\alpha}(\bord{}{\alpha}U \tensor V).
\end{equation}
Both sets in the right-hand side are constructible $k$\nbd molecules by the inductive hypothesis.
Moreover,
\begin{align*}
	(U \tensor \bord{k-n}{(-)^n\alpha}V) \cap \bord{k}{\alpha}(\bord{}{\alpha}U \tensor V) & = \bord{}{\alpha}U \tensor \bord{k-n}{(-)^n\alpha}V \\
	& = \bord{}{\alpha}(U \tensor \bord{k-n}{(-)^n\alpha}V) \cap \bord{k-1}{-\alpha}(\bord{}{\alpha}U \tensor V).
\end{align*}
By the inductive hypothesis on the main statement, $\bord{}{\alpha}U \tensor \bord{k-n}{(-)^n\alpha}V$ is a constructible submolecule of $\bord{}{\alpha}(U \tensor \bord{k-n}{(-)^n\alpha}V)$ and of $\bord{k-1}{-\alpha}(\bord{}{\alpha}U \tensor V)$, while by the inductive hypothesis for the sub-lemma
\begin{align*}
	\bord{k-1}{\alpha}(\bord{}{\alpha}U \tensor V) & \sqsubseteq \underbrace{\bord{}{\alpha}\ldots\bord{}{\alpha}}_{n+m-(k-1)}(U \tensor V), \\
	\bord{}{-\alpha}(U \tensor \bord{k-n}{(-)^n\alpha}V) & \sqsubseteq \underbrace{\bord{}{-\alpha}\ldots\bord{}{-\alpha}}_{n+m-(k-1)}(U \tensor V).
\end{align*}
This proves that $\underbrace{\bord{}{\alpha}\ldots\bord{}{\alpha}}_{n+m-k}(U \tensor V)$ is a constructible molecule, splitting as in (\ref{eq:graykdecomp}), with $\bord{k}{\alpha}(\bord{}{\alpha}U \tensor V)$ as a constructible submolecule. Similarly, when $k \geq m$, from the decomposition of (\ref{eq:grayboundary})
\begin{equation*}
(\bord{k-m}{\alpha}U \tensor V) \cup \bigcup_{i=k-(m-1)}^{n\land k} \bord{i}{\alpha}U \tensor \bord{k-i}{(-)^i\alpha}V = (\bord{k-m}{\alpha}U \tensor V) \cup \bord{k}{\alpha}(U \tensor \bord{}{(-)^n\alpha}V)
\end{equation*}
we obtain that $\bord{k}{\alpha}(U \tensor \bord{}{(-)^n\alpha}V)$ is a constructible submolecule.
\end{proof}

If $U$ and $V$ are atoms, there is nothing else to prove. Suppose $U$ is non-atomic and splits as $U_1 \cup U_2$, so $U \tensor V = (U_1 \tensor V) \cup (U_2 \tensor V)$. By the inductive hypothesis, $(U_1 \tensor V)$ and $(U_2 \tensor V)$ are both constructible $(n+m)$\nbd molecules, their intersection $(U_1 \cap U_2) \tensor V$ is a constructible $(n+m-1)$\nbd molecule, and
\begin{align*}
	(U_1 \tensor V) \cap (U_2 \tensor V) & \sqsubseteq \bord{}{+}U_1 \tensor V \sqsubseteq \bord{}{+}(U_1 \tensor V), \\
	(U_1 \tensor V) \cap (U_2 \tensor V) & \sqsubseteq \bord{}{-}U_2 \tensor V \sqsubseteq \bord{}{-}(U_2 \tensor V).
\end{align*}
It only remains to show the following.

\begin{sublem}
$\bord{}{-}(U_1 \tensor V) \sqsubseteq \bord{}{-}(U \tensor V)$ and $\bord{}{+}(U_2 \tensor V) \sqsubseteq \bord{}{+}(U \tensor V)$.
\end{sublem}
\begin{proof}[Proof of the Sub-Lemma]
We can assume the statement is true for $U \tensor V'$ where $V'$ has a lower dimension than $V$. Split $\bord{}{-}(U \tensor V)$ as $\tilde{U}_1 \cup \tilde{U}_2$, where 
\begin{equation*}
	\tilde{U}_1 := (\bord{}{-}U \tensor V) \cup (U_1 \tensor \bord{}{(-)^{n-1}}V), \quad \quad \quad \tilde{U}_2 := U_2 \tensor \bord{}{(-)^{n-1}}V. 
\end{equation*}
We want to prove that this is a decomposition into constructible submolecules, and that $\bord{}{-}(U_1 \tensor V) \sqsubseteq \tilde{U}_1$. First of all,
\begin{equation*}
	\tilde{U}_1 \cap \tilde{U}_2 = \bord{}{-}U_2 \tensor \bord{}{(-)^{n-1}}V
\end{equation*}
is a constructible molecule, and a constructible submolecule of $\bord{}{-}\tilde{U}_2$. By the inductive hypothesis for the sub-lemma, we also have 
\begin{equation*}
	\bord{}{+}\tilde{U}_2 = \bord{}{+}(U_2 \tensor \bord{}{(-)^{n-1}}V) \sqsubseteq \bord{}{+}(U \tensor \bord{}{(-)^{n-1}}V) \sqsubseteq \bord{n+m-2}{+}(U \tensor V).
\end{equation*}
We still have to show that $\bord{}{-}\tilde{U}_1 \sqsubseteq \bord{n+m-2}{-}(U \tensor V)$, that $\tilde{U}_1 \cap \tilde{U}_2 \sqsubseteq \bord{}{+}\tilde{U}_1$, and that $\tilde{U}_1$ is a constructible molecule with $\bord{}{-}(U_1 \tensor V) \sqsubseteq \tilde{U}_1$.

First, let us show that the boundaries $\bord{}{\alpha}\tilde{U}_1$ are constructible molecules. For all $k < n$, we have
\begin{align*}
	\underbrace{\bord{}{-}\ldots\bord{}{-}}_{(n+m-1)-k}\tilde{U}_1 & = \bord{k}{-}(\bord{}{-}U \tensor V), \\
	\underbrace{\bord{}{+}\ldots\bord{}{+}}_{(n+m-1)-k}\tilde{U}_1 & = \bord{k}{+}((\bord{}{+}U_1 \cup \bord{}{-}U_2) \tensor V),
\end{align*}
which are constructible molecules by the inductive hypothesis and by Lemma \ref{lem:boundarymove}. For $k \geq n$,
\begin{align*}
	\underbrace{\bord{}{-}\ldots\bord{}{-}}_{(n+m-1)-k}\tilde{U}_1 & = \bord{k}{-}(\bord{}{-}U \tensor V) \cup (U_1 \tensor \bord{k-n}{(-)^{n-1}}V), \\
	\underbrace{\bord{}{+}\ldots\bord{}{+}}_{(n+m-1)-k}\tilde{U}_1 & = \bord{k}{+}((\bord{}{+}U_1 \cup \bord{}{-}U_2) \tensor V) \cup (U_1 \tensor \bord{k-n}{(-)^{n}}V),
\end{align*} 
and by an inductive step similar to the proof of Sub-Lemma \ref{globtensor_sub1}, these are decompositions into constructible submolecules. Moreover,
\begin{equation*}
	\tilde{U}_1 \cap \tilde{U}_2 \sqsubseteq (\bord{}{+}U_1 \cup \bord{}{-}U_2) \tensor \bord{}{(-)^{n-1}}V \sqsubseteq \bord{}{+}((\bord{}{+}U_1 \cup \bord{}{-}U_2) \tensor V) \sqsubseteq \bord{}{+}\tilde{U}_1.
\end{equation*}

Next, let us show that $\tilde{U}_1$ is a constructible molecule with $\bord{}{-}(U_1 \tensor V) \sqsubseteq \tilde{U}_1$. We will show that, for all $\bord{}{-}U_1 \sqsubseteq W \sqsubseteq \bord{}{-}U$,
\begin{equation*}
	\tilde{W} := (W \tensor V) \cup (U_1 \tensor \bord{}{(-)^{n-1}}V)
\end{equation*}
is a constructible molecule, and $\bord{}{-}(U_1 \tensor V) \sqsubseteq \tilde{W}$, by induction on increasing $W$. For $W = \bord{}{-}U_1$, we have $\tilde{W} = \bord{}{-}(U_1 \tensor V)$, and we are done.

Suppose that $W \sqsubseteq \bord{}{-}U$ splits as $W_1 \cup W_2$, with $\bord{}{-}U_1 \sqsubseteq W_i$, and $\bord{}{-}(U_1 \tensor V) \sqsubseteq \tilde{W_i}$; without loss of generality, let $i = 1$. Then, $\tilde{W} = \tilde{W}_1 \cup \tilde{W}_2$, for
\begin{equation*}
	\tilde{W}_1 := (W_1 \tensor V) \cup (U_1 \tensor \bord{}{(-)^{n-1}}V), \quad \quad \quad \tilde{W}_2 := W_2 \tensor V;
\end{equation*}
both of these are constructible molecules, and so is their intersection 
\begin{equation*}
	\tilde{W}_1 \cap \tilde{W}_2 = (W_1 \cap W_2) \tensor V.
\end{equation*}
The relations $\tilde{W}_1 \cap \tilde{W}_2 \sqsubseteq \bord{}{-}\tilde{W}_2$ and $\bord{}{+}\tilde{W}_2 \sqsubseteq \bord{}{+}\tilde{W}$ are both immediate. For $k < n$,
\begin{equation*}
	\bord{k}{+}\tilde{W}_1 = \bord{k}{+}(W_1[\bord{}{+}U_1/\bord{}{-}U_1] \tensor V),
\end{equation*}
while for $k \geq n$
\begin{equation*}
	\bord{k}{+}\tilde{W}_1 = \bord{k}{+}(W_1[\bord{}{+}U_1/\bord{}{-}U_1] \tensor V) \cup (U_1 \tensor \bord{k-n}{(-)^{n}}V),
\end{equation*}
and we show by another induction in the style of Sub-Lemma \ref{globtensor_sub1} that these are decompositions into constructible submolecules. Then
\begin{equation*}
	\tilde{W}_1 \cap \tilde{W}_2 \sqsubseteq \bord{}{+}W_1 \tensor V = \bord{}{+}W_1[\bord{}{+}U_1/\bord{}{-}U_1] \tensor V \sqsubseteq \bord{}{+}\tilde{W}_1.
\end{equation*}
Moreover, for $k < n$ and $k \geq n$, respectively,
\begin{equation*}
	\bord{k}{-}\tilde{W}_1 = \bord{k}{-}(W_1 \tensor V) \quad \text{ or } \quad \bord{k}{-}(W_1 \tensor V) \cup (U_1 \tensor \bord{k-n}{(-)^{n-1}}V)
\end{equation*}
and another inductive argument leads us to $\bord{}{-}\tilde{W_1} \sqsubseteq \bord{}{-}\tilde{W}$. 

This proves that $\tilde{W} = \tilde{W}_1 \cup \tilde{W}_2$ is a decomposition into constructible submolecules, and $\bord{}{-}U_1 \sqsubseteq W$. Because chains $\bord{}{-}U_1 \sqsubset \ldots \sqsubset \bord{}{-}U$ are finite, $\tilde{U}_1$ is a constructible molecule and $\bord{}{-}(U_1 \tensor V) \sqsubseteq \tilde{U}_1$.

The only thing left to check is that $\bord{}{-}\tilde{U}_1 \sqsubseteq \bord{n+m-2}{-}(U \tensor V)$: this is immediate if $m = 1$; otherwise, we have reduced the problem of showing
\begin{equation*}
	(\bord{}{-}U \tensor V) \cup (U_1 \tensor \bord{}{(-)^{n-1}}V) \sqsubseteq (\bord{}{-}U \tensor V) \cup (U \tensor \bord{}{(-)^{n-1}}V)
\end{equation*}
to showing that
\begin{equation*}
	\bord{}{-}(\bord{}{-}U \tensor V) \cup (U_1 \tensor \bord{m-2}{(-)^{n-1}}V) \sqsubseteq \bord{}{-}(\bord{}{-}U \tensor V) \cup (U \tensor \bord{m-2}{(-)^{n-1}}V),
\end{equation*}
and we can turn this into another inductive argument in the style of Sub-Lemma \ref{globtensor_sub1}. 

We have proved that $\bord{}{-}(U_1 \tensor V) \sqsubseteq \tilde{U}_1 \sqsubseteq \bord{}{-}(U \tensor V)$. A dual argument proves that $\bord{}{+}(U_2 \tensor V) \sqsubseteq \bord{}{+}(U \tensor V)$.
\end{proof}

This completes the proof that $U \tensor V = (U_1 \tensor V) \cup (U_2 \tensor V)$ is a decomposition into constructible submolecules. By a dual argument, we prove that if $V$ splits as $V_1 \cup V_2$, then $U \tensor V$ splits as $(U \tensor V_1) \cup (U \tensor V_2)$, and we are done.
\end{proof}

\section{Representable constructible polygraphs} \label{sec:represent}

The theory developed in this work may serve as a basis for an approach to weak higher categories, modelled by constructible polygraphs that satisfy a \emph{representability} condition, similar to opetopic higher categories. In the light of \cite{henry2018regular}, this could be a good framework to study semi-strictification in the sense of C.\ Simpson's conjecture, beyond the ``groupoidal'' case (all cells are weakly invertible). In \cite{hadzihasanovic2018weak}, we developed this programme in the special case of dimension 2, proving an equivalence between bicategories and representable merge-bicategories --- a truncation of representable constructible polygraphs to dimension 2 --- which we then exploited in a semi-strictification argument.

For the interested reader, we give some basic definitions here, without much comment. In what follows, we identify a constructible directed complex with its embedding into $\cpol$. 

\begin{dfn} 
Let $X$ be a constructible polygraph and $U$ a constructible $n$\nbd molecule. A \emph{constructible $n$\nbd diagram} $x$ of shape $U$ in $X$ is a map $x: U \to X$. If $U$ is an atom, we call $x$ an \emph{$n$\nbd cell} of $X$.
\end{dfn}

\begin{dfn}
For $n > 0$, a \emph{ternary} $(n+1)$\nbd atom is a constructible $(n+1)$\nbd atom with three $n$\nbd dimensional elements. Necessarily, one boundary of $W$ splits into two atoms $W_+$ and $W_-$, with $W_+ \cap W_- \subseteq \bord{}{+}W_+ \cap \bord{}{-}W_-$, and the other boundary is a single atom $W_0$. 

A \emph{horn} of $W$ is any of the following subsets:
\begin{equation*}
	\Lambda^W_0 := W_+ \cup W_-, \quad \quad \Lambda^W_- := W_+ \cup W_0, \quad \quad \Lambda^W_+ := W_- \cup W_0.
\end{equation*}
The horn $\Lambda^W_\zeta \incl W$ is called a \emph{composition horn} if $\zeta = 0$, a \emph{left division horn} if $\zeta = -$, and a \emph{right division horn} if $\zeta = +$.
\end{dfn}

\begin{dfn}
Let $X$ be a constructible polygraph. A \emph{horn of $W$ in $X$} is a pair of a horn $\Lambda \incl W$ and a morphism $\lambda: \Lambda \to X$. A \emph{filler} for the horn is an $(n+1)$\nbd cell $h: W \to X$ such that
\begin{equation*}
\begin{tikzpicture}[baseline={([yshift=-.5ex]current bounding box.center)}]
	\node[scale=1.25] (0) at (0,1.5) {$\Lambda$};
	\node[scale=1.25] (1) at (2,1.5) {$X$};
	\node[scale=1.25] (2) at (0,0) {$W$};
	\draw[1c] (0) to node[auto] {$\lambda$} (1);
	\draw[1cinc] (0) to (2);
	\draw[1c] (2) to node[auto,swap] {$h$} (1);
\end{tikzpicture}
\end{equation*}
commutes. 

Let $x: U \to X$ be an $n$\nbd cell of $X$ and $V \sqsubseteq \bord{}{\alpha}U$ a constructible submolecule of its boundary. A \emph{horn for $x$ at $V$} is a pair of an inclusion $\imath: U \incl W$ such that $\imath(U) = W_\alpha$ and $\imath(V) = W_+ \cap W_-$, and a horn $\lambda: \Lambda^W_\zeta \to X$ with $\zeta \in \{-\alpha,0\}$ such that
\begin{equation*}
\begin{tikzpicture}[baseline={([yshift=-.5ex]current bounding box.center)}]
	\node[scale=1.25] (0) at (0,1.5) {$U$};
	\node[scale=1.25] (1) at (2,0) {$X$};
	\node[scale=1.25] (2) at (0,0) {$\Lambda^W_\zeta$};
	\draw[1c] (0) to node[auto] {$x$} (1);
	\draw[1cinc] (0) to (2);
	\draw[1c] (2) to node[auto,swap] {$\lambda$} (1);
\end{tikzpicture}
\end{equation*}
commutes.
\end{dfn}

\begin{remark} \label{rmk:dual_horn}
If $W$ is a ternary $(n+1)$\nbd atom, then $\oppn{n+1}{W}$ is still ternary, and every horn of $W$ is isomorphic to a horn of $\oppn{n+1}{W}$. Thus, for each horn for $x: U \to X$ at $V$, given by $U \incl \Lambda \incl W$ and $\lambda: \Lambda \to X$, there is a \emph{dual} horn, where $\Lambda$ is seen as a horn of $\oppn{n+1}{W}$.
\end{remark}

\begin{dfn} \label{dfn:universality}
Let $X$ be a constructible polygraph. Coinductively, we define families $\divis{n}{X}$ of pairs $(x \in X(U), V \sqsubseteq \bord{}{\alpha}U)$ of an $n$\nbd cell of shape $U$ and a constructible submolecule of $\bord{}{\alpha}U$, to be read ``$x$ is \emph{universal} at $V \sqsubseteq \bord{}{\alpha}U$''. If $\divis{n}{P}$ is defined, we also define a family $\equi{n}{X}$ of $n$\nbd cells, the \emph{$n$-equivalences} of $X$, by
\begin{equation*}
	\equi{n}{X} := \{x \in X(U) \,|\, \dmn{U} = n, (x, \bord{}{+}U) \text{ and } (x, \bord{}{-}U) \in \divis{n}{X}\}.
\end{equation*} 

We let $(x, V \sqsubseteq \bord{}{\alpha}U) \in \divis{n}{X}$ if and only if all horns $\Lambda \incl W, \lambda: \Lambda \to X$ for $x$ at $V$ have a filler $h: W \to X$ which
\begin{enumerate}
	\item is an $(n+1)$\nbd equivalence, and
	\item is universal at $W_{-\alpha}$.
\end{enumerate}
\end{dfn}

\begin{remark}
Similarly to simplicial and opetopic approaches to higher categories, the idea is that an equivalence in $X$ of shape $W$, a ternary atom, exhibits the image of $W_0$ as a weak composite of the images of $W_+$ and $W_-$. 

The notion of a cell of shape $U$ universal at $\bord{}{\alpha}U$ is analogous to a universal cell in the opetopic or multitopic approach \cite{baez1998higher}. One novelty of our approach is to consider, at the same time, universality at a part of the boundary, which captures universal properties such as those of Kan extensions and of Kan lifts. 

We think that our second universality requirement on horn fillers may subsume, with a purely existential statement, what was achieved through a universal quantification on higher-dimensional universal cells in the opetopic approach. This allows us to obtain a proper coinductive definition, and tackle genuinely infinite-dimensional higher categories, whereas the opetopic approach had to resort to a truncation to dimension $n$.
\end{remark}

\begin{exm} 
A 1-cell $e$ of a constructible polygraph $X$, which is of shape $\vec{I}$, can only be universal at $\bord{}{+}\vec{I}$ or at $\bord{}{-}\vec{I}$. If it is universal at $\bord{}{+}\vec{I}$, the relevant horns are those of the forms
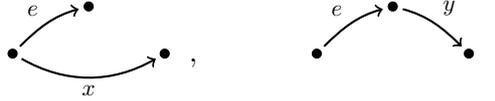
\begin{equation*} 
\begin{tikzpicture}[baseline={([yshift=-.5ex]current bounding box.center)}]
\begin{scope}
	\node[0c] (0) at (-1, 0) {};
	\node[0c] (1) at (1,0) {};
	\node[0c] (o) at (0,.625) {};
	\draw[1c,out=-30,in=-150] (0) to node[auto,swap] {$x$} (1);
	\draw[1c,out=45,in=-165] (0) to node[auto] {$e$} (o);
	\node[scale=1.25] at (1.375,-.125) {,};
\end{scope}
\begin{scope}[shift={(4,0)}]
	\node[0c] (0) at (-1, 0) {};
	\node[0c] (1) at (1,0) {};
	\node[0c] (o) at (0,.625) {};
	\draw[1c,out=-15,in=135] (o) to node[auto] {$y$} (1);
	\draw[1c,out=45,in=-165] (0) to node[auto] {$e$} (o);
	\node[scale=1.25] at (1.375,-.125) {,};
\end{scope}
\end{tikzpicture}
\end{equation*}
which must have fillers
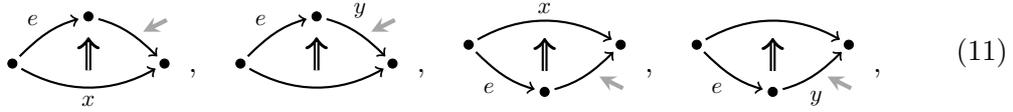
\begin{equation} \label{eq:1horn}
\begin{tikzpicture}[baseline={([yshift=-.5ex]current bounding box.center)}]
\begin{scope}
	\node[0c] (0) at (-1, 0) {};
	\node[0c] (1) at (1,0) {};
	\node[0c] (o) at (0,.625) {};
	\draw[1c,out=-30,in=-150] (0) to node[auto,swap] {$x$} (1);
	\draw[1c,out=-15,in=135] (o) to (1);
	\draw[1c,out=45,in=-165] (0) to node[auto] {$e$} (o);
	\draw[2c] (0,-.3) to (o);
	\draw[follow] (1.125,.625) to (.625,.375);
	\node[scale=1.25] at (1.375,-.125) {,};
\end{scope}
\begin{scope}[shift={(3,0)}]
	\node[0c] (0) at (-1, 0) {};
	\node[0c] (1) at (1,0) {};
	\node[0c] (o) at (0,.625) {};
	\draw[1c,out=-30,in=-150] (0) to (1);
	\draw[1c,out=-15,in=135] (o) to node[auto,pos=.3] {$y$} (1);
	\draw[1c,out=45,in=-165] (0) to node[auto] {$e$} (o);
	\draw[2c] (0,-.3) to (o);
	\draw[follow] (1.125,.625) to (.625,.375);
	\node[scale=1.25] at (1.375,-.125) {,};
\end{scope}
\begin{scope}[shift={(6,0)}]
	\node[0c] (0) at (-1, .25) {};
	\node[0c] (1) at (1,.25) {};
	\node[0c] (i) at (0,-.375) {};
	\draw[1c,out=30,in=150] (0) to node[auto] {$x$} (1);
	\draw[1c,out=15,in=-135] (i) to (1);
	\draw[1c,out=-45,in=165] (0) to node[auto,swap] {$e$} (i);
	\draw[2c] (i) to (0,.55);
	\draw[follow] (1.125,-.375) to (.625,-.125);
	\node[scale=1.25] at (1.375,-.125) {,};
\end{scope}
\begin{scope}[shift={(9,0)}]
	\node[0c] (0) at (-1, .25) {};
	\node[0c] (1) at (1,.25) {};
	\node[0c] (i) at (0,-.375) {};
	\draw[1c,out=30,in=150] (0) to (1);
	\draw[1c,out=15,in=-135] (i) to node[auto,swap,pos=0.3] {$y$} (1);
	\draw[1c,out=-45,in=165] (0) to node[auto,swap] {$e$} (i);
	\draw[2c] (i) to (0,.55);
	\draw[follow] (1.125,-.375) to (.625,-.125);
	\node[scale=1.25] at (1.375,-.125) {,};
\end{scope}
\end{tikzpicture}
\end{equation}
all of which are both equivalences, and universal at the constructible submolecule indicated by the grey arrow. 

In merge-bicategories, seen as ``2-truncated'' constructible polygraphs, $e$ is what we called (tensor and par) left universal in \cite{hadzihasanovic2018weak}. Intuitively, existence of fillers for the first horn mean that compatible 1-cells can be \emph{factorised} as $e$ followed by another 1-cell, while the existence of fillers for the second horn mean that compatible 1-cells can be \emph{composed} with $e$ on the left. The fact that the 2-cells in (\ref{eq:1horn}) have pairwise the same universal properties implies that by factorising then composing, or composing then factorising, we will obtain the same result, up to equivalence.

Dually, if $e$ is universal at $\bord{}{-}\vec{I}$, the relevant horns are
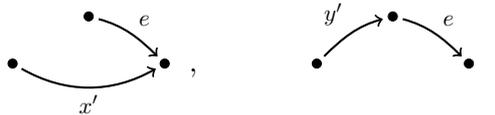
\begin{equation*} 
\begin{tikzpicture}[baseline={([yshift=-.5ex]current bounding box.center)}]
\begin{scope}
	\node[0c] (0) at (-1, 0) {};
	\node[0c] (1) at (1,0) {};
	\node[0c] (o) at (0,.625) {};
	\draw[1c,out=-30,in=-150] (0) to node[auto,swap] {$x'$} (1);
	\draw[1c,out=-15,in=135] (o) to node[auto] {$e$} (1);
	\node[scale=1.25] at (1.375,-.125) {,};
\end{scope}
\begin{scope}[shift={(4,0)}]
	\node[0c] (0) at (-1, 0) {};
	\node[0c] (1) at (1,0) {};
	\node[0c] (o) at (0,.625) {};
	\draw[1c,out=-15,in=135] (o) to node[auto] {$e$} (1);
	\draw[1c,out=45,in=-165] (0) to node[auto] {$y'$} (o);
	\node[scale=1.25] at (1.375,-.125) {,};
\end{scope}
\end{tikzpicture}
\end{equation*}
with fillers
\begin{equation} \label{eq:1hornR}
\begin{tikzpicture}[baseline={([yshift=-.5ex]current bounding box.center)}]
\begin{scope}
	\node[0c] (0) at (-1, 0) {};
	\node[0c] (1) at (1,0) {};
	\node[0c] (o) at (0,.625) {};
	\draw[1c,out=-30,in=-150] (0) to node[auto,swap] {$x'$} (1);
	\draw[1c,out=-15,in=135] (o) to node[auto] {$e$} (1);
	\draw[1c,out=45,in=-165] (0) to (o);
	\draw[2c] (0,-.3) to (o);
	\draw[follow] (-1.125,.625) to (-.625,.375);
	\node[scale=1.25] at (1.375,-.125) {,};
\end{scope}
\begin{scope}[shift={(3,0)}]
	\node[0c] (0) at (-1, 0) {};
	\node[0c] (1) at (1,0) {};
	\node[0c] (o) at (0,.625) {};
	\draw[1c,out=-30,in=-150] (0) to (1);
	\draw[1c,out=-15,in=135] (o) to node[auto] {$e$} (1);
	\draw[1c,out=45,in=-165] (0) to node[auto,pos=.7] {$y'$} (o);
	\draw[2c] (0,-.3) to (o);
	\draw[follow] (-1.125,.625) to (-.625,.375);
	\node[scale=1.25] at (1.375,-.125) {,};
\end{scope}
\begin{scope}[shift={(6,0)}]
	\node[0c] (0) at (-1, .25) {};
	\node[0c] (1) at (1,.25) {};
	\node[0c] (i) at (0,-.375) {};
	\draw[1c,out=30,in=150] (0) to node[auto] {$x'$} (1);
	\draw[1c,out=15,in=-135] (i) to node[auto,swap] {$e$} (1);
	\draw[1c,out=-45,in=165] (0) to (i);
	\draw[2c] (i) to (0,.55);
	\draw[follow] (-1.125,-.375) to (-.625,-.125);
	\node[scale=1.25] at (1.375,-.125) {,};
\end{scope}
\begin{scope}[shift={(9,0)}]
	\node[0c] (0) at (-1, .25) {};
	\node[0c] (1) at (1,.25) {};
	\node[0c] (i) at (0,-.375) {};
	\draw[1c,out=30,in=150] (0) to (1);
	\draw[1c,out=15,in=-135] (i) to node[auto,swap] {$e$} (1);
	\draw[1c,out=-45,in=165] (0) to node[auto,swap,pos=.7] {$y'$} (i);
	\draw[2c] (i) to (0,.55);
	\draw[follow] (-1.125,-.375) to (-.625,-.125);
	\node[scale=1.25] at (1.375,-.125) {;};
\end{scope}
\end{tikzpicture}
\end{equation}
in a merge-bicategory, $e$ is what we called (tensor and par) right universal. Finally, $e$ is an equivalence if both the fillers (\ref{eq:1horn}) and (\ref{eq:1hornR}) exist.

A 2-cell $t$ of $X$ has shape $U^{(n,m)}$ for some $n, m >0$, with $\bord{}{-}U^{(n,m)} = \#_0^n \vec{I}$ and $\bord{}{+}U^{(n,m)} = \#_0^m \vec{I}$. The submolecules $W_{[i_1,i_2]}$ of $\bord{}{-}U^{(n,m)}$ can be identified with sub-intervals $[i_1,i_2]$ of $[1,n]$, and the submolecules $W_{[j_1,j_2]}$ of $\bord{}{+}U^{(n,m)}$ with sub-intervals $[j_1,j_2]$ of $[1,m]$; in a merge-bicategory, universality of $t$ at $W_{[i_1,i_2]}$ or at $W_{[j_1,j_2]}$ is what we called universality at $\bord{[i_1,i_2]}{-}$ or at $\bord{[j_1,j_2]}{+}$ in \cite{hadzihasanovic2018weak}. 

To see how the universality of horn fillers can subsume a condition of uniqueness of factorisations, consider the first 2-cell in $(\ref{eq:1horn})$, which we relabel
\begin{equation} \label{eq:uniqueness1}
\begin{tikzpicture}[baseline={([yshift=-.5ex]current bounding box.center)}]
\begin{scope}
	\node[0c] (0) at (-1, 0) {};
	\node[0c] (1) at (1,0) {};
	\node[0c] (o) at (0,.625) {};
	\draw[1c,out=-30,in=-150] (0) to node[auto,swap] {$x$} (1);
	\draw[1c,out=-15,in=135] (o) to node[auto] {$e\backslash x$} (1);
	\draw[1c,out=45,in=-165] (0) to node[auto] {$e$} (o);
	\draw[2c] (0,-.3) to node[auto] {$t\;$} (o);
	\node[scale=1.25] at (1.375,-.125) {.};
\end{scope}
\end{tikzpicture}
\end{equation}
If there is another 2-cell
\begin{equation*} 
\begin{tikzpicture}[baseline={([yshift=-.5ex]current bounding box.center)}]
\begin{scope}
	\node[0c] (0) at (-1, 0) {};
	\node[0c] (1) at (1,0) {};
	\node[0c] (o) at (0,.625) {};
	\draw[1c,out=-30,in=-150] (0) to node[auto,swap] {$x$} (1);
	\draw[1c,out=-15,in=135] (o) to node[auto] {$z$} (1);
	\draw[1c,out=45,in=-165] (0) to node[auto] {$e$} (o);
	\draw[2c] (0,-.3) to node[auto] {$s\;$} (o);
	\node[scale=1.25] at (1.375,-.125) {,};
\end{scope}
\end{tikzpicture}
\end{equation*}
the two together form a horn in $X$, which by the universal property of $t$ has a 3-dimensional filler
\begin{equation} \label{eq:uniqueness3}
\begin{tikzpicture}[baseline={([yshift=-.5ex]current bounding box.center)}]
\begin{scope}
	\node[0c] (0) at (-1, 0) {};
	\node[0c] (1) at (1,0) {};
	\node[0c] (o) at (0,1.5) {};
	\draw[1c,out=-30,in=-150] (0) to node[auto,swap] {$x$} (1);
	\draw[1c,out=-90,in=165] (o) to node[auto,swap,pos=.6] {$e\backslash x\!\!\!\!$} (1);
	\draw[1c,out=-15,in=90] (o) to node[auto] {$z$} (1);
	\draw[1c,out=90,in=-165] (0) to node[auto] {$e$} (o);
	\draw[2c] (-.3,-.1) to node[auto] {$t\;$} (-.3,1.2);
	\draw[2c] (.5,.3) to (.5,1.2);
	\draw[3c1] (1.25,.625) to (2.75,.625);
	\draw[3c2] (1.25,.625) to (2.75,.625);
	\draw[3c3] (1.25,.625) to (2.75,.625);
	\draw[follow] (1.25,.8) to (.6,.5);
\end{scope}
\begin{scope}[shift={(4,0)}]
	\node[0c] (0) at (-1, 0) {};
	\node[0c] (1) at (1,0) {};
	\node[0c] (o) at (0,1.5) {};
	\draw[1c,out=-30,in=-150] (0) to node[auto,swap] {$x$} (1);
	\draw[1c,out=-15,in=90] (o) to node[auto] {$z$} (1);
	\draw[1c,out=90,in=-165] (0) to node[auto] {$e$} (o);
	\draw[2c] (0,-.1) to node[auto] {$s\;$} (0,1.2);
	\node[scale=1.25] at (1.375,-.125) {,};
\end{scope}
\end{tikzpicture}
\end{equation}
which is an equivalence and universal at the indicated submolecule. If $X$ is ``1-truncated'', for example if it is the nerve of a 1-category $C$ --- so all cells of dimension 2 or higher correspond to identities in $C$ --- then the existence of the filler (\ref{eq:uniqueness1}) implies that $x$ factors in $C$ as the composite $e \cp{0} e\backslash x$, while the existence of the filler (\ref{eq:uniqueness3}) implies that, if $x$ factors as $e \cp{0} z$ for some other $z$, then $e\backslash x = z$. Since $x$ is arbitrary, this is a ``unique factorisation'' property exhibiting $e$ as an isomorphism in $C$.

In a similar fashion, if $X$ is ``2-truncated'', for example if it is the nerve of a 2-category, the universality of (\ref{eq:uniqueness3}) where indicated by the grey arrow implies uniqueness of the factorisation of $s$ through $t$. The existence of the horn filler (\ref{eq:uniqueness3}) is equivalent to $t$ exhibiting $e \backslash x$ as a left Kan extension of $x$ along $e$; by considering more general horns, we can see in fact that $e \backslash x$ is an \emph{absolute} Kan extension.

In general, universality of 2-cells at different submolecules of their boundary captures the universal properties of natural isomorphisms, absolute Kan extensions and absolute Kan lifts in 2-categories.
\end{exm}

\begin{dfn} 
A map $f: X \to Y$ of constructible polygraphs is \emph{strong} if it sends $n$\nbd equivalences of $X$ to $n$\nbd equivalences of $Y$.
\end{dfn}

\begin{dfn} \label{dfn:represent}
A constructible polygraph $X$ is \emph{representable} if, for all constructible $n$\nbd diagrams $x$ of $X$, there exist an $n$\nbd cell $\underline{x}$ and an $(n+1)$\nbd equivalence $c(x)$ with $\bord{}{\alpha}c(x) = x$ and $\bord{}{-\alpha}c(x) = \underline{x}$.

Representable constructible polygraphs and strong maps form a category $\rcpol$.
\end{dfn}

\begin{remark}
If $x$ is of shape $U$, then $c(x)$ is of shape $U \celto \tilde{U}$ or $\tilde{U} \celto U$, as in Construction \ref{cons:ou}. As in the 2-dimensional case, it should suffice to require that $\underline{x}$ and $c(x)$ exist when $x$ is a cell, or when it is a ``binary'' $n$\nbd diagram containing exactly two $n$\nbd cells.
\end{remark}

There is a functor $-_\cat{O}: \cpol \to \globset$, obtained by restricting presheaves on $\globe$ to its full subcategory $\cat{O}$. We expect the following to be true, for an adequate algebraic definition of weak higher category with an underlying $\omega$\nbd graph (possibly, a variant of Batanin's \cite{batanin1998monoidal} or Leinster's \cite{leinster2004higher}).
\begin{conj}
Let $f: X \to Y$ be a strong map of representable constructible polygraphs. Then $X_\cat{O}$ and $Y_\cat{O}$ admit the structure of an algebraic weak higher category, and $f_\cat{O}: X_\cat{O} \to Y_\cat{O}$ of a functor of weak higher categories.
\end{conj}

\end{document}